\documentclass[oneside,english,a4paper]{amsart}

\usepackage[T1]{fontenc}
\usepackage[utf8]{inputenc}
\usepackage[hidelinks]{hyperref}
\usepackage{amstext,amsthm,amssymb,babel,stackrel,graphicx,cancel,mathtools,microtype}
\numberwithin{equation}{section}
\numberwithin{figure}{section}

\theoremstyle{plain}

\newtheorem{thm}{Theorem}[section]
\newtheorem{prop}[thm]{Proposition}
\newtheorem{lem}[thm]{Lemma}
\newtheorem{cor}[thm]{Corollary}

\theoremstyle{remark}

\newtheorem{rem}[thm]{Remark}
\newtheorem{defn}[thm]{Definition}

\def\al{\alpha}
\def\be{\beta}
\def\ga{\gamma}
\def\la{\lambda}

\def\ep{\epsilon}

\def\s{\sigma}

\def\jp#1{\langle#1\rangle}
\def\norm#1{\|#1\|}
\def\ol#1{\overline{#1}}
\def\ul#1{\underline{#1}}
\def\oul#1{\ul{\ol{#1}}}

\def\tilde#1{\widetilde{#1}}

\def\Bb#1{\Big(#1\Big)}
\def\bb#1{\big(#1\big)}
\def\mc#1{\mathcal{#1}}
 
\def\les{\lesssim}
\def\hook{\hookrightarrow}
\def\weak{\rightharpoonup}

\def\d{\partial}
\def\na{\nabla}
\def\De{\Delta}

\def\N{\mathbb{N}}
\def\Z{\mathbb{Z}}

\def\R{\mathbb{R}}
\def\C{\mathbb{C}}

\def\S{\mathbb{S}}
\def\P{\mathbb{P}}

\def\E{\mathcal{E}}
\def\F{\mathcal{F}}

\def\I{\mathcal{I}}

\def\g{\mathbf{g}}
\def\h{\mathbf{h}}
\def\G{\mathbf{G}}
\def\H{\mathbf{H}}

\def\Im{\mathrm{Im}}

\def\supp{\mathrm{supp}}
\def\one{\mathbf{1}}
\def\loc{\mathrm{loc}}

\def\aff{\mathrm{aff}}
\def\pc{\mathrm{pc}}

\def\sc{\mathrm{sc}}

\DeclareMathOperator{\sech}{sech}

\begin{document}
\title[Universal bound on blow-up rate for mass-critical NLS]{A universal bound on the blow-up rate for the focusing mass-critical nonlinear Schr\"odinger equation}

\author{Beomjong Kwak}
\email{beomjong@kaist.ac.kr}
\address{Department of Mathematical Sciences, Korea Advanced Institute of Science
and Technology, 291 Daehak-ro, Yuseong-gu, Daejeon 34141, Korea}

\author{Soonsik Kwon}
\email{soonsikk@kaist.edu}
\address{Department of Mathematical Sciences, Korea Advanced Institute of Science
and Technology, 291 Daehak-ro, Yuseong-gu, Daejeon 34141, Korea}

\keywords{nonlinear Schrödinger equation, mass-critical, blow-up rate, self-similar, soliton}
\subjclass[2020]{35B44, 35Q55, 37K40}

\begin{abstract}
In this paper, we investigate a universal blow-up bound for the focusing mass-critical nonlinear Schr\"odinger equation for general initial data in $L^2(\R^d)$, extending previous knowledge for mass near the ground-state threshold due to Merle and Rapha\"el. The main results are twofold.
First, we show the nonexistence of self-similar rate blow-up solutions.
Second, under radial symmetry, we establish the sharp log--log correction to the self-similar bound on the blow-up rate.
The proofs are based on a new analysis of general blow-up solutions, which does not rely on any ansatz or variational structure.
\end{abstract}
\maketitle
\section{Introduction}
We consider the focusing nonlinear Schr\"odinger equation
(NLS) on $\R^{d}$:
\begin{equation}\label{eq:NLS}\tag{NLS}
\begin{cases}
iu_{t}+\De u+\left|u\right|^{\frac{4}{d}}u=0,\\
u(0)=u_0\in L^{2}(\R^{d}).
\end{cases}
\end{equation}
\eqref{eq:NLS} is mass-critical in the sense that the mass $M(u)=\int_{\R^d} |u|^2 dx$
is invariant under the scaling symmetry $u(t,x) \mapsto \lambda^{\frac{d}{2}}u(\la^2 t,\la x) $. It is well known that the Cauchy problem is locally well-posed in the critical space $C^0L^2\cap L^{\frac{2d+4}{d}}_{t,x}([0,T_0]\times \R^d)$. As a consequence of the local theory, if a solution $u$ ceases to exist past time $T<\infty $, then $\|u\|_{L^{\frac{2d+4}{d}}_{t,x}([0,T)\times \R^d)}=\infty$. In this case, we say that the solution $u(t)$ \textit{blows up} at time $T$.

\eqref{eq:NLS} has the ground-state soliton solution $u(t,x)=e^{it}Q(x)$. Here, the ground state $Q$ is the unique
positive, radial, and exponentially decaying solution of $-Q+\De Q+|Q|^{\frac4d}Q=0$.

A fundamental fact is that $e^{it}Q$ serves as a \textit{threshold} for long-time dynamics. In the sub-threshold regime $M(u)<M(Q)$, all solutions are global and scatter \cite{KillipTaoVisan09radial-2D,KillipVisanZhang08radial-3D+,Dodson15focusing}. At the threshold $M(u)=M(Q)$, the dynamics are rigid: in most settings, the only non-scattering solutions are rescaled solitons and their pseudo-conformal transforms \cite{Merle93minimalmass,Killip-Li-Visan-Zhang,Dodson21,Dodson21-1}.
Above the threshold, $M(u)>M(Q)$, various dynamics are possible. All negative-energy initial data of sufficient decay lead to finite-time blow-up \cite{Glassey} and various blow-up rates have been constructed \cite{BourgainWang,Perelman,MerleRaphael05CMP,MartelRaphael}.

Accordingly, understanding the general blow-up dynamics beyond the threshold mass $M(Q)$ has been of interest.
An elementary fact is that the blow-up rate is at most \emph{self-similar}; by a scaling argument, it is straightforward that, for a solution $u$ in $H^1$ that blows up at time $T$,
\begin{equation}\label{eq:scaling bound}
\|\na u(t)\|_{L^2(\R^d)}\gtrsim\frac{1}{\sqrt{T-t}}.
\end{equation}
Note that \eqref{eq:scaling bound} makes sense only for $H^1$-data. An analogous statement for general $L^2$-solutions is the concentration of mass, first shown by Bourgain \cite{Bourgain-refinements} on $\R^2$:
\begin{equation}\label{eq:Bourgain}
\limsup_{t\to T}\sup_{x_0\in\R^2}\int_{x_0+[-\sqrt{T-t},\sqrt{T-t}]^2}|u(t,x)|^2 dx\ge c
\end{equation}
for some constant $c=c(M)>0$.
See \cite{BegoutVargas07concentration,Keraani06blowup,KillipTaoVisan09radial-2D,KillipVisanZhang08radial-3D+} for further developments.

However, both numerical simulations and formal derivations have long suggested a gap between the actual dynamics and the standard scaling bound \eqref{eq:scaling bound} \cite{Fraiman1985asymptotic, LPSS}. A classical fact is that a stationary self-similar ansatz has infinite mass, which heuristically rules out a self-similar blow-up approaching that particular ansatz by the mass conservation. Nevertheless, for general data,
whether the self-similar bounds \eqref{eq:scaling bound} and \eqref{eq:Bourgain} could be improved remained an open question, as emphasized in survey articles \cite{Bourgain00problems, Raphael2008CMI, MartelMerleRaphaelSzeftel2014_arXiv}.

A major advance toward understanding blow-up dynamics was achieved by Merle and Raphaël \cite{MerleRaphael05Annals,MerleRaphael03GAFA,MerleRaphael04Invent,Raphael05Annalen,MerleRaphael06AMS,MerleRaphael05CMP}. Through a modulation analysis, they
developed a rigid description of the dynamics in the \emph{near-soliton regime} $M(Q)<M(u)\le M(Q)+\alpha_*$. They showed that, for $H^1$-solutions in the near-soliton regime, the blow-up profile is $Q$ and the blow-up rate is 
either of the following:
\[
    \norm{\nabla u(t)}_{L^2} \sim 
    \sqrt{\frac{\log|\log(T-t)|}{T-t}}
    \quad\text{or}\quad
    \norm{\nabla u(t)}_{L^2} \gtrsim \frac{1}{T-t}.
\]
In particular, the former rate, called the \emph{log--log rate}, provides a sharp lower bound improving \eqref{eq:scaling bound} and occurs for all negative-energy initial data. This log--log rate was first formally derived by Fraiman \cite{Fraiman1985asymptotic} and further supported by numerical evidence due to Landman--Papanicolaou--Sulem--Sulem \cite{LPSS}. Its first rigorous construction was then given by Perelman \cite{Perelman}.

Apart from the near-soliton regime, the nature of general blow-up dynamics remains largely unknown. Historically, the only available information regarding the blow-up rate was the elementary self-similar scaling bound \eqref{eq:scaling bound}. A fundamental open question in the field was whether genuine self-similar blow-up solutions could exist. Even within the near-soliton regime, the potential for such solutions posed a significant challenge to establishing the universality of the blow-up profile $Q$ \cite{MerleRaphael04Invent}.

In this paper, we provide universal information on blow-up dynamics without any size constraints on the initial data. By improving the bound in \eqref{eq:Bourgain} for general $L^2$-solutions, we resolve the aforementioned question and establish a new universal bound on the blow-up rate.

Our main contribution is twofold. First, we rule out the existence of self-similar blow-up for arbitrary $L^2$-solutions.
\begin{thm}[Nonexistence of self-similar blow-up]\label{thm:no self similar}
Let $d\ge 1$. Let $u(t)$ be a solution to \eqref{eq:NLS} in $L^2(\R^d)$ that blows up at time $T<\infty$. For every $\ep>0$, we have
\begin{equation}\label{eq:no self similar}
\limsup_{t\to T}\sup_{x_0\in\R^d}\int_{x_0+[-\ep\sqrt{T-t},\ep\sqrt{T-t}]^d}|u(t,x)|^2 dx\ge M(Q).
\end{equation}
\end{thm}
Since $\ep>0$ is arbitrary, $M(Q)$ mass concentrates on scales $o(\sqrt{T-t})$. In particular, self-similar blow-up cannot occur.

Second, for radial solutions, we obtain the optimal quantitative improvement of Theorem~\ref{thm:no self similar}. As a main result of the paper, we establish the sharp log--log bound on the blow-up rate.

\begin{thm}[Log--log blow-up bound]\label{thm:loglog}
Let $d\ge2$. Let $u(t)$ be a radial solution to \eqref{eq:NLS} in $L^2(\R^d)$ that blows up at time $T>0$. Then the following hold:
\begin{enumerate}
    \item\label{enu:parabolic convergence}
    There exists an asymptotic profile $z^*\in L^2(\R^d)$ such that for every $r>0$,
\begin{equation}\label{eq:z* def}
\lim_{t\to T}\norm{u(t)-z^*}_{L^2(|x|\ge r\sqrt{T-t})}=0.
\end{equation}
    \item\label{enu:loglog}
    Let $0<\delta<1$. We define the effective radius of the blow-up core: for $0\le t<T$,
    \begin{equation}\label{eq:la def}
\la_\delta(t)\coloneqq \inf\{ r>0:\norm{u(t)-z^{*}}_{L^2(|x|\ge r)}\le\delta\}.
\end{equation}
Then there exist $t_0<T$ and $C>0$ such that for every $t_*\in[t_0,T)$,
\begin{equation}\label{eq:loglog}
\int_{0}^{t_*}\frac{\la_\delta^{2}(t)}{T-t}\frac{dt}{T-t}\le C^2\int_{0}^{t_*}\frac1{\log|\log (T-t)|}\frac{dt}{T-t}.
\end{equation}
Moreover, we can set $C=\sqrt{2\pi+o_{\delta\to 0}(1)}\cdot|\log\delta|$.
\end{enumerate}
\end{thm}
\begin{rem}
Theorem~\ref{thm:loglog} measures the blow-up rate in a stronger sense than \eqref{eq:Bourgain}.
Interpreting $\la_\delta(t)$ as the largest bubble radius of $u(t)$
 within a bubble-tower structure in the blow-up core, Theorem~\ref{thm:loglog} captures the dynamics at the largest bubbling scale.
In contrast, \eqref{eq:Bourgain} detects only the smallest scale in the bubble tower. 
\end{rem}
Theorem~\ref{thm:loglog} is optimal in two senses. First, it captures the sharp log--log correction to the concentration scale.
Second, the asymptotic size $C=\sqrt{2\pi+o_\delta(1)}\cdot|\log\delta|$ is sharp.
The factor $|\log\delta|$ reflects the exponential tail of the blow-up profile, which matches that of $Q$. The asymptotic coefficient $\sqrt{2\pi}$ coincides with the sharp scaling constant attained by log--log blow-up solutions in the near-soliton regime \cite{MerleRaphael05CMP}.\footnote{Note that this is shown under a spectral assumption, which has been verified for $d\le 12$~\cite{FibichMerleRaphael06spectral,Yang18Blowup}.}

Theorem~\ref{thm:loglog} gives the sharp improvement of \eqref{eq:Bourgain}.
\begin{cor}[Sharp mass concentration]\label{cor:loglog la Bourgain}
Let $d\ge 2$. Assume that a radial solution $u$ to \eqref{eq:NLS} blows up at time $T<\infty$. Then there exists $R>0$ such that for every $r\ge R$,
\begin{equation}\label{eq:loglog la Bourgain}
\limsup_{t\to T}\int_{|x|\le r\sqrt{\frac{T-t}{\log|\log(T-t)|}}}|u(t,x)|^2 dx\ge\int_{|x|\le\frac r{\sqrt{2\pi}} }|Q(x)|^2dx.
\end{equation}
\end{cor}
For the proofs of Theorems \ref{thm:no self similar} and \ref{thm:loglog}, we develop an analysis that does not rely on any ansatz or variational structure. 
This analysis captures the log--log gap from self-similarity at the linear level, handling the nonlinearity perturbatively at the quantitative level.

This approach also relates the emergence of the log--log blow-up rate to the spatial decay of the blow-up core and the regularity of the asymptotic profile $z^*$, a connection previously understood only in the near-soliton regime \cite{MerleRaphael05CMP}. As a result, we show that if the log--log blow-up rate bound is saturated, then the time average of the renormalized solution exhibits precise exponential spatial decay and $z^*\notin H^s$ for any $s>0$.

\begin{thm}[Saturated regime]\label{thm:saturate loglog}
Let $d\ge 2$. Let $u$ be a radial solution to \eqref{eq:NLS} with initial data $u_0\in L^2(\R^d)$ and the asymptotic profile $z^*$ as in \eqref{eq:z* def}. Assume that there exist $\delta>0$, $\ep>0$, and $t_0<T$ such that for $t_*\in[t_0,T)$,
\begin{equation}\label{eq:saturate loglog}
\int_0^{t_*}\frac{\la_\delta^2(t)}{T-t}\frac{dt}{T-t}\ge \ep\int_0^{t_*}\frac{1}{\log|\log(T-t)|}\frac{dt}{T-t}.
\end{equation}
Then the following hold:
\begin{enumerate}
    \item\label{enu:Thm profile} 
Set
\[
U(t,r)\coloneqq r^{\frac d2}u(t,r).
\]
    There exist $0<c\le C$ such that, for $r\gg1$ and $0<T-t_*\ll_r1$,
    \begin{equation}\label{eq:precise exponential}
    e^{-Cr}\le\frac1{|\log(T-t_*)|}\int_0^{t_*}\Big|U\Big(t,r\sqrt{\tfrac{T-t}{\log|\log(T-t_*)|}}\Big)\Big|^2\frac{dt}{T-t}\le e^{-cr}.
    \end{equation}
    \item\label{enu:Thm z*} For any $s>0$, either $u_0\notin H^s(\R^d)$ or $z^*\notin H^s(\R^d)$.
\end{enumerate}
\end{thm}
\subsection*{Acknowledgements} 
The authors were partially supported by the National Research Foundation of Korea, RS-2019-NR040050 and RS-2022-NR069873.
\section{Preliminaries}\label{sec:preliminaries}
\subsection{Notations}
We summarize frequently used notations here. We denote $A\les B$ if $A\le C_dB$ holds for an implicit constant $C_d>0$ depending only on the dimension $d$. We denote the Strichartz exponent $p=p_d=\frac{2d+4}d$ and call $(q,r)\in[2,\infty]^2$ a Strichartz pair if $\frac d2=\frac 2q+\frac dr$. We use the Japanese bracket notation $\jp x\coloneqq \sqrt{|x|^{2}+1}$ for $x\in\R^{d}$. We denote the radial derivative by $\d_r=\frac{x}{|x|}\cdot\na$.

\subsubsection*{Fourier transform} We denote by $\mc S(\R^{d})$ the Schwartz space on $\R^{d}$. For a function of space or spacetime variables, $f=f(t,x)$ or $f=f(x)$, we denote by $\widehat f$ or $\F f$ the Fourier transform of $f$
with respect to the spatial variable $x$. We use both frequency and spatial smooth cutoffs. Let $\psi^{\le1}:\R^{d}\to[0,1]$ be a radial smooth function such that $\psi^{\le1}(x)=1$ for $|x|\le1$ and $\psi^{\le1}(x)=0$ for $|x|\ge2$. For $R>0$, we denote $\psi^{\le R}= \psi^{\le1}(R^{-1}\cdot)$, $\psi^{R}= \psi^{\le R}-\psi^{\le R/2}$, and $\psi^{\sim R}= \psi^{\le 4R}-\psi^{\le R/8}$. 
We define the Littlewood-Paley projections as follows. For $N>0$, we denote by $P_{\le N}$ the Littlewood-Paley projection with multiplier $\psi^{\le N}$. We also use the shorthand notation $f_{\le N}\coloneqq P_{\le N}f$, and similarly for other frequency localizations.

\subsubsection*{Function spaces}
On $\R\times\R^d$, we use mixed Lebesgue norms $L^q_tL^r_x(\R\times \R^d)$. On $\R^d$, we work with (in)homogeneous Sobolev and Besov spaces, denoted by $\dot W^{\theta,r}({W}^{\theta,r})$ and $\dot B^\theta_{r,\eta}({B}^\theta_{r,\eta})$, respectively. For instance, the homogeneous Besov norm admits $\norm{f}_{\dot B^\theta_{r,\eta}}\sim\norm{N^\theta\norm{P_Nf}_{L^r}}_{\ell^\eta_N}$. In addition, for $-\frac d2\le \theta\le\frac d2$, we denote
\[
2_\theta=\frac{2d}{d-2\theta}.
\]
Note that $(2,2_1)$ is a Strichartz pair and $(2,2_{-1})$ is its dual.
\subsubsection*{Weighted Besov estimates}
We use the following Besov-type estimates containing various spatial weights and frequency cutoffs.
\begin{lem}
Let $s,\theta\in\R$, $R>0$, and $r,\eta\in[1,\infty]$. For $u\in L^2(\R^d)$, we have
\begin{align}\label{eq:weighted LP}
\norm{\jp x^{-1}u_{\ge1}}_{H^{\theta}(\R^d)}&\les_\theta\norm{N^{\theta}\norm{\jp x^{-1}u_N}_{L^2}}_{\ell^2_N},&\theta>0,
\\
\label{eq:PN l2 decouple}
\norm{N^\theta\norm{P_N(\psi^{\ge R/N}u)}_{L^2(\R^d)}}_{\ell^2_N}&\les_{R}\norm u_{\dot H^\theta(\R^d)},&|\theta|\ll1,
\\
\label{eq:homogeneous embedding Lp, >N}
\norm{N^{s-\theta}\norm{|Nx|^\theta u_{\ge N}}_{L^r(\R^d)}}_{\ell^\eta_N}&\les_{s,\theta}\norm{N^{s-\theta}\norm{|Nx|^\theta u_N}_{L^r(\R^d)}}_{\ell^\eta_N},&s>0,
\\
\label{eq:homogeneous embedding Lp, d>N}
\norm{N^{s-\theta}\norm{|Nx|^\theta u_{< N}}_{L^r(\R^d)}}_{\ell^\eta_N}&\les_{s,\theta}\norm{N^{s-\theta}\norm{|Nx|^\theta u_N}_{L^r(\R^d)}}_{\ell^\eta_N},&s<0.
\end{align}
\end{lem}
\begin{proof}
Since $\norm{(1-P_{\sim N})(\jp x^{-1}u_N)}_{H^\theta}\les\norm{\jp x^{-1}u_N}_{L^2}$ for $N\ge1$, \eqref{eq:weighted LP} holds.

For \eqref{eq:PN l2 decouple}, we observe that, by scaling $\norm{P_1(\psi^{\ge R}u)}_{L^2}\les_{R}\norm u_{\dot H^\theta}$ for $\theta=\pm\frac1{10}$,
\[
\norm{N^\theta \norm{P_N(\psi^{\ge R/N}u)}_{L^2}}_{\ell^\infty_N}\les_{R}\norm u_{\dot H^\theta}.
\]
A real interpolation then yields \eqref{eq:PN l2 decouple}.

We turn to \eqref{eq:homogeneous embedding Lp, >N}. By the triangle inequality, we have
\begin{align*}
\norm{N^{s-\theta}\norm{|Nx|^\theta u_{\ge N}}_{L^r}}_{\ell^\eta_N}&\le\norm{\norm{N^{s-\theta}\norm{|Nx|^\theta u_{KN}}_{L^r}}_{\ell^\eta_N}}_{\ell^1_{K\ge1}}
\\&\le\norm{K^{-s}\norm{M^{s-\theta}\norm{|Mx|^\theta u_M}_{L^r}}_{\ell^\eta_M}}_{\ell^1_{K\ge1}},
\end{align*}
where we substituted $KN=M$. Since $\sum_{K\ge1}K^{-s}=O_s(1)$, \eqref{eq:homogeneous embedding Lp, >N} follows. In the case $s<0$, the same estimate over $K\le1$ concludes \eqref{eq:homogeneous embedding Lp, d>N}.
\end{proof}
The next lemma shows a bilinear estimate on radial functions.
\begin{lem}\label{lem:bilinear}
Let $d\ge2$ and $a>0$.
We define a weight function and a bilinear form for radial functions $f,g\in H^1(\R^d)$
\[
\varphi^a(x)\coloneqq \max\{|x|-a,\,0\},\quad [f,g]_a
\coloneqq
\int_{\R^d}  \overline{f}\,\nabla g\cdot \nabla \varphi^a \, dx.
\]
Then, we have the estimate
\begin{equation}\label{eq:bilinear}
\left|[f,g]_a\right|
\lesssim
(a^{-\frac12}+1)
\|f\|_{B^{\frac12}_{2,1}(\R^d)}
\|g\|_{B^{\frac12}_{2,1}(\R^d)}.
\end{equation}
\end{lem}

\begin{proof}

Note that $\Delta\varphi^{a}$ consists of a function on $\{|x|>a\}$ of size $O(a^{-1})$
and a unit surface measure supported on the sphere $\{|x|=a\}$. Integrating by parts, we obtain
\begin{equation}\label{eq:int by parts}
\left|[f,g]_a+\ol{[g,f]_a}\right|=\left|\int_{\R^d}\na(\ol fg)\cdot\na\varphi^a dx\right|\les a^{-1}\norm{\ol fg}_{L^1}+a^{d-1}|f(a)|\cdot|g(a)|.
\end{equation}
Since $\norm{\nabla\varphi^{a}}_{L^\infty(\R^d)}\les1$, it follows that $[\cdot,\cdot]_a : L^{\infty}(\R^{d})\times W^{1,1}(\R^{d}) \to \C$ is $O(1)$-bounded. Moreover, by the radial Sobolev embedding $|f(a)|\lesssim a^{1-d}\|f\|_{W^{1,1}(\R^{d})}$ and \eqref{eq:int by parts}, $[\cdot,\cdot]_a : W^{1,1}(\R^{d})\times L^{\infty}(\R^{d}) \to \C$ is $O(a^{-1}+1)$-bounded. 
Interpolating between these two estimates yields \eqref{eq:bilinear}, finishing the proof.
\end{proof}
\subsection{The Schr\"odinger operator and the virial identity}
We denote the linear Schr\"odinger propagator by $e^{it\De}\phi$ for $\phi\in L^{2}(\R^{d})$, i.e.,
\[
e^{it\De}\phi=\mathcal{F}^{-1}(e^{-it\left|\xi\right|^{2}}\F\phi).
\]
For a function $f:\R\times\R^{d}\to\C$, we denote the retarded Duhamel operators by
\[
\mc D^+ f(t)\coloneqq-i\int_{-\infty}^{t}e^{i(t-s)\De}f(s)ds,\quad \mc D^- f(t)\coloneqq i\int_{t}^{\infty}e^{i(t-s)\De}f(s)ds.
\]
We also denote by $K_N(t,x,y)$ the integral kernel of the operator $P_N e^{it\Delta}$. It is standard that (see, e.g., \cite[Lemma~2.4]{KillipTaoVisan09radial-2D}), for $m\ge0$,
\begin{equation}\label{eq:kernel}
|K_N(t,x,y)|\les_m\begin{cases}
t^{-\frac d2},&|x-y|\sim|Nt|,
\\
N^d\langle N^2t\rangle^{-m}\langle N(x-y)\rangle^{-m},&\text{otherwise.}
\end{cases}
\end{equation}
\subsubsection*{Symmetry group} \eqref{eq:NLS} enjoys a rich family of symmetries, which play a fundamental role in the analysis of its dynamics. To streamline the presentation, we collect these symmetries and introduce a convenient group-action notation.

Let $t_0\in\R$, $x_0,\xi_0\in\R^d$, and $\la>0$ be parameters for symmetries.
\begin{align*}
u(t,x)&\mapsto u(t,x-x_0)&\quad\text{(space translation)}\\
u(t,x)&\mapsto u(t-t_0,x)&\quad\text{(time translation)}\\
u(t,x)&\mapsto \la^{\frac d2}u(\la^2 t,\la x)&\quad\text{($L^2$-critical scaling)}\\
u(t,x)&\mapsto e^{ix\cdot\xi_0-it|\xi_0|^2}u(t,x-2\xi_0t)&\quad\text{(Galilean boost)}\\
u(t,x)&\mapsto |t|^{-\frac d2}e^{\frac i{4t}|x|^2}u(-t^{-1},-t^{-1}x)&\quad\text{(pseudo-conformal symmetry)}
\end{align*}
Note that all above symmetries preserve the time-orientation, as well as the mass.

We denote by $\G$ the group generated by the symmetry transformations above. The group $\G$ acts both on the time axis $\R$ and on the function space $L^{\infty}L^{2}(\R\times\R^{d})$. For $\g\in\G$, we denote by $\g t$ the action of $\g$ on the time variable $t\in\R$, which is a linear fractional transformation. We also write $\g u$ for the action of $\g$ on $u$.
For instance, if $\g\in\G$ corresponds to the scaling and translation symmetry
\[
u(t,x)\mapsto \lambda^{\frac d2}u\bigl(\lambda^{2}(t-t_{0}),\,\lambda(x-x_{0})\bigr),
\]
then the associated actions can be represented as
\[
\g t=\lambda^{2}(t-t_{0}),
\quad
(\g u)(t,x)
=\lambda^{\frac d2}
u\bigl(\lambda^{2}(t-t_{0}),\,\lambda(x-x_{0})\bigr).
\]
The following transform $\g_1\in\G$ is of particular importance in our analysis:
\begin{equation}\label{eq:g1 def}
\g_1 t=\frac{t+1}{3-t},\quad
\g_1 u(\g_1^{-1}t,x)
=
\Bb{\frac{1+t}2}^{\frac d2}
e^{-i\frac{(1+t)}{16}|x|^2}u\Bb{t,\frac{1+t}2x}.
\end{equation}
\begin{rem}\label{rem:g1 def}
For $R>0$ and a function $u$ supported in a parabolic region
\[
\Gamma=\{(t,x):|x|\le R\sqrt t\}\subset[0,\infty)\times\R^d,
\]
the transformed function $\g_1 u$ is supported in the elliptic region
\[
\Gamma_1=\{(\g_1^{-1} t,x):\tfrac{1+t}2|x|\le R\sqrt{t}\}=\{(\tau,x):|x|^2\le R^2(1-\tfrac14(\tau-1)^2)\}.
\]
In particular, $\Gamma_1$ attains the largest diameter at time $1=\g_11$.
\end{rem}
For standard profile decomposition arguments, we denote by $\G_{\mathrm{aff}}$ the subgroup of $\G$ generated by spacetime translations, scalings, and Galilean transformations.

\subsection{Local smoothing estimates}
We use local smoothing effects of the Schr\"odinger
operator \cite{constantin-saut,sjolin,vega}. It is standard (see, e.g., \cite[Theorem 2.1]{kenig-ponce-vega}) that for $R>0$,
\begin{equation}\label{eq:local smoothing}
R^{-\frac12}\norm{\psi^{\le R}e^{it\De}\phi}_{L^2\dot H^{\frac12}(\R\times\R^d)}\les \norm{\phi}_{L^2(\R^d)}.
\end{equation}
The following is a local smoothing effect adapted to the self-similar scaling.
\begin{lem}
For $\phi\in L^2(\R^d)$ and $N\in2^\Z$, we have
\begin{equation}
\label{eq:weighted local smoothing}
N^{\frac12}\norm{t^{-\frac14}\jp{x/\sqrt t}^{-1}
e^{it\De}P_N\phi}_{L^2L^2((0,\infty)\times\R^d)}\les\norm\phi_{L^2(\R^d)}.
\end{equation}
\end{lem}
\begin{proof}
By scaling, we may assume $N=1$. We claim that, for $R\ge1$,
\begin{equation}
\label{eq:weighted local smoothing claim}
R^{-\frac12}\norm{t^{-\frac14}\psi^{\le R\sqrt t}
e^{it\De}P_1\phi}_{L^2L^2((0,\infty)\times\R^d)}\les\norm\phi_{L^2(\R^d)}.
\end{equation}
Once \eqref{eq:weighted local smoothing claim} is shown, summing over $R\in2^\N$ yields \eqref{eq:weighted local smoothing}.

To show \eqref{eq:weighted local smoothing claim}, we split the time interval into the regions $t\lesssim R^{2}$ and $t\gg R^{2}$. 
For $t\lesssim R^{2}$, the bound follows from the $L^{2}$-conservation. Thus, it suffices to consider the contribution from the region $t\gg R^{2}$, where genuine local smoothing is required. 

For each dyadic number $T\gg R^2$, applying \eqref{eq:local smoothing} on $[T,2T]$ yields
\begin{equation}\label{eq:l2 local smoothing, effective, prep}
(R\sqrt T)^{-\frac12}\norm{\psi^{\le R\sqrt T}e^{it\Delta}P_1(\psi^{\sim T}\phi)}_{L^2L^2([T,2T]\times\R^d)}
\lesssim
\|\psi^{\sim T}\phi\|_{L^2}.
\end{equation}
For the remainder, since $T\gg R\sqrt T$, the kernel bound \eqref{eq:kernel}
implies
\begin{equation}\label{eq:l2 local smoothing, away, prep}
\norm{\psi^{\le R\sqrt T}e^{it\Delta}P_1(\phi-\psi^{\sim T}\phi)}_{L^2L^2([T,2T]\times\R^d)}
\lesssim T^{-10}\|\phi\|_{L^2}.
\end{equation}
By \eqref{eq:l2 local smoothing, effective, prep} and \eqref{eq:l2 local smoothing, away, prep}, we obtain
\[
R^{-\frac12}T^{-\frac14}\norm{\psi^{\le R\sqrt T}e^{it\De}P_1\phi}_{L^2L^2([T,2T]\times\R^d)}\les \norm{\psi^{\sim T}\phi}_{L^2}+T^{-1}\norm\phi_{L^2}.
\]
Taking an $\ell^2$-summation over dyadic numbers $T\gg R^2$ concludes
\eqref{eq:weighted local smoothing claim}.
\end{proof}

\subsection{In/out decomposition and radial Strichartz estimates}

When working in the radial setting, it is advantageous to use the in/out decomposition for the Schr\"odinger operator. This decomposition, originally introduced in \cite{Tao04} and further refined in \cite{KillipTaoVisan09radial-2D,KillipVisanZhang08radial-3D+}, separates a radial wave into components propagating toward and away from the origin. 
A key benefit is that the associated integral kernels enjoy an additional decay. We briefly recall the in/out decomposition and its basic properties below; see \cite{KillipTaoVisan09radial-2D,KillipVisanZhang08radial-3D+} for precise definitions and \cite[Lemmas~4.1 and~4.2]{KillipVisanZhang08radial-3D+} for the proof.
\begin{prop}[In/out decomposition, \cite{KillipTaoVisan09radial-2D,KillipVisanZhang08radial-3D+}]
Let $d\ge2$. There exist linear operators $P^\pm$ that map a radial function $f\in L^2(\R^d)$ to radial functions $P^\pm f:\R^d\to\C$ satisfying the following:
\begin{itemize}
\item $P^{+}f+P^{-}f=f$ holds.
\item $P^\pm$ have scaling invariance: for $\la>0$, $P^\pm f(\la x)=(P^\pm f(\la\cdot))(x)$.
\item Denote by $K^\pm_1(t,x,y)$ the integral kernels for $P^\pm P_1e^{\mp it\De}$ and let $m\ge 1$. We have a pointwise kernel bound: for $t\ge0$ and $x,y\in\R^{d}$ such that $|x|\ge1$,
\begin{equation}\label{eq:inout bound}
{
|K^\pm_1(t,x,y)|\les_m\begin{cases}
|x|^{-\frac{d-1}{2}}\left|y\right|^{-\frac{d-1}{2}}t^{-\frac{1}{2}}, &(|y|-|x|)\sim t\gtrsim 1\\
|x|^{-\frac{d-1}{2}}\jp{y}^{-\frac{d-1}{2}}\jp{t+|x|-|y|}^{-m}, &\mathrm{otherwise}
\end{cases}.}
\end{equation}
\end{itemize}
\end{prop}

We combine $P^\pm$ with the frequency localization $P_N$ and the spatial cutoff $\psi^{\ge 1/N}$ to avoid singular behavior at the origin $x=0$. For $f\in L^2(\R^d)$ and $N>0$, we denote
\[
\P_N^\pm f\coloneqq \tfrac12 P_{\sim N}(\psi^{<1/N}P_Nf)+P_{\sim N}(\psi^{\ge1/N}P^\pm P_Nf).
\]
By \eqref{eq:kernel} and \eqref{eq:inout bound}, the integral kernel $\mathbb K^\pm_1(t,x,y)$ for $\P_1^\pm e^{\mp it\De}$ satisfies
\begin{equation}\label{eq:inout bound K}
|\mathbb K^\pm_1(t,x,y)|\les_m\begin{cases}
\jp x^{-\frac{d-1}{2}}\left|y\right|^{-\frac{d-1}{2}}t^{-\frac{1}{2}}, &(|y|-|x|)\sim t\gtrsim 1\\
\jp x^{-\frac{d-1}{2}}\jp{y}^{-\frac{d-1}{2}}\jp{t+|x|-|y|}^{-m}, &\mathrm{otherwise}
\end{cases}.
\end{equation}
In particular, the second case of \eqref{eq:inout bound K} yields, for $r\in[1,\infty]$ and $\theta\in\R$,
\begin{equation}\label{eq:Ppm is Aut}
\norm{\jp{Nx}^{-\theta} \P_N^{\pm} f}_{L^r(\R^{d})}
\lesssim_\theta
\norm{\jp{Nx}^{-\theta} f}_{L^r(\R^{d})}.
\end{equation}

We use a radial improvement of Strichartz estimate by Shao \cite{Shao2009} and Guo--Wang \cite{GuoWang2014improved} for $d\ge2$: For $N\in2^\Z$ and radial $\phi\in L^2(\R^d)$, we have
\begin{equation}\label{eq:L2 Shao}
\norm{P_N e^{it\De}\phi}_{L^2L^r(\R\times\R^d)}\les_r N^{\frac d2-1-\frac dr}\norm{\phi}_{L^2(\R^d)},\quad r>r_d=\tfrac{4d-2}{2d-3}.
\end{equation}
By the radial Sobolev inequality, \eqref{eq:L2 Shao} yields
\begin{equation}\label{eq:weighted shao}
\norm{|x|^{\theta}P_Ne^{it\De}\phi}_{L^2L^{2_1}(\R\times\R^d)}\les_{\theta} N^{-\theta}\norm\phi_{L^2(\R^d)},\quad \theta<(d-1)\bb{\tfrac1{r_d}-\tfrac1{2_1}}.
\end{equation}
For convenience, we often use a fixed instance of such $\theta$:
\begin{equation}\label{eq:weighted shao def}
\ep_d\coloneqq\frac{2d-3}{4d}.
\end{equation}
Interpolating between \eqref{eq:weighted shao} and the $L^\infty L^2$-norm bound yields
\begin{equation}\label{eq:weighted shao Lq}
\norm{|x|^{\theta}P_Ne^{it\De}\phi}_{L^2L^{2_1}\cap L^pL^p(\R\times\R^d)}\les N^{-\theta}\norm\phi_{L^2(\R^d)},\quad 0\le\theta\ll1.
\end{equation}
\subsection{Carleman inequalities}
We record a Carleman inequality for the Schr\"odinger operator on
$\mathbb R^d$ due to Ionescu and Kenig \cite{IonescuKenig04Carleman}, which
extends earlier Carleman-type estimates developed in
\cite{KenigPonceVega03unique,KenigSogge88unique}. We state below a particular case of the inequality adapted to endpoint Strichartz norms in dimensions $d\ge 3$:
\begin{prop}[{\cite[Theorem 2.1]{IonescuKenig04Carleman}}]\label{prop:Carleman}
Let $d\ge 3$. Denote $x=(x_1,\ldots,x_d)\in\R^d$. Suppose that $\beta\ge 0$, $-\infty<t_0<t_1<\infty$. For $f\in L^2L^{2_{-1}}([t_0,t_1]\times\R^d)$, let $u\in C^0L^2([t_0,t_1]\times\R^d)$ be an inhomogeneous solution to $i\d_t u+\De u=f$.
We have 
\begin{equation}\label{eq:Carleman}
\norm{e^{\be x_1}u}_{C^0L^2\cap L^2L^{2_1}}\les_d \norm{e^{\be x_1}f}_{L^2L^{2_{-1}}}+\max_{t\in\{t_0,t_1\}}\norm{e^{\be x_1}u(t)}_{L^2(\R^d)}.
\end{equation}
\end{prop}
We also recall the following unique continuation result, shown in \cite{IonescuKenig04Carleman} as a consequence of the Carleman inequality developed in there:
\begin{prop}\label{prop:Carleman uniqueness}
Let $b\in\R$. Let $u$ be a solution to \eqref{eq:NLS} of lifespan $I\supset[0,1]$. If both $u(0)$ and $u(1)$ are supported in $\{x\in\R^d:x_1\le b\}$, then $u=0$.
\end{prop}
\begin{proof}
This follows directly from \cite[Theorem 2.5]{IonescuKenig04Carleman} with $u_1=u$, $u_2=0$, $p_1=p_2=q_1=q_2=\frac{d}2+1$, $V=0$, $F(z)=-|z|^{\frac4d}z$, and $G=|\na F|$.
\end{proof}

\section{Nonexistence of self-similar blow-up rate solutions}

In this section, we prove Theorem~\ref{thm:no self similar} by contradiction. Assuming that Theorem~\ref{thm:no self similar} fails, in Section~\ref{subsec:reduction to an AP case}, we extract an almost periodic
solution modulo the self-similar scaling. This reduction relies on a concentration-compactness argument that is rather a standard tool in the global theory of critical dispersive equations. In Section~\ref{subsec:nonexistence of AP}, we rule out such solutions. This part constitutes the core of our analysis.
\subsection{Reduction to almost periodic solutions modulo self-similar scaling}\label{subsec:reduction to an AP case}
We first introduce the notion of almost periodicity modulo self-similar scaling.

\begin{defn}[Almost periodic modulo self-similar scaling]\label{defn:self similar AP}
A solution $u$ to \eqref{eq:NLS} with lifespan $(0,\infty)$ is said to be \emph{almost periodic modulo self-similar scaling} if the set $\{t^{\frac d4}u(t,t^{\frac12} \cdot):t>0\}$ is precompact in $L^2(\R^d)$.
\end{defn}
\begin{prop}[Reduction to a solution almost periodic modulo self-similar scaling]\label{prop:main for AP}
Assume Theorem~\ref{thm:no self similar} fails. Then, there exists a solution $u$ of lifespan $(0,\infty)$ that is almost periodic modulo self-similar scaling.
\end{prop}
The rest of this subsection is devoted to the proof of Proposition~\ref{prop:main for AP}. Negating Theorem~\ref{thm:no self similar}, there exist $\ep_*>0$ and $c_*<M(Q)$ such that
\begin{equation}\label{eq:measurement of self-similar}
\sup_{t\in[T_0,T)}\sup_{x_0\in\R^d}M\bb{\chi_{x_0+[-\ep_*\sqrt {T-t},\ep_*\sqrt{T-t}]^d}u(t)}\le c_*
\end{equation}
holds for some $T_0<T$ and a solution $u$ on $[T_0,T)$ with forward lifespan $T$.
Via time-translation and scaling, we may further assume $T_0=0$ and $T=1$ in  \eqref{eq:measurement of self-similar}.
We denote by $\mc A$ the nonempty set
\begin{align*}
\mc A=&\{\phi\in L^2(\R^d):\text{the solution }u\text{ to \eqref{eq:NLS} with initial data }u(0)=\phi
\\&\text{has forward lifespan }T=1\text{ and satisfies \eqref{eq:measurement of self-similar} with $T_0=0$}\}.
\end{align*}
We denote the critical mass by $M_c=\inf\{M(\phi):\phi\in\mc A\}$.

The proof of Proposition~\ref{prop:main for AP} relies on conventional profile decompositions \cite{BahouriGerard,KenigMerle}.
\begin{prop}[Nonlinear profile decomposition]\label{prop:chenjie linear}
Let $\{ z_n\}_{n\in\N}$ be a bounded sequence in $L^2(\R^d)$. After passing to a subsequence (still denoted by $n$), there exist profiles $\{\phi^j\}_{j\in\N}\subset L^2(\R^d)$, $\{\g_{n,j}\}_{j,n\in\N}\subset\G_\aff$, and solutions $\{v^j\}_{j\in\N}$
to \eqref{eq:NLS} with lifespans $(T_-^j,T_+^j)\subset \R$ satisfying the following:
\begin{enumerate}
\item\label{enu:profile 1}
Denoting, for each $J\in\N$,
\[
 z_{n}^{>J}\coloneqq  z_n-\sum_{j\le J}\g_{n,j}(e^{it\De}\phi^j)(0),
\]
we have the decoupling of mass
\begin{equation}\label{eq:profile decomp-mass decoupling}
\lim_{n\to\infty}M( z_n)-M( z_{n}^{>J})-\sum_{j\le J}M(\phi^j)=0.
\end{equation}
Moreover, $\g_{n,j}0\in(-\infty,0]$ and $v^j$ behave as one of the following:
\begin{subequations}\label{eq:profile 2}
\begin{align}
\label{eq:profile 2 case g=0}
&\g_{n,j}0=0,\quad T_-^j<0<T_+^j, \quad v^j(0)=\phi^j
\\
\label{eq:profile 2 case g->infty}
&\g_{n,j}0\to-\infty,\quad T_-^j=-\infty, \quad e^{-it\De}v^j(t)\to\phi^j\text{ as }t\to-\infty.
\end{align}
\end{subequations}

\item\label{enu:profile 3}
Let $\{T_n\}_{n\in\N}\subset[0,\infty)$ be a time sequence such that for each $j$, either
\begin{equation}\label{eq:profile 3 condt}
v^j \text{ scatters forward}\quad \text{or}\quad \limsup_{n\to\infty} \g_{n,j} T_n<T_+^j.
\end{equation}
Let $\{u_n\}$ be a sequence of solutions to \eqref{eq:NLS} with initial data $\{ z_n\}$.
Then, the lifespans of $\g_{n,j}v^j$ and $u_n$ contain $[0,T_n]$ for $n\gg1$ and
\begin{equation}\label{eq:profile decomp-nonlinear profile decomp}
\limsup_{J\to\infty}\limsup_{n\to\infty}\norm{u_n-\sum_{j\le J} \g_{n,j}v^j-e^{it\De} z_{n}^{>J}}_{C^0L^2\cap L^pL^p([0,T_n]\times\R^d)}=0.
\end{equation}
Moreover, for each $j\in\N$ and $\tau\in\R$, along any sequence of $n$ such that $\tau\in[\g_{n,j}0,\g_{n,j}T_n]$, we have the $L^2$-weak convergence
\begin{equation}\label{eq:nonlinear weak lim}
(\g_{n,j}^{-1}u_n)(\tau)\weak v^j(\tau).
\end{equation}
\end{enumerate}
\end{prop}
\begin{proof}
This proposition is essentially standard; for an explicit reference, we refer to \cite{Fan21sequential}. Note that $\g_{n,j}0\le0$ is assumed since we consider only forward evolution.

See \cite[(2.22), (2.25), Definitions 2.4--2.5]{Fan21sequential} for \eqref{enu:profile 1} and \cite[Proposition 2.11]{Fan21sequential} for \eqref{eq:profile decomp-nonlinear profile decomp}.
The weak convergence \eqref{eq:nonlinear weak lim} is also standard. For a comprehensive proof, one can follow the proof in \cite[Lemma 5.10]{KillipVisanNote} with the $L^2$-inner product, replacing $v^j_n$ therein ($\g_{n,j}v^j$ in our notation) by $\g_{n,j}(e^{it\De}\phi)$ for a test function $\phi\in L^2(\R^d)$.
\end{proof}
Hereafter, by scaling, we assume that $T_+^j$ in Proposition~\ref{prop:chenjie linear} is either $1 $ or $\infty $.

We recall a consequence of the sub-threshold global well-posedness for mass-critical NLS in \cite{KillipTaoVisan09radial-2D,KillipVisanZhang08radial-3D+}.
The following concentration property is its byproduct:

Let $u$ be a radial solution to \eqref{eq:NLS} with forward lifespan
$T<\infty$.
Then we have
\begin{equation}\label{eq:KillipVisan conseq}
\limsup_{t\to T}
M\bigl( \one_{|x|\le (T-t)^{1/3}}\, u(t) \bigr)
\ge M(Q).
\end{equation}

See \cite[Corollary~1.12]{KillipTaoVisan09radial-2D} for more details. Moreover, as observed in \cite{KillipTaoVisan09radial-2D,KillipVisanZhang08radial-3D+}, the argument essentially
does not rely on radial symmetry or the dimension. Repeating the proof for general solutions yields the following
nonradial analogue:

For any solution $u$ to \eqref{eq:NLS} with forward
lifespan $T<\infty$, we have
\begin{equation}\label{eq:KillipVisan conseq nonradial}
\limsup_{t\to T}\;
\sup_{x_0\in\R^d}
M\bigl(
\chi_{x_0+[-(T-t)^{1/3},(T-t)^{1/3}]^d}\, u(t)\bigr)\ge M(Q).
\end{equation}
\begin{lem}\label{lem:vj is in A}
Let $\{ z_n\}_{n\in\N}\subset\mc A$ be a sequence such that $M( z_n)\to M_c$. Let $\{v^j\}_{j\in\N}$ be as in Proposition~\ref{prop:chenjie linear}.
Then, there exists an index $j$ such that $v^j(0) \in \mc A$.
\end{lem}
\begin{proof}
We use the notation given in Proposition~\ref{prop:chenjie linear}.
Up to relabeling, we assume that $j=1,\ldots,j_0$ are the indices such that $v^j$ does not scatter forward. For each $n$, since $ z_n\in\mc A$, $u_n$ blows up at time $1$. Thus, $j_0\ge1$ follows.

Suppose for contradiction that this lemma fails.
We show that, for each $j\le j_0$, there exist times $0\le\tau^j<\tau^j_+<T_+^j$ such that
\begin{equation}\label{eq:tauj tauj+ xj}
\sup_{x_0\in\R^d}M\bb{\chi_{x_0+[-\ep_*\sqrt{\tau^j_+-\tau^j},\ep_*\sqrt{\tau^j_+-\tau^j}]^d}v^j(\tau^j)}> c_*.
\end{equation}

We treat separately the cases $T_+^j=\infty$ and $T_+^j=1$. If $T_+^j=\infty$, we set $\tau^j=0$. Since $v^j$ does not scatter forward, $M(v^j)\ge M(Q)>c_*$ follows, giving \eqref{eq:tauj tauj+ xj} for sufficiently large $\tau^j_+$.
If $T_+^j=1$, then by $v^j(0)\notin\mc A$, there exists $\tau^j\in[0,1)$ such that
\begin{equation}\label{eq:measurement of v1}
\sup_{x_0\in\R^d}M\bb{\chi_{x_0+[-\ep_*\sqrt{1-\tau^j},\ep_*\sqrt{1-\tau^j}]^d}v^j(\tau^j)}> c_*.
\end{equation}

Setting $\tau^j_+=1-\ep$ with a sufficiently small $\ep>0$ yields \eqref{eq:tauj tauj+ xj}, as claimed.

Let $k\le j_0$ be an index such that, after passing to a subsequence,
\begin{equation}\label{eq:Tn def}
\g_{n,k}^{-1}\tau^k_+=\min_{j\le j_0}\g_{n,j}^{-1}\tau^j_+\eqqcolon T_n.
\end{equation}

For each $j\le j_0$, by $\g_{n,j}0\le0$, $0\le\tau^j<\tau^j_+<T^j_+$, and \eqref{eq:Tn def}, we have
\begin{equation}\label{eq:condition for profile 3}
\limsup_{n\to\infty}\g_{n,j}T_n\le\tau_+^j<T_+^j,\quad \tau^k\in[\g_{n,k}0,\g_{n,k}T_n],
\end{equation}
which is the assumption of \eqref{eq:nonlinear weak lim} for the index $k$ with $\tau=\tau^k$. By \eqref{eq:nonlinear weak lim}, we have
\begin{equation}\label{eq:un conv to vj1}
(\g_{n,k}^{-1}u_n)(\tau^{k})\weak v^k(\tau^{k})\quad\text{ as }\,n\to\infty.
\end{equation}

By \eqref{eq:tauj tauj+ xj} and \eqref{eq:un conv to vj1}, we have
\begin{align}\label{eq:prot for un is A, 2}
&\limsup_{n\to\infty}\sup_{x_0\in\R^d}M\bb{\chi_{x_0+[-\ep_*\sqrt{\g_{n,k}^{-1}\tau^k_+-\g_{n,k}^{-1}\tau^k},\ep_*\sqrt{\g_{n,k}^{-1}\tau^k_+-\g_{n,k}^{-1}\tau^k}]^d}u_n(\g_{n,k}^{-1}\tau^{k})}
\\
=&\limsup_{n\to\infty}\sup_{x_0\in\R^d}M\bb{\chi_{x_0+[-\ep_*\sqrt{\tau^k_+-\tau^k},\ep_*\sqrt{\tau^k_+-\tau^k}]^d}(\g_{n,k}^{-1}u_n)(\tau^{k})}> c_*\nonumber.
\end{align}

Since $z_n\in\mc A$, $u_n$ satisfies \eqref{eq:measurement of self-similar}. Comparing \eqref{eq:measurement of self-similar} with \eqref{eq:prot for un is A, 2} yields $\g_{n,k}^{-1}\tau_+^k\ge1$ along a subsequence. By \eqref{eq:Tn def}, we obtain $\g_{n,j}^{-1}\tau_+^j\ge1$, implying a version of \eqref{eq:profile 3 condt}
\begin{equation}\label{eq:proof of vj is in A}
\limsup_{n\to\infty}\g_{n,j}1\le\tau_+^j< T_+^j,\quad j\le j_0.
\end{equation}
Hence, $u_n$ exists on $[0,1]$ for $n\gg1$. This contradicts the assumption $z_n\in\mc A$.
\end{proof}
Let $\H$ be the group of transforms $e^{i\theta}\h_{\xi,x_0}$, $(\theta,\xi,x_0)\in(\R/2\pi\Z)\times\R^d\times\R^d$ acting on $L^2(\R^d)$ by
\[
\h_{\xi,x_0}\phi(x)\coloneqq e^{-ix\cdot\xi}\phi(x+x_0),\quad\phi\in L^2(\R^d).
\]
\begin{lem}\label{lem:one profile stays in AT}
Let $\{ z_n\}_{n\in\N}\subset\mc A$ be a sequence satisfying $M( z_n)=M_c$. Let $\{u_n\}_{n\in\N}$ be a sequence of solutions to \eqref{eq:NLS} with initial data $u_n(0)=z_n$. Assume that $u_n\mid_{t<0}$ is unbounded in $L^pL^p$.
Then, along a subsequence, $\H z_n$ converges in the quotient topology $L^2(\R^d)/\H$.
\end{lem}
\begin{proof}
We use the notation given in Proposition~\ref{prop:chenjie linear}. By Lemma~\ref{lem:vj is in A}, there exists $j\in\N$ such that $v^j(0)\in\mc A$. Since $M(v^j)\ge M_c$, \eqref{eq:profile decomp-mass decoupling} yields $M( z_{n}^{>j})\to0$ and $\phi^j$ is the only nonzero profile.
Moreover, taking $T_n=0$ in the nonlinear profile decomposition
\eqref{eq:profile decomp-nonlinear profile decomp} and passing to a
subsequence, we obtain
\begin{equation}\label{eq:phi-v1 goes zero}
\| z_n-(\g_{n,j}v^j)(0)\|_{L^2}\to 0
\quad \text{as } n\to\infty.
\end{equation}

We consider two alternatives of \eqref{eq:profile 2} separately. 

Assume that \eqref{eq:profile 2 case g->infty} holds. Then, $v^j$ scatters backward in time. By scaling invariance of the Strichartz norm $L^pL^p$, it follows that
$\{\g_{n,j}v^j\}_{n\in\N}$ is bounded in $L^pL^p((-\infty,0]\times\R^d)$.
Combining this boundedness with \eqref{eq:phi-v1 goes zero} and perturbation theory contradicts the $L^pL^p$-unboundedness of $u_n\mid_{t<0}$.

Next, we turn to the case \eqref{eq:profile 2 case g=0}. Since $v^j(0)\in\mc A$, $v^j$ has forward lifespan $1$. Thus, $\g_{n,j}v^j$ has forward lifespan $\g_{n,j}^{-1}1$. The concentration bound \eqref{eq:KillipVisan conseq nonradial} yields
\begin{equation}\label{eq:conc to 0 g}
\limsup_{t\to \g_{n,j}^{-1}1}\sup_{x_0\in\R^d}M\bb{\chi_{x_0+[-(\g_{n,j}^{-1}1-t)^{1/3},(\g_{n,j}^{-1}1-t)^{1/3}]^d}\g_{n,j}v^j(t,x)}\ge M(Q)>c_*.
\end{equation}
Since $ z_n\in\mc A$, $u_n$ satisfies \eqref{eq:measurement of self-similar}. Thus, by \eqref{eq:profile decomp-nonlinear profile decomp} and \eqref{eq:conc to 0 g}, we obtain $\lim_{n\to\infty}\g_{n,j}^{-1}1\ge 1$ after passing to a subsequence. Since each $u_n$ blows up at time $t=1$, the condition \eqref{eq:profile 3 condt} with $T_n=1$ cannot hold. Hence, we have $\limsup_{n\to\infty}\g_{n,j}1\ge1$.

Consequently, along a subsequence, we have $\g_{n,j}1\to 1$ as $n\to\infty$. This implies $\H(\g_{n,j}v^j)(0)\to \H v^j(0)$ in $L^2/\H$; in view of \eqref{eq:phi-v1 goes zero}, the proof is complete.
\end{proof}
We now finish the proof of Proposition~\ref{prop:main for AP}.
\begin{proof}[Proof of Proposition~\ref{prop:main for AP}]
By the minimality of $M_c$, there exists a sequence $\{ z_n\}_{n\in\N}\subset\mc A$ such that $M( z_n)\to M_c$. By Lemma~\ref{lem:vj is in A}, there exists $ \phi\in\mc A$ with $M( \phi)=M_c$.
We introduce a flow map $\Phi$ on $L^2(\R^d)$ associated with \eqref{eq:NLS} in self-similar variables, defined as follows: for $s\in\R$ and a solution $u$ to \eqref{eq:NLS} on $[-1,-e^{-s}]$,
\begin{equation}\label{eq:Phi def}
\Phi(s)u(-1)\coloneqq t^{\frac d4}u(-t,t^{\frac12}x),\quad t=e^{-s}.
\end{equation}

Let $\Phi_\H$ be the flow map on $L^2/\H$ induced by $\Phi$. Denote $\phi_\H=\H\phi\in L^2/\H$.

Let $w$ be a solution to \eqref{eq:NLS} with $w(-1)= \phi$. Since $ \phi\in\mc A$, $w$ has forward lifespan $[-1,0)$ and, by scaling, we have $\Phi(s) \phi\in\mc A$ for $s\ge0$.

By Lemma~\ref{lem:one profile stays in AT}, the forward $\omega$-limit $\omega^+(\phi_\H)\subset L^2/\H$ is nonempty and compact. We choose $\phi_\H^*\in\omega^+(\phi_\H)$, then $\phi_\H^*$ generates a complete $\Phi_\H$-orbit for $s\in\R$.
Let $\phi^*\in L^2$ be a representative of $\phi_\H^*$ and $w^*$ be a solution to \eqref{eq:NLS} with $w^*(-1)=\phi^*$, which aligns with \eqref{eq:Phi def}. Then, since $\Phi_\H(s)\phi_\H^*$ is defined for all $s\in\R$, $w^*$ exists on $(-\infty,0)$.
Let $u$ be the time-reversal of $w^*$. Since the $\Phi_\H$-orbit of $\phi^*_\H$ is precompact, $\{t^{\frac d4}\H u(t,t^{\frac12}\cdot):t>0\}$ is precompact in $L^2/\H$.

We now remove the quotient $\H$ in the compactness.
We use a standard fact (see, e.g., \cite[(2.19)]{Dodson15focusing}): there exist symmetry parameters $\xi(\cdot),x(\cdot):(0,\infty)\to\R^d$ such that
\begin{equation}\label{eq:h precompact}
\{t^{\frac d4}\h_{t^{\frac12}\xi(t),t^{-\frac12}x(t)}u(t,t^{\frac12}\cdot)\}=\{t^{\frac d4}(\h_{\xi(t),x(t)} u(t))(t^{\frac12} \cdot)\}\text{ is $L^2$-precompact}
\end{equation}
and $\d_t\xi=O(t^{-\frac32})$.
Then, $\xi(t)=O(t^{-\frac12})+\xi_\infty$ holds with $\xi_\infty=\lim_{t\to\infty}\xi(t)$, which can be set as $0$ up to a Galilean transform of $u$. Then $t^{\frac12}\xi(t)=O(1)$, hence \eqref{eq:h precompact} holds with $\xi(t)=0$. Similarly, we can choose $x(t)$ such that \eqref{eq:h precompact} holds with $\h_{0,x(t)}$ and $\d_tx=O(t^{-\frac12})$. After translating $u$ by $x_0=\lim_{t\to0}x(t)$, $x(t)=O(t^{\frac12})$ holds. Then \eqref{eq:h precompact} follows with $\xi(t)=x(t)=0$, completing the proof.
\end{proof}
We use a regularity gain result due to Dodson \cite{Dodson12defoc-3D+,Dodson16defoc-2D,Dodson16defoc-1D}, originally developed in the study of global well-posedness for the mass-critical NLS \cite{Dodson12defoc-3D+,Dodson16defoc-2D,Dodson16defoc-1D,Dodson15focusing}.
\begin{prop}[Gain of regularity]\label{prop:Dodson}
Let $u$ be as in Proposition~\ref{prop:main for AP}. For any $\theta\in[0,1+\frac4d)$, we have $ u(t)\in H^\theta(\R^d)$ and 
\begin{equation}\label{eq:Dodson}
\norm{u(t)}_{\dot H^\theta(\R^d)}\les_u t^{-\frac\theta2},\quad t>0.
\end{equation}
\end{prop}
\begin{proof}
This gain of regularity \eqref{eq:Dodson} is a direct consequence of
\cite[Theorem~5.1]{Dodson16defoc-1D} in $d=1$, \cite[Theorem~5.1]{Dodson16defoc-2D} in $d=2$, and
\cite[Theorem~3.13]{Dodson12defoc-3D+} in $d\ge3$.
\end{proof}
We use underline notation for the self-similar renormalization,
\[
\ul u(s,y)=t^{\frac d4}u(t,x),\quad t=e^s,\quad x=t^{\frac12}y.
\]
Let $u^\pc$ be the time-reversed pseudo-conformal transform of $u$
\begin{equation}\label{eq:u pc def}
u^\pc(t,x)\coloneqq |t|^{-\frac d2}e^{\frac i{4t}|x|^2}\ol{u}(t^{-1},t^{-1}x),\quad t>0.
\end{equation}
For functions with superscripts $\pc$, we use the underline notation with time-reversal of the variable $s$, which is convenient for the following identity:
\begin{equation}\label{eq:ul u^pc}
\ul{u^{\pc}}(s,y)=e^{\frac{i}{4}|y|^2}\,\ol{\ul{u}}(s,y).
\end{equation} 
As a result, $u^{\pc}$ is also almost periodic modulo self-similar scaling.
We frequently exploit $u^{\pc}$ to derive the spatial decay for $u$ via the
identity
\begin{equation}\label{eq:ix/2t u}
iy\ul u = 2\nabla\ul u - 2 e^{\frac{i}{4}|y|^2}\nabla\oul{u^\pc}.
\end{equation}
By \eqref{eq:Dodson}, \eqref{eq:ul u^pc}, and \eqref{eq:ix/2t u}, we obtain for $0\le \theta<1+\frac4d$
\begin{equation}\label{eq:yu Hs}
\norm{\ul u(s)}_{H^\theta}+\norm{y\ul u(s)}_{L^2}\les_{\theta}1.
\end{equation}
For $d=1,2$, since $H^{\theta}(\R^d)\hook W^{1,\infty}(\R^d)$ for $2<\theta<3$, we similarly obtain
\begin{equation}\label{eq:u 4/d}
\norm{y\ul u(s)}_{L^\infty}\les1.
\end{equation}
By \eqref{eq:yu Hs}, the orbit of $\ul u(s)$ is precompact in $H^1(\R^d)$.
Hence, we obtain
\begin{align}\label{eq:rv2 goes zero}
r\norm{\ul u(s)}_{L^2(r\S^{d-1})}^2\les r\norm{|\ul u|(s)}_{L^2(|y|\ge r)}\norm{\na|\ul u|(s)}_{L^2(|y|\ge r)}=o_r(1),\quad s\in\R.
\end{align}
\subsection{Nonexistence of almost periodic solution modulo self-similar scaling}\label{subsec:nonexistence of AP}
In this subsection, we fix an almost periodic self-similar solution $u$ in Proposition~\ref{prop:main for AP}. Our goal is to prove the following proposition, which concludes Theorem~\ref{thm:no self similar}.
\begin{prop}\label{prop:no self similar AP} There does not exist a solution $u$ given in Proposition~\ref{prop:main for AP}.
\end{prop}
We define $v$ by the transform of $u$
\begin{equation}\label{eq:Lens}
v(s,y)=e^{-\frac i8|y|^2}\ul u(s,y).
\end{equation}
Note that $v$ is a Lens transform of $u$; see, for example, \cite[Chapter 8.3]{Fibich2015NLS}. Such quadratic phase arises naturally in connection with the pseudo-conformal symmetry.

$v$ solves the nonlinear Schr\"odinger equation with a repulsive harmonic potential,
\begin{equation}\label{eq:NLS harmonic}
i\d_sv+\De v+\tfrac1{16}|y|^2v+|v|^{\frac4d}v=0.
\end{equation}

We use the virial identity for $v$ with a weight function.
For a multiplier $\varphi\in C_c^\infty(\R^d)$, we denote the virial functional of $v(s)$ by
\begin{equation}\label{eq:V def}
V_{\varphi}(s)\coloneqq \int_{\R^d} |v(s,y)|^2 \varphi(y) dy .
\end{equation}
Then $V_\varphi$ satisfies the virial identities
\begin{align}\label{eq:V'=W}
\d_sV_\varphi&=2\Im\int_{\R^d}\ol v\na v\cdot\na\varphi  dy,
\\\label{eq:Virial in v init}
\d_{ss}V_\varphi&=\int_{\R^d}\frac y4|v|^2\cdot\na\varphi+4\ol{\d_{y_j}v}\d_{y_k}v\d_{y_jy_k}\varphi-|v|^2\De^2\varphi-\frac4{d+2}|v|^{2+\frac4d}\De\varphi dy.
\end{align}
The last term in \eqref{eq:Virial in v init}, which arises from the nonlinearity in \eqref{eq:NLS}, is always treated as an error term. The core analysis is instead governed by the quadratic terms in \eqref{eq:Virial in v init}, which originate from the linear Schr\"odinger equation.

We note the regularity and decay of $v$:
by \eqref{eq:yu Hs} and \eqref{eq:u 4/d}, we have
\begin{align}\label{eq:v cpt}
&\norm{v(s)}_{H^1}+\norm{yv(s)}_{L^2}\les 1,&d\ge1,
\\
\label{eq:v 4/d is -2}
&|v(s,y)|=|\ul u(s,y)|\les_u\min\{1,|y|^{-1}\}\les |y|^{-\frac d2},& d=1,2.
\end{align}
We now prove Proposition~\ref{prop:no self similar AP}.
The argument splits according to the dimension $d$.
For $d=1,2$, we develop an integro-differential inequality that excludes self-similar blow-up, using the virial identity \eqref{eq:Virial in v init} with a suitable family of weights $\{\varphi^a\}$. This inequality technically relies on the pointwise decay \eqref{eq:v 4/d is -2} available for $d=1,2$; for $d\ge 3$, we make a different approach based on the $L^2_t$-Carleman inequality \eqref{eq:Carleman}.

\begin{proof}[Proof of Proposition~\ref{prop:no self similar AP} for $d=1,2$]
For $a>0$,
let $\varphi^{a}$ be a radial weight function
\[
\varphi^{a}(x)=\max\{|x|-a,0\},\quad x\in\R^d.
\]
We will deduce from \eqref{eq:Virial in v init} a single integro-differential inequality of $a$. For this, we take a weighted average
on the variable $s$. For $F:\R\to\C$, we denote
\footnote{The choice of exponent $\frac14$ is arbitrary; our analysis works with $\mu(s)=e^{-\nu|s|}$ for any $\nu<\frac12$.}
\[
\I F\coloneqq \int_\R F(s)\mu(s) ds,\quad \mu(s)\coloneqq e^{-\frac14|s|}.
\]
For a radial weight function $\varphi\in C^\infty_c(\R^d)$, by an integration by parts, we have
\begin{equation}\label{eq:V''F-VF''}
\I \d_{ss}V_{\varphi}=\int_\R \d_{ss}V_{\varphi}\cdot\mu ds=\int_\R V_{\varphi}\cdot\d_{ss}\mu ds=\frac1{16}\I V_{\varphi}-\frac12 V_\varphi(0).
\end{equation}
We denote by $\cancel{\na}=\na-\frac y{|y|}\d_r$ the angular gradient.
By \eqref{eq:V''F-VF''} and the identity
\[
\d_{x_jx_k}\varphi\ol{\d_{x_j}v}\d_{x_k}v=\varphi_{rr}|v_r|^2+\frac1r\varphi_r|\cancel\na v|^2,
\]
\eqref{eq:Virial in v init} can be rewritten as
\begin{align}\label{eq:Virial in v<VF''-V''F}
&\I\int_{\R^d}\frac r4 \varphi_r|v|^2+4\varphi_{rr}|v_r|^2+\frac{4}{r}\varphi_r|\cancel{\na} v|^2-\De\varphi\De(|v|^2)-\frac{4}{d+2}(\Delta\varphi)|v|^{2+\frac4d}  dy
\\=&\frac1{16}\I V_\varphi-\frac12 V_\varphi(0).\nonumber
\end{align}

Although $\varphi^a\notin C_c^\infty(\R^d)$, the identity \eqref{eq:Virial in v<VF''-V''F} stays valid for $\varphi=\varphi^a$. Indeed, one may approximate $\varphi^a$ by  $\varphi\in C^\infty_c(\R^d)$ and pass to a limit using \eqref{eq:yu Hs} and \eqref{eq:v cpt}.

We also introduce the spherical mass density
\[
W(s,r)\coloneqq \int_{r\S^{d-1}}|v(s,y)|^2 d\s,\quad r>0.
\]
Define 
\[ 
\mc E\coloneqq \{a \in (0,\infty):  \I W(a)\neq 0\}. 
\]
Then for each $a\in\mc E$, using $\varphi^a_{rr}=\delta_0(|y|-a)$, $2|v||v_r|\ge|(|v|^2)_r|$, and Cauchy-Schwarz, we estimate
\begin{equation}\label{eq:CS EW}
\I\int_{\R^d}4\varphi_{rr}^a|v_r|^2 dy=\I\int_{a\S^{d-1}}4|v_r|^2 d\s\ge\frac{\left|\I\int_{a\S^{d-1}}(|v|^2)_r d\s\right|^2}{\I\int_{a\S^{d-1}}|v|^2 d\s}.
\end{equation}
An integration by parts on the sum of \eqref{eq:Virial in v<VF''-V''F} and \eqref{eq:CS EW} yields
\begin{align}\label{eq:Virial in EW}
&\int_a^\infty\frac r4 \I W dr+\Bb{\frac{(\I W_r)^2}{\I W}-\I W_{rr}}(a)+\int_a^\infty\frac{(d-1)(d-3)}{r^3} \I W  dr
\\&-\frac{4}{d+2}\I\int_{\R^d}(\De\varphi^a)|v|^{2+\frac4d} dy\le \frac1{16}\I V_{\varphi^a}-\frac12 V_{\varphi^a}(0).\nonumber
\end{align}
Noticing that $V_{\varphi^a}(s)=\int_a^\infty(r-a)W(s,r) dr$, \eqref{eq:Virial in EW} yields
\begin{subequations}\label{eq:Virial fight mother}
\begin{align}
\int_a^\infty \frac{a+r}8 \mathcal I W dr
+\Bigl(\frac{(\mathcal I W_r)^2}{\mathcal I W}-\mathcal I W_{rr}\Bigr)(a)
+\int_a^\infty\frac{(d-1)(d-3)}{r^3}\,\mathcal I W dr
\label{eq:Virial fight}
\\
-\frac{4}{d+2}\,
\mathcal I\!\int_{\R^d}(\Delta\varphi^a)\,|v|^{2+\frac4d} dy
\;\le\;
-\frac12 V_{\varphi^a}(0)
\;\le\;0 .
\label{eq:tag nonlinear integral for Sec 5}
\end{align}
\end{subequations}
Using the decay bound \eqref{eq:v 4/d is -2}, we estimate \eqref{eq:Virial fight mother} as follows: there exists $c=O_u(1)$ such that for every $a\in\mc E$,
\[
\int_a^\infty\frac{a+r}8\I W(r)dr+\Bigl(\frac{(\I W_r)^2}{\I W}-\I W_{rr}\Bigr)(a)
-\frac c{a^2}\I W(a)-\int_a^\infty \frac c{r^3}\I W(r)dr\le 0.
\]
We perform an logarithmic change of variable on $\I W$ by setting
\begin{equation}\label{eq:ga def}
\gamma(r)\coloneqq-\log\I W(r),\quad r\in\mc E.
\end{equation}
In terms of $\gamma$, we obtain an ODE convexity inequality
\begin{equation}\label{eq:ODE ineq}
\int_a^\infty\Bb{\frac{a+r}8-\frac c{r^3}} e^{-\gamma(r)} dr+\Bb{\gamma''(a)-\frac c{a^2}}e^{-\gamma(a)}\le0,\quad a\in\mc E.
\end{equation}
Denote $a_*=2c^{\frac14}$. Noticing that $\frac r8\ge\frac c{r^3}$ for $r\ge a_*$, multiplying $e^{\gamma(a)}$ to \eqref{eq:ODE ineq} yields
\begin{equation}\label{eq:ODE ineq for a!=a0}
\gamma''(a)-\frac c{a^2}\le
-\frac a8\int_a^\infty e^{\gamma(a)-\gamma(r)} dr,\quad a\in\mc E\cap(a_*,\infty).
\end{equation}
Proposition~\ref{prop:Carleman uniqueness} yields $\E\cap(a_*,\infty)\neq\emptyset$. Then, by \eqref{eq:ODE ineq for a!=a0}, $\E\supset(a_*,\infty)$ follows.

By \eqref{eq:ODE ineq for a!=a0}, we obtain
\begin{equation}\label{eq:gamma' bound}
\gamma'(a)\le\gamma'(a_*)+\int_{a_*}^a\frac{c}{r^2}dr\le\gamma'(a_*)+\frac c{a_*},\quad a>a_*.
\end{equation}
On the other hand, since $\int_0^\infty e^{-\gamma(r)} dr=8M(v)$ is finite, we have
\begin{equation}\label{eq:limsup gamma}
\limsup_{a\to\infty}\gamma(a)=\infty.
\end{equation}
By \eqref{eq:limsup gamma}, the right-hand side of \eqref{eq:gamma' bound} is positive. Hence, \eqref{eq:ODE ineq for a!=a0} can be reduced to
\begin{equation}\label{eq:ODE ineq corrected}
\gamma''(a)-\frac c{a^2}\le -\frac a8\Bigl(\gamma'(a_*)+\frac c{a_*}\Bigr)^{-1},
\quad a>a_*.
\end{equation}
Integrating \eqref{eq:ODE ineq corrected} twice in $a$ forces $\gamma(a)\to -\infty$
as $a\to\infty$, contradicting \eqref{eq:limsup gamma}. This completes the proof.
\end{proof}
We turn to dimensions $d\ge 3$. In high dimensions, \eqref{eq:Dodson} is not sufficient to deduce \eqref{eq:v 4/d is -2} and we instead use a different argument based on the $L^2_t$-Carleman inequality \eqref{eq:Carleman}. We begin with a geometric heuristic.
If one regards the effective support of $u$ as the parabolic region
\[
\Gamma=\{(t,x): |x|\le \sqrt{t}\},
\]
then the effective support of $\g_1 u$ is an ellipsoidal region whose diameter is maximal at time $\g_1 1=1$ (see Remark~\ref{rem:g1 def}).
This is conceptually in contrast to the dispersive behavior of a solution. To obtain a contradiction, we quantify the effective support radius via exponentially weighted $L^2_x$ norms and employ the Carleman inequality.

As a preparatory step, we first establish sufficient decay of $v$.
\begin{lem}\label{lem:poly bound}
Let $m=2d+2$. Then, for all $s\in\R$, we have $V_{r^m}(s)=O(1)$.
\end{lem}
\begin{proof}
By scaling invariance, we may assume $s=0$.

For $\rho\ge1$, let $\varphi_{\rho}\in C^4(\R^d)$ be a radial function defined by
\[
\varphi_\rho(0)=\d_r\varphi_\rho(0)=\d_{rr}\varphi_\rho(0)=0,\quad \d_{rrr}\varphi_\rho(r)=\psi^{\le\rho}(r)\cdot\d_{rrr}(r^m),
\]
so that $\varphi_\rho(r)=r^m$ for $r\le\rho$ and is a quadratic polynomial for $r\ge2\rho$. We claim
\begin{equation}\label{eq:poly claim}
V_{\varphi_{\rho}}(0)=O(1)
\end{equation}
with a bound independent of $\rho$.
Once \eqref{eq:poly claim} is established, letting $\rho\to\infty$ yields $V_{r^m}(0)=O(1)$ as desired.
By \eqref{eq:Virial in v<VF''-V''F}, for $\varphi\in C^\infty_c(\R^d)$, we have
\begin{align}\label{eq:Virial for d>=3}
\frac1{16}\I V_\varphi-\frac12V_\varphi(0)=&\I\int_{\R^d}\frac r4 \varphi_r|v|^2+4\varphi_{rr}|v_r|^2+\frac{4}{r}\varphi_r|\cancel{\na} v|^2-(\De^2\varphi)|v|^2
\\
&-\frac{4}{d+2}(\Delta\varphi)|v|^{2+\frac4d} dy.\nonumber
\end{align}
By a limiting argument using \eqref{eq:v cpt}, \eqref{eq:Virial for d>=3} remains valid for
$\varphi=\varphi_{\rho}$.

Henceforth, we fix $\varphi=\varphi_{\rho}$.
Using $\varphi\le r\varphi_r$, we estimate 
\begin{equation}\label{eq:V phi bound}
\frac1{16}\I V_{\varphi}=\frac1{16}\I\int_{\R^d}\varphi|v|^2 dy\le\I\int_{\R^d}\frac r{8}\varphi_r|v|^2 dy.
\end{equation}
Note that $\De^2\varphi(r)\les \varphi_r(r)$ for $r\ge1$. Let $R_{1}=R_1(d)<\infty$ be a number such that
\begin{equation}\label{eq:v square bound}
(\De^2\varphi)|v|^2\le \frac r{16}\varphi_r|v|^2,\quad r\ge R_1.
\end{equation}
We use the Sobolev inequality for the sphere $r\S^{d-1}$, 
\[
\norm{v}_{L^{\frac{2d}{d-2}}(r\S^{d-1})}^2\les r^{\frac2d-2}\norm{v}_{L^2(r\S^{d-1})}^2+r^{\frac2d}\norm{\cancel\na v}_{L^2(r\S^{d-1})}^2
\]
to deduce that for $R\ge 1$, 
\begin{align}\label{eq:Sobolev bound}
&\I\int_{|y|\ge R}(\De\varphi)|v|^{2+\frac4d} dy\le\I\int_R^\infty(\De\varphi)\norm{v}_{L^{\frac{2d}{d-2}}(r\S^{d-1})}^2\norm{v}_{L^2(r\S^{d-1})}^{\frac4d} dr \\ 
\les&\I\int_{|y|\ge R}(\De\varphi)(|v|^2+|\cancel\na v|^2) dy\cdot\sup_{s\in\R}\sup_{r\ge R}r^{\frac2d}\norm{v(s)}_{L^2(r\S^{d-1})}^{\frac4d}.\nonumber
\end{align}
By \eqref{eq:rv2 goes zero}, we have
\begin{equation}\label{eq:r 2/d v 4/d goes zero}
\lim_{R\to\infty}\sup_{s\in\R}\sup_{r\ge R}r^{\frac2d}\norm{v(s)}_{L^2(r\S^{d-1})}^{\frac4d}=0.
\end{equation}
Again, observing that $\De\varphi\les r^{-1}\varphi_r$ for $r\ge1$, and using \eqref{eq:Sobolev bound} and \eqref{eq:r 2/d v 4/d goes zero}, there exists $R_{2}=R_2(d)<\infty$ such that
\begin{equation}\label{eq:Sobolev adapt bound}
\frac4{d+2}\I\int_{|y|\ge R_2}(\De\varphi)|v|^{2+\frac4d} dy\le\I\int_{|y|\ge R_2}\frac r{16}\varphi_r|v|^2+\frac4r\varphi_r|\cancel\na v|^2 dy.
\end{equation}
Combining \eqref{eq:Virial for d>=3}, \eqref{eq:V phi bound}, \eqref{eq:v square bound}, and \eqref{eq:Sobolev adapt bound} altogether, we conclude
\begin{equation}\label{eq:Virial for d>=3, polynomial bound}
\frac12V_\varphi(0)\le\I\int_{|y|<R_1}(\De^2\varphi)|v|^2 dy+\frac4{d+2}\I\int_{|y|<R_2}(\De\varphi)|v|^{2+\frac4d} dy\le O_u(1),
\end{equation}
where we used \eqref{eq:v cpt} for the last inequality. This finishes the proof.
\end{proof}
The next lemma strengthens Lemma~\ref{lem:poly bound} to super-exponential spatial decay of $v$.
\begin{lem}\label{lem:exp bound}
Let $\be>0$. Then, for all $s\in\R$, we have $V_{e^{\be r}}(s)=O_{\be} (1)$.
\end{lem}
\begin{proof}
We mostly follow the proof of Lemma~\ref{lem:poly bound},
with suitable modifications for the exponential weight.
Fix $\rho\ge1$ and let $\varphi=\varphi_{\beta,\rho}$ be a
radial function such that
\[
\varphi(0)=\d_r\varphi(0)=\d_{rr}\varphi(0)=0,\quad \d_{rrr}\varphi(r)=\psi^{\le\rho}(r)\cdot\d_{rrr}(e^{\be r}).
\]
The proof proceeds as in Lemma~\ref{lem:poly bound}, except that the pointwise bound $
\Delta\varphi \les r^{-1}\varphi_r$ for $r\ge1$ should be replaced by the weaker bound $\De\varphi\les_\be\varphi_r$.
Thus, in order~to deduce \eqref{eq:Sobolev adapt bound} from \eqref{eq:Sobolev bound}, we require an $r$-weighted strengthening of \eqref{eq:r 2/d v 4/d goes zero}, namely
\begin{equation}\label{eq:prep for exponential}
\lim_{R\to\infty}\sup_{s\in\R}\sup_{r\ge R} r\cdot r^{\frac2d} \norm{v(s)}_{L^2(r\S^{d-1})}^{\frac4d}=0.
\end{equation}
To verify \eqref{eq:prep for exponential}, we invoke \eqref{eq:yu Hs}, \eqref{eq:rv2 goes zero}, and Lemma~\ref{lem:poly bound}. Then, we obtain
\[
r\sup_{s\in\R}\norm{v(s)}_{L^2(r\S^{d-1})}^2 \lesssim \sup_{s\in\R}r\norm{v(s)}_{L^2(|y|\ge r)}\,\norm{\nabla v(s)}_{L^2(|y|\ge r)} \lesssim_v r^{-d},
\]
which immediately yields \eqref{eq:prep for exponential}.

With \eqref{eq:prep for exponential} in hand, arguing as in Lemma~\ref{lem:poly bound} completes the proof.
\end{proof}
Finally, we rule out the possibility of superexponential decay.
\begin{proof}[Finish of the proof of Proposition~\ref{prop:no self similar AP} for $d\ge3$]
Let $\ep=\ep(v)>0$ and $\be=\be(v)>0$ be numbers to be fixed later.
By \eqref{eq:v cpt}, there exists $R_\ep=R_\ep(v)>0$ such that
\begin{equation}\label{eq:prep for d>3, v}
\sup_{s\in\R}\norm{v(s)}_{L^2(y_1>R_\ep)}\le\ep.
\end{equation}
Lemma~\ref{lem:exp bound} yields $\norm{e^{\be y_1}v(s)}_{L^2}\les_\be1$, thus there exists $s_0\in\R$ such that
\begin{equation}\label{eq:s0 def}
2\norm{e^{\be y_1}v(s_0)}_{L^2}\ge \sup_{s\in\R}\norm{e^{\be y_1}v(s)}_{L^2}\eqqcolon B.
\end{equation}
By Proposition~\ref{prop:Carleman uniqueness}, either of $v(\pm1)$ is nonzero in $\{y_1\ge R_\ep+1\}$. Thus, we have
\begin{equation}\label{eq:B large}
B\gtrsim_{v,\ep} e^{\be(R_\ep+1)}.
\end{equation}
By scaling, we may assume $s_0=0$.
Denoting $t_{\pm1}=\g_1^{-1}e^{\pm1}$, \eqref{eq:Carleman} yields
\begin{align}\label{eq:Carleman ineq for u}
&\norm{e^{\be x_1} \g_1 u(1)}_{L^2}+\norm{e^{\be x_1}\g_1 u}_{L^2L^{2_1}([t_{-1},t_1]\times\R^d)}
\\\les&\max_{t\in\{t_{-1},t_1\}}\norm{e^{\be x_1}\g_1 u(t)}_{L^2}+\norm{e^{\be x_1}|\g_1 u|^{\frac4d}\g_1 u}_{L^2L^{2_{-1}}([t_{-1},t_1]\times\R^d)}.\nonumber
\end{align}
We rewrite \eqref{eq:Carleman ineq for u} in $(s,y)$-variable: on $[-1,1]\times\R^d$,
\begin{equation}\label{eq:Carleman ineq for v}
\norm{\psi_\be v(0)}_{L^2}+\norm{\psi_\be v}_{L^2L^{2_1}}
\les \max_{s=\pm1}\norm{\psi_\be v(s)}_{L^2}+\norm{\psi_\be |v|^{\frac4d}v}_{L^2L^{2_{-1}}},
\end{equation}
where the function $\psi_\be$ is a transform of $e^{\be x_1}$ given as
\[
\psi_\be(s,y)\coloneqq e^{\be x_1/(\frac{1+t}2)}=e^{\be y_1/(\frac{1+t}{2\sqrt t})}=e^{\be y_1\sech(\frac s2)}.
\]
Note that the factor $\frac{1+t}2$ comes from the representation of $\g_1$, \eqref{eq:g1 def}.

We estimate the right-hand side of \eqref{eq:Carleman ineq for v}. Denoting $\eta=\sech(\frac12)<1$, \eqref{eq:s0 def} yields
\begin{equation}\label{eq:v hom}
\norm{\psi_\be v(s)}_{L^2}=\norm{e^{\be\eta y_1}v(s)}_{L^2}\le\norm{e^{\be y_1}v(s)}_{L^2}^\eta\norm{v(s)}_{L^2}^{1-\eta}\les_{v} B^\eta,\quad s=\pm1.
\end{equation}
On $[-1,1]\times\{y_1>R_\ep\}$, \eqref{eq:prep for d>3, v} yields
\begin{equation}\label{eq:v inhom right}
\norm{\psi_\be|v|^{\frac4d}v}_{L^2L^{2_{-1}}(y_1>R_\ep)}\le\ep^{\frac4d}\norm{\psi_\be v}_{L^2L^{2_1}}.
\end{equation}
On $[-1,1]\times\{y_1\le R_\ep\}$, \eqref{eq:v cpt} yields
\begin{equation}\label{eq:v inhom left}
\norm{\psi_\be |v|^{\frac4d}v}_{L^2L^{2_{-1}}(y_1\le R_\ep)}\les e^{\be R_\ep}\norm{v}_{L^\infty L^2}^{\frac4d}\norm{v}_{L^\infty L^{2_1}}\les_v e^{\be R_\ep}.
\end{equation}
Combining \eqref{eq:s0 def}, \eqref{eq:v hom}, \eqref{eq:v inhom right}, \eqref{eq:v inhom left}, and \eqref{eq:Carleman ineq for v} altogether, we deduce that 
\begin{equation}\label{eq:Carleman ineq for v, conseq}
\tfrac12B+\norm{\psi_\be v}_{L^2L^{2_1}}\les_v B^\eta+e^{\be R_\ep}+\ep^{\frac4d}\norm{\psi_\be v}_{L^2L^{2_1}}.
\end{equation}
Choosing $\ep\ll_v1$ and $\beta\gg_{v,\ep}1$ yields a contradiction to \eqref{eq:B large}. This completes the proof.
\end{proof}
\section{Decomposition of a radial solution to \texorpdfstring{\eqref{eq:NLS}}{(NLS)}}\label{sec:regularity}

In this section, we provide a decomposition of a solution $u$ to \eqref{eq:NLS}.
In Section~\ref{sec:main}, we will quantify the radiation from mass concentration in a self-similar region, which is a key observation toward Theorem~\ref{thm:loglog}. To make this precise, we design a decomposition of $u$ into a non-radiative part $\mathring u$ and a radiative part $\tilde u$.

In order to estimate such $\mathring u$ and $\tilde u$, we first provide estimates of the operator $\P_N^\pm\mc D^\mp$.
Hereafter, we always consider radial functions and keep denoting $s=\log t$.

We record a consequence of \eqref{eq:inout bound K}:
Let $\theta\ge0$, $1<\tilde q<q\le\infty$, $1\le r,\tilde r\le\infty$, $\nu\ge0$, and $\rho\ge1$. If $\nu=0$, or if $q<\infty$ and $\nu\ll_{q,\tilde q}1$, then for $f\in C^\infty_c((0,\infty)\times\R^d)$ supported in $\{|x|\le\rho\}$, $\phi\in\mc S(\R^d)$, both $u=\P_1^\pm\mc D^\mp f$ and $u=\P_1^\pm\mc D^\mp(\phi*f)$ satisfy
\begin{equation}\label{eq:ineff Ppm}
\norm{e^{-\nu|s|}\jp x^\theta u}_{L^qL^r((0,\infty)\times\R^d)}\les_{\theta,\phi}\rho^{O_{d,\theta}(1)}\norm {e^{-\nu|s|}f}_{L^{\tilde q,\infty}L^{\tilde r}((0,\infty)\times\R^d)},
\end{equation}
The next lemma shows more advanced estimates.
\begin{lem}
Let $d\ge 2$, $\theta\ge 0$, $\ep>0$, $q\in[1,\infty]$, $0\le\nu\ll_\ep1$, and $a\ge2$.
For radial $f\in C^\infty_c((0,\infty)\times\R^d)$ and $N\in2^\Z$, $u=\P_N^\pm\mc D^\mp f$ satisfies
\begin{align}\label{eq:Ppm N stri}
\norm{\jp{Nx}^{\theta+\frac12}u}_{L^qL^{2_1}((0,\infty)\times\R^d)}&\les_{\theta,\ep} \norm{\jp{Nx}^{\theta+3\ep}f}_{L^qL^{2_{-1}}((0,\infty)\times\R^d)},
\\
\label{eq:Ppm N stri'}
\norm{e^{-\nu|s|}\jp{Nx}^{\theta+\frac12}u}_{L^2 L^{2_1}((0,\infty)\times\R^d)}&\les_{\theta,\ep}\norm{e^{-\nu|s|}\jp{Nx}^{\theta+3\ep}f}_{L^2L^{2_{-1}}((0,\infty)\times\R^d)}.
\end{align}
There exists $k=k(d,\theta)$ such that if $\supp(f)\subset\{|x|\le a\sqrt t\}$, then for $\rho>0$,
\begin{align}\label{eq:Ppm exterior stri}
&\norm{e^{-\nu|s|}\jp{N x}^{\theta}u}_{L^2 L^{2_1}(|x|\ge(a+\rho)\sqrt t)}
\\
\les_{\theta,\rho}& a^k\min\big\{\norm{e^{-\nu|s|}f}_{L^2 L^{2_{-1}}}\,,\,\norm{e^{-\nu|s|}f}_{L^{p',\infty} L^{p'}}\,,\,N^{-\frac12}\norm{e^{-\nu|s|}t^{\frac14}f}_{L^2L^2}\big\}.\nonumber
\end{align}
Let $\ep_d$ be as in \eqref{eq:weighted shao def}.
If $\supp(f)\subset\{(a-1)\sqrt t\le|x|\le a\sqrt t\}$, then we have
\begin{align}\label{eq:Ppm interior local smoothing}
\norm{e^{-\nu|s|}\jp{Nx}^{\ep_d}u}_{L^2L^{2_1}}&\les a^kN^{-\frac12}\norm{e^{-\nu|s|}t^{\frac14}f}_{L^2L^2},
\\\label{eq:Ppm interior local strichartz}
\norm{e^{-\nu|s|}\jp{Nx}^{\ep_d}u}_{L^2L^{2_1}}&\les a^k\norm{\norm{e^{-\nu|s|}\chi_{[M,2M]}f}_{L^{p'}L^{p'}}}_{\ell^2_M}.
\end{align}
\end{lem}
\begin{proof}
Up to splitting $f\mid_{t>1}$ and $f\mid_{t\le1}$, it suffices to show with the time weight $e^{-\nu|s|}$ replaced by $e^{\pm\nu s}$. Then, by scaling, we may assume $N=1$.
We claim that
\begin{equation}\label{eq:L1 Linfty estimate}
\norm{\jp{x}^{\theta+\frac12+\frac{d-2}2+\ep}\P_1^\pm e^{\mp it\De}\phi}_{L^\infty(\R^d)}
\les_{\theta,\ep} \jp t^{-1-\ep}\norm{\jp x^{\theta-\frac{d-2}2+2\ep}\phi}_{L^1(\R^d)}.
\end{equation}
Once \eqref{eq:L1 Linfty estimate} is shown, by H\"older's inequality using $\jp{x}^{-\frac{d-2}2-\ep}\in L^{2_1}$, it follows that
\[
\norm{\jp{x}^{\theta+\frac12}\P_1^\pm e^{\mp it\De}\phi}_{L^{2_1}(\R^d)}
\les_{\theta,\ep} \jp t^{-1-\ep}\norm{\jp x^{\theta+3\ep}\phi}_{L^{2_{-1}}(\R^d)}.
\]
Since the convolution by $\jp t^{-1-\ep}$ is a bounded map on $L^q((0,\infty))$ for $1\le q\le\infty$ and $L^2((0,\infty);e^{\pm 2\nu s}dt)$, we obtain \eqref{eq:Ppm N stri} and \eqref{eq:Ppm N stri'}.

We show the claim \eqref{eq:L1 Linfty estimate}; by \eqref{eq:inout bound K}, it suffices to show
\begin{align}\label{eq:L1 Linfty main}
\jp x^{-\frac12}\jp y^{-\frac12}t^{-\frac12}\les\jp t^{-1-\ep }\jp x^{-\theta-\frac12-\ep}\jp y^{\theta+2\ep}&,\,\,|y|-|x|\sim t\gtrsim1
\\
\label{eq:L1 Linfty sub}
\jp x^{-\frac12}\jp y^{-\frac12}\jp{t+|x|-|y|}^{-m}\les
\jp t^{-1-\ep}\jp x^{-\theta-\frac12-\ep}\jp y^{\theta+2\ep}&,\,\,\text{otherwise}
\end{align}
for some $m=m(\theta,\ep)$. 
Equation \eqref{eq:L1 Linfty main} holds since $|y|\gtrsim\max\{|x|,t\}$.

For \eqref{eq:L1 Linfty sub}, where $|y|-|x|\sim t\gtrsim1$ fails, we have $\jp{t+|x|-|y|}\sim\jp{t+||x|-|y||}$. Thus, \eqref{eq:L1 Linfty sub} can be simplified to
\begin{equation}\label{eq:L1 Linfty sub reduced}
\jp x^{\theta+\ep}\jp{t+||x|-|y||}^{-m}\les_{\theta,\ep}\jp t^{-1-\ep}\jp y^{\theta+\frac12+2\ep}.
\end{equation}
With the choice $m=\theta+1+2\ep$, taking the product of
\[
\jp{t+||x|-|y||}^{-m}\le\jp t^{-1-\ep}\jp{|x|-|y|}^{-m+1+\ep}\le\jp t^{-1-\ep}\jp{|x|-|y|}^{-\theta-\ep}
\]
and Peetre's inequality
\[
\jp x^{\theta+\ep}\les\jp y^{\theta+\ep}\jp{|x|-|y|}^{\theta+\ep}
\]
concludes \eqref{eq:L1 Linfty sub reduced}.

We turn to showing \eqref{eq:Ppm exterior stri}. We use variables $t$ and $\tau$ to track $f$ and $u$, e.g. $f(t)$ and $u(\tau)$. We split $f$ into time intervals $\{t\les a^2\}$ and $\{t\gg a^2\}$. On the interval $\{t\les a^2\}$, where $\supp(f)$ has diameter $O(a^2)$, \eqref{eq:Ppm exterior stri} follows from \eqref{eq:ineff Ppm}.
For the region $\{t\gg a^2\}$, we observe that for any $0<\tau<t$ and $x,y\in\R^d$ such that
\[
t\gg a^2,\quad |x|\ge (a+\rho)\sqrt{\tau},\quad |y|\le a\sqrt t,
\]
we have
\[
|y|-|x|\le a\sqrt t-(a+\rho)\sqrt{\tau}\le o_{t/a^2}(1)(t-\tau),\;\;\jp{t-\tau+|x|-|y|}\gtrsim_\rho \jp{t-\tau}+a^{-1}(\jp x+\jp y).
\]
Thus, fixing a large $k$, by the second case of \eqref{eq:inout bound K}, for every $1\le\tilde r\le\infty$, we have
\[
\norm{\jp x^{\theta}\P_1^+e^{-i(t-\tau)\De}f(t)}_{L^{2_1}(|x|\ge (a+\rho)\sqrt\tau)}\les_\rho a^k\jp{t-\tau}^{-10}\norm{f(t)}_{L^{\tilde r}},
\]
giving \eqref{eq:Ppm exterior stri} for $\P_1^+$. The same proof applies for $\P_1^-$.

We now show \eqref{eq:Ppm interior local smoothing}. Let $\delta=\delta(d)>0$ be a small number to be fixed later. We split the time interval $(0,\infty)$ into $I_n=(t_n,t_{n+1}]=(\delta a^2n^2,\delta a^2(n+1)^2]$ for $n\in\N$.

We estimate the contribution of $\chi_{I_m}f$ to $\chi_{I_n}u$ for $m,n\ge0$. By \eqref{eq:ineff Ppm}, the contribution of $\chi_{I_0}f$ satisfies \eqref{eq:Ppm interior local smoothing}. Hence, we consider $m\ge1$.
In the case $m=n$, choosing $\theta>\ep_d$ satisfying \eqref{eq:weighted shao} and $\al=\theta-\ep_d$,
by \eqref{eq:weighted local smoothing} and Christ--Kiselev, we have
\begin{align}\label{eq:proof of ppm interior local smoothing, within In, eff}
&\norm{\jp x^{\ep_d}\chi_{I_n}\P_1^\pm\mc D^\mp(\chi_{I_n}f)}_{L^2L^{2_1}(|x|\ge\frac12a\sqrt{t_n})}
\\\les&(a\sqrt{t_n})^{\ep_d-\theta}\norm{\jp x^{\theta}\chi_{I_n}\P_1^\pm\mc D^\mp(\chi_{I_n}f)}_{L^2L^{2_1}(|x|\ge\frac12a\sqrt{t_n})}\nonumber
\\
\les&
a^{O(1)}(a\sqrt{t_n})^{\ep_d-\theta}\norm{t^{\frac14}\chi_{I_n}f}_{L^{\frac2{1+2\al}}L^2}
\nonumber
\\\les&
a^{O(1)}(a\sqrt{t_n})^{\ep_d-\theta}|I_n|^\al\norm{t^{\frac14}\chi_{I_n}f}_{L^2L^2}\nonumber
\les
a^{O(1)}\norm{t^{\frac14}\chi_{I_n}f}_{L^2L^2}.
\end{align}

By \eqref{eq:inout bound K}, choosing sufficiently small $\delta>0$, we also have
\begin{equation}\label{eq:proof of ppm interior local smothing, within In, ineff}
\norm{\jp x^{\ep_d} \chi_{I_n}\P_1^\pm\mc D^\mp(\chi_{I_n}f)}_{L^2L^{2_1}(|x|<\frac12a\sqrt{t_n})}\les \norm{t^{\frac14}\chi_{I_n}f}_{L^{2}L^2}.
\end{equation}
In the case $|n-m|\gg_\delta1$, since $|t_n-t_m|\gg a\sqrt{t_m}$, by \eqref{eq:inout bound K}, we have
\begin{equation}\label{eq:proof of ppm interior local smoothing, outside In}
\norm{\jp{x}^{\ep_d}\chi_{I_n}\P_1^\pm\mc D^\mp(\chi_{I_m}f)}_{L^2L^{2_1}}
\les_\delta(n+m)^{-10}\norm{t^{\frac14}\chi_{I_m}f}_{L^2L^2}.
\end{equation}
In the case $0<|n-m|\les_\delta1$, by \eqref{eq:weighted local smoothing} and \eqref{eq:weighted shao}, we have
\begin{equation}\label{eq:proof of ppm interior local smoothing, near In}
\norm{\jp{x}^{\ep_d}\chi_{I_n}\P_1^\pm\mc D^\mp(\chi_{I_m}f)}_{L^2L^{2_1}}\les a^{O(1)}\norm{t^{\frac14}\chi_{I_m} f}_{L^2L^2}.
\end{equation}
Applying the Schur test to \eqref{eq:proof of ppm interior local smoothing, within In, eff}--\eqref{eq:proof of ppm interior local smoothing, near In} gives \eqref{eq:Ppm interior local smoothing}. The same proof using the Strichartz estimate instead of \eqref{eq:weighted local smoothing} gives \eqref{eq:Ppm interior local strichartz}. This completes the proof.
\end{proof}

In adaptation to \eqref{eq:Ppm N stri}, it is natural to measure $\mathring u$ and $\tilde u$ in Besov-like spaces involving weights of the form $\jp{Nx}^\theta$, $\theta\in\R$. We introduce $L^2$-critical function spaces $X$ for $\mathring u$ and $S$ for $\tilde u$ in such forms.
To keep track of slight perturbations of Sobolev norm-exponents, we denote $\sigma_d=\frac{\ep_d}{100d}$.
For an interval $I\subset\R$, let $X(I)$ and $S(I)$ be spaces of $u\in L^\infty L^2(I\times\R^d)$ with finite norms
\begin{align*}
\norm{u}_{X(I)}=&\norm{N^{-1}\norm{\jp{Nx}^{\frac52-2\s_d}u_N}_{L^\infty L^{2_1}(I\times\R^d)}}_{\ell^\infty_N},
\\
\norm{u}_{S(I)}=&\norm{\norm{u_N}_{L^\infty L^2(I\times\R^d)}}_{\ell^2_N}+\norm{\norm{\jp{Nx}^{\ep_d}u_N}_{L^2L^{2_1}(I\times\R^d)}}_{\ell^2_N}.
\end{align*}
Note that in our argument, it suffices that the exponent in $X$ norm is $2+$. 

By \eqref{eq:weighted shao}, $e^{it\De}:L^2\rightarrow S(\R)$ is a bounded operator.
We also note an embedding property of $S(I)$: for a Strichartz pair $(q,r)\neq(2,\infty)$,
\begin{equation}\label{eq:S embed}
S(I)\hook L^\infty L^2\cap L^2\dot B^0_{2_1,2}(I\times\R^d)\hook L^qL^r(I\times\R^d).
\end{equation}
In the next lemma, we observe that $u(t)$ enjoys some spatial decay when $u\in X(I)$.
\begin{lem}\label{lem:X embed}
For an interval $I\subset\R$, $u\in X(I)$, $t\in I$, and $\theta\in[0,\frac32-3\s_d)$, we have
\begin{align}\label{eq:X embed Hs}
\norm{u(t)}_{\dot H^\theta(|x|\ge r)}\les r^{-\theta}(\norm u_{X(I)}+\norm u_{L^\infty L^2(I\times\R^d)})&,\quad r>0,
\\
\label{eq:X embed u}
|u(t,x)|\les|x|^{-\frac d2}(\norm{u}_{X(I)}+\norm u_{L^\infty L^2(I\times\R^d)})&,\quad x\in\R^d\setminus\{0\}.
\end{align}
\end{lem}
\begin{proof}
Up to scaling, we may assume $r=1$ and $|x|=1$. For $N\ge1$, since $\norm{f}_{L^2}\les N^{-1}\norm{\jp{Nx}^{1+\s_d}f}_{L^{2_1}}$ for $f\in L^2(\R^d)$, we have
\[
\norm{u_N(t)}_{\dot H^\theta(|x|\ge 1)}
\les
N^{\theta-(\frac32-3\s_d)}\norm{\jp{Nx}^{\frac32-3\s_d}u_{N}(t)}_{L^2}
\les N^{\theta-(\frac32-3\s_d)}\norm{u}_{X(I)}.
\]
Summing over $N\in2^\N$ and combining with the $L^2$-norm bound for $N\le1$ yields \eqref{eq:X embed Hs}.
By \eqref{eq:X embed Hs} and the radial Sobolev inequality, we obtain \eqref{eq:X embed u}.
\end{proof}
We consider the Fourier operator $N\d_{x_j}\De^{-1}P_{\sim N}$ for $j=1,\ldots,d$ and $N\in2^\Z$. We denote $\phi^{j,N}=\F^{-1}(-iN\xi_j|\xi|^{-2}\psi^{\sim N})$ so that $\phi^{j,N}\ast f=N\d_{x_j}\De^{-1}P_{\sim N}f$.

In the next lemma, we show the main goal of this section: a solution $u$ to \eqref{eq:NLS} on an interval $I$ can be decomposed into a regular part $\mathring u\in X(I)$ and a radiative part $\tilde u\in S(I)$.
To handle the non-algebraic nonlinearity $\mc N(u)$, we use Bony linearizations.
We use the notation of \cite{LWP(working)}: for $v\in\C$, we denote
\[
\mc N(v)=-|v|^{\frac 4d}v,\quad A(v)=(A_1(v),A_2(v))=(\d_z\mc N(v),\d_{\bar z}\mc N(v)).
\]

For $w\in\C$ and $A=(A_1,A_2)\in\C\times\C$, we denote
\[
w\times A=wA_1+\ol w A_2.
\]

With these notations, we have the expression
\[
\mc N(v+w)-\mc N(v)=w\times\textstyle\int_0^1A((1-\theta)(v+w)+\theta v)d\theta.
\]

\begin{lem}\label{lem:regularity}
Let $d\ge2$ and $M>0$. Let $u$ be a radial solution to \eqref{eq:NLS} on an interval $I\subset(0,\infty)$ such that $M(u)\le M$. Denote $\ul I=\{s\in\R: e^s=t\in I\}$. Then, there exists a decomposition $\mathring u+\tilde u=u$ of $\mathring u\in X(I)$ and $\tilde u\in S(I)$ such that
\begin{align}\label{eq:S<L and X<L}
\norm{\mathring u}_{L^\infty L^2(I\times\R^d)}+\norm{\mathring u}_{X(I)}+\norm{\tilde u}_{S(I)}&\les_M1,
\\
\label{eq:local smoothing of tilde u}
\norm{\jp y^{-1}P_{\ge1}\ul{\tilde u}}_{L^2H^{\frac12}(\ul I\times\R^d)}&\les_{M} 1.
\end{align}
\end{lem}
\begin{proof}
We first show this lemma for the case where $I=[\tau_-,\tau_+]$ is a compact interval. In this proof, we omit the time interval $I$ for function spaces. 

Given a decomposition $v+w=u$ of $v\in X$ and $w\in S$, we define an updated decomposition $v_*+w_*=u$.
Hereafter, all summations are over $N\in2^\Z$. We use the Einstein summation convention for the index $j$ and the notation $\pm$ denotes summation over both signs, e.g., $f_\pm=f_++f_-$.
First, we decompose $P_Nu$ into 
\begin{equation}\label{eq:decomp u}
P_Nu=P_N(\psi^{<R/N}u)+P_N(\psi^{\ge R/N}u),
\end{equation}
where $R=R(M)$ is a number to be chosen later. Denoting
\[
f^N\coloneqq P_{\sim N}\bb{2\na\psi^{\ge R/N}\cdot\na u+(\De\psi^{\ge R/N})u},
\]
we have the equation for $P_N(\psi^{\ge R/N}u)$,
\[
(i\d_t+\De)P_N(\psi^{\ge R/N}u)=P_N(f^N+\psi^{\ge R/N}\mc N(u)).
\]
We expand $P_N(\psi^{\ge R/N}u)=\P^\pm_N(\psi^{\ge R/N}u)$ in Duhamel form
\begin{equation}\label{eq:decomp PN(>u)}
P_N(\psi^{\ge R/N}u)=\P^\pm_Ne^{i(t-\tau_\pm)\De}(\psi^{\ge R/N}u(\tau_\pm))
+\P^\pm_N\mc D^\mp\bb{f^N+\psi^{\ge R/N}\mc N(u)}.
\end{equation}
Using the identity $\P_N^\pm=\P_N^\pm (\phi^{j,N}*N^{-1}\d_{x_j})$, we decompose
\begin{align}\label{eq:decomp Nu}
\P_N^\pm\mc D^\mp\bb{\psi^{\ge R/N}\mc N(u)}=&\P_N^\pm\mc D^\mp\bb{\psi^{\ge R/N}(\mc N(u)-\mc N(u_{<N}))}
\\&+\P^\pm_N\mc D^\mp\bb{\phi^{j,N}*N^{-1}\d_{x_j}(\psi^{\ge R/N}\mc N(u_{<N}))}.\nonumber
\end{align}
We further decompose each summand in \eqref{eq:decomp Nu} as
\begin{align}\label{eq:g 1,N def}
\psi^{\ge R/N}(\mc N(u)-\mc N(u_{<N}))=&\psi^{\ge R/N}w_{\ge N}\times\textstyle\int_0^1A((1-\theta)u+\theta u_{<N})d\theta
\\&+\psi^{\ge R/N}v_{\ge N}\times\textstyle\int_0^1A((1-\theta)u+\theta u_{<N})d\theta\nonumber
\\=&g^{1,N}+h^{1,N},\nonumber
\\
N^{-1}\d_{x_j}(\psi^{\ge R/N}\mc N(u_{<N}))=&N^{-1}\psi^{\ge R/N}\d_{x_j} w_{<N}\times A(u_{<N})\nonumber
\\&+N^{-1}\psi^{\ge R/N}\d_{x_j}v_{<N}\times A(u_{<N})\nonumber
\\&+N^{-1}\d_{x_j}\psi^{\ge R/N}\mc N(u_{<N})=g^{2,N}+h^{2,N}+l^{N}.\nonumber
\end{align}
Let $w_*$ be the sum of the linear radiation and the $w$-contributed nonlinear term
\[
w_*=\sum_N \P^\pm_Ne^{i(t-\tau_\pm)\De}\bb{\psi^{\ge R/N}u(\tau_\pm)}
+\sum_N\P^\pm_N\mc D^\mp\bb{g^{1,N}+\phi^{j,N}*g^{2,N}}
=w_*^0+w_*^1
\]
and $v_*$ be the remainder
\[
v_*=\sum_N P_N\bb{\psi^{<R/N}u}
+\sum_N\P^\pm_N \mc D^\mp\bb{f^N+h^{1,N}+\phi^{j,N}*(h^{2,N}+l^{N})}
=v_*^0+v_*^1.
\]
Unfolding the expansions \eqref{eq:decomp u}--\eqref{eq:g 1,N def} yields $v_*+w_*=u$.

We estimate $v_*$ and $w_*$. For $N\in2^\Z$ and $\theta\in[0,1]$, we have
\begin{equation}\label{eq:A bound}
\norm{A((1-\theta)u+\theta u_{<N})}_{L^\infty L^{\frac d2}}\les\norm{u}_{L^\infty L^2}^{\frac4d}\les1.
\end{equation}
Since $0<\ep_d<1$ and $0<1-\frac{\s_d}2<1$, by \eqref{eq:homogeneous embedding Lp, >N} and \eqref{eq:homogeneous embedding Lp, d>N}, we have
\begin{align}\label{eq:weighted estimate on w}
&\norm{\norm{|Nx|^{\ep_d}w_{\ge N}}_{L^2L^{2_1}}}_{\ell^2_N}+\norm{N^{-1}\norm{|Nx|^{\ep_d}\na w_{<N}}_{L^2L^{2_1}}}_{\ell^2_N}
\\\les&\norm{\norm{\jp{Nx}^{\ep_d}w_N}_{L^2L^{2_1}}}_{\ell^2_N}\les\norm{w}_{S},\nonumber
\\&\norm{N^{-1}\norm{|Nx|^{2-\frac{\s_d}2}v_{\ge N}}_{L^\infty L^{2_1}}}_{\ell^\infty_N}+\norm{N^{-2}\norm{|Nx|^{2-\frac{\s_d}2}\na v_{<N}}_{L^\infty L^{2_1}}}_{\ell^\infty_N}\nonumber
\\\les&\norm{N^{-1}\norm{\jp{Nx}^{2-\frac{\s_d}2}v_N}_{L^\infty L^{2_1}}}_{\ell^\infty_N}\les\norm {v}_X.\nonumber
\end{align}
Taking products with \eqref{eq:A bound}, we obtain bounds on $g^{k,N}$ and $h^{k,N}$ in \eqref{eq:g 1,N def} as
\begin{align}\label{eq:g bound prep}
\norm{\norm{|Nx|^{\ep_d}g^{k,N}}_{L^2L^{2_{-1}}}}_{\ell^2_N}&\les\norm w_{S},
\\\label{eq:h bound prep}
\norm{N^{-1}\norm{|Nx|^{2-\frac{\s_d}2}h^{k,N}}_{L^\infty L^{2_{-1}}}}_{\ell^\infty_N}&\les\norm v_X.
\end{align}
Since $g^{k,N}$ and $h^{k,N}$ are supported in $\{|x|\gtrsim R/N\}$, \eqref{eq:g bound prep} and \eqref{eq:h bound prep} yield
\begin{align}\label{eq:g bound}
\norm{\norm{\jp{Nx}^{\ep_d-{\s_d}}g^{k,N}}_{L^2L^{2_{-1}}}}_{\ell^2_N}&\les R^{-{\s_d}}\norm w_{S},
\\
\label{eq:h bound}
\norm{N^{-1}\norm{\jp{Nx}^{2-\frac32\s_d}h^{k,N}}_{L^\infty L^{2_{-1}}}}_{\ell^\infty_N}&\les R^{-\s_d}\norm v_X.
\end{align}
Since the cutoffs $\na\psi^{\ge R/N}$ and $\De\psi^{\ge R/N}$ are supported in $\{|x|\le R/N\}$, we have
\begin{equation}\label{eq:f bound}
\norm{N^{-1}\norm{\jp{Nx}^{2} f^N}_{L^\infty L^{2_{-1}}}}_{\ell^\infty_N}+\norm{N^{-1}\norm{\jp{Nx}^{2} l^{N}}_{L^\infty L^{2_{-1}}}}_{\ell^\infty_N}\les_{R}1.
\end{equation}
By \eqref{eq:PN l2 decouple}, \eqref{eq:Ppm is Aut}, and \eqref{eq:weighted shao}, we have
\begin{equation}\label{eq:11}
\norm{w_*^0}_{S}\les_R\norm{\norm{P_N(\psi^{\ge R/N}u(\tau_\pm))}_{L^2(\R^d)}}_{\ell^2_N}\les1.
\end{equation}
By the Strichartz estimate and \eqref{eq:g bound}, we have
\begin{equation}\label{eq:21 Stri}
\norm{\norm{P_Nw_*^1}_{L^\infty L^2}}_{\ell^2_N}\les R^{-\s_d}\norm w_{S}.
\end{equation}
Applying \eqref{eq:Ppm N stri} with parameters $(q,\theta)=(2,0)$ to \eqref{eq:g bound}, we obtain
\begin{equation}\label{eq:21 Besov}
\norm{\norm{\jp{Nx}^{\frac12}P_Nw_*^1}_{L^2L^{2_1}}}_{\ell^2_N}\les R^{-\s_d}\norm w_{S}.
\end{equation}
Similarly, applying \eqref{eq:Ppm N stri} with $(q,\theta)=(\infty,2-2\s_d)$ to \eqref{eq:h bound} and \eqref{eq:f bound} yields
\begin{equation}\label{eq:X bound}
\norm{v_*^1}_X\les R^{-\s_d}\norm{v}_X+O_R(1).
\end{equation}
For $v_*^0$, since $\supp(\psi^{<R/N})\subset\{|x|\le R/N\}$, we have
\begin{equation}\label{eq:12}
\norm{v_*^0}_X\les\norm{N^{-1}\norm{\jp{Nx}^{\frac 52-2\s_d}P_N(\psi^{<R/N}u)}_{L^\infty L^{2_1}}}_{\ell^\infty_N}\les_{R}\norm{u}_{L^\infty L^2}\les1.
\end{equation}
Taking the sum of \eqref{eq:11}, \eqref{eq:21 Stri}, \eqref{eq:21 Besov}, \eqref{eq:X bound}, and \eqref{eq:12}, we have
\begin{align}\label{eq:regularity bootstrap w}
\norm{w_*}_S&\le\norm{w_*^0}_S+\norm{w_*^1}_S\les O_R(1)+R^{-\s_d}\norm w_S,
\\
\label{eq:regularity bootstrap v}
\norm{v_*}_X&\le\norm{v_*^0}_X+\norm{v_*^1}_X\les O_R(1)+R^{-\s_d}\norm v_X.
\end{align}

We find $\mathring u$ and $\tilde u$. Let $\{v_n\}\subset X$ and $\{w_n\}\subset S$ be sequences defined inductively as
\begin{equation}\label{eq:seq def}
(v_0,w_0)\coloneqq(0,u)\text{ and }(v_{n+1},w_{n+1})\coloneqq((v_n)_*,(w_n)_*).
\end{equation}
Here we have $\norm{w_0}_S=O_{u,I}(1)$ by standard local theory. 
By \eqref{eq:regularity bootstrap w} and \eqref{eq:regularity bootstrap v}, fixing sufficiently large $R=R(M)\gg1$, we have
\begin{equation}\label{eq:regularity a priori}
\norm{v_n}_X,\norm{w_n}_S=O_M(1),\quad n\gg1.
\end{equation}
Thus, passing to a subsequence of $n$, $v_n$ converges weakly in $L^2L^2$ to some $\mathring u$. Then, $w_n$ converges weakly to $\tilde u\coloneqq u-\mathring u$. By \eqref{eq:regularity a priori} and $S\hook L^\infty L^2$, \eqref{eq:S<L and X<L} follows.

We turn to showing \eqref{eq:local smoothing of tilde u}. By the local smoothing effect \eqref{eq:weighted local smoothing}, for $n\in\N$, we have
\begin{equation}\label{eq:prep for local smoothing tilde w0}
\norm{N^{\frac12}\norm{t^{-\frac14}\jp{x/\sqrt t}^{-1}P_N(w_n)_*^0}_{L^2L^2}}_{\ell^2_N}\les\norm{\norm{P_N(\psi^{\ge R/N}u(\tau_\pm))}_{L^2}}_{\ell^2_N}\les1.
\end{equation}
Under the self-similar renormalization, \eqref{eq:prep for local smoothing tilde w0} can be rewritten as
\begin{equation}\label{eq:prep for local smoothing tilde w0 sc}
\norm{\norm{N^{\frac12}\norm{\jp y^{-1}P_N\ul{(w_n)_*^0}}_{L^2_y}}_{\ell^2_N}}_{L^2_s}\les1.
\end{equation}
Since $\jp y^{-1}|y|^{-\frac12}\in L^d(\R^d)$, by H\"older's inequality, we have
\begin{equation}\label{eq:prep for local smoothing tilde w1 sc}
\norm{\norm{N^{\frac12}\norm{\jp y^{-1}P_N\ul{(w_n)_*^1}}_{L^2_y}}_{\ell^2_N}}_{L^2_s}\les\norm{\norm{\norm{|Ny|^{\frac12}P_N\ul{(w_n)_*^1}}_{L^{2_1}_y}}_{\ell^2_N}}_{L^2_s},
\end{equation}
which is $O_M(1)$ for $n\gg_u1$ due to \eqref{eq:21 Besov} and \eqref{eq:regularity a priori}.
Applying \eqref{eq:weighted LP} to the sum of \eqref{eq:prep for local smoothing tilde w0 sc} and \eqref{eq:prep for local smoothing tilde w1 sc} yields \eqref{eq:local smoothing of tilde u} for $\tilde u_n$. Passing to a limit, the same holds for $\tilde u$.

Finally, we extend the result to a non-compact interval $I\subset(0,\infty)$. Let $\{I_m\}_{m\in\N}$ be a sequence of compact intervals with $I_m\uparrow I$. We decompose $u=\mathring u_m+\tilde u_m$ as in this lemma on each interval $I_m$. Then, passing to a limit completes the proof.
\end{proof}
Hereafter, we often use Sobolev exponents of the form $\frac14+\frac k2$, $k\in\Z$. The choice of the number $\frac14$ is made purely for convenience and can be replaced by any $\al\in(0,\frac12)$.

For $a\ge1$, we fix a radial function $\chi_a\in C^3(\R^d)$ such that $\chi_a(x)=0$ for $|x|\le a-1$, $\chi_a(x)=1$ for $|x|\ge a$, and $\norm{\chi_a}_{C^3}\le10d$. We also denote $\varphi_a(t,x)=\chi_a(x/\sqrt t)$.

In the next lemma, we provide a variant of Lemma~\ref{lem:regularity}, which plays a key role in detecting the sharp coefficient $\sqrt{2\pi}$ in Theorem \ref{thm:loglog}. We decompose $\varphi_au=\mathring u_a+\tilde u_a$ so that high regularity of $\ul{\mathring u_a}$ in the exterior region $\{|y|\ge a+2\iota\}$ is controlled by the local smoothing norm of $\ul u$ in the intermediate region $\{a-\iota\le|y|\le a\}$.
\begin{lem}\label{lem:regularity sec6}
Let $d\ge2$, $M>0$, and $\iota>0$. Let $u$ be a radial solution to \eqref{eq:NLS} on $(0,1]$ such that $M(u)\le M$. For $\tau\ll_{u,\iota}1$ and $a\gg_{M,\iota}1$, we have the following:

Denote $u_a=\varphi_a u$.
Then, there exists a decomposition $\mathring u_a+\tilde u_a=u_a$ on $I=(0,\tau]$ such that for $s_*<0$ and $0\le\nu\ll_d1$,
\begin{align}\label{eq:S<L and X<L sec6}
&\norm{e^{-\nu|s-s_*|}\na\ul{\mathring u_a}}_{L^2H^{\frac14}(\ul I\times\{|y|\ge a+2\iota\})}
\les a^{O(1)}\norm{e^{-\nu|s-s_*|}\ul{u}}_{L^2H^{\frac12}(\ul I\times\{a-\iota\le|y|\le a\})},\hspace{-1mm}
\\
\label{eq:local smoothing of tilde u sec6}
&\norm{\tilde u_a}_{S(I)}+\norm{\jp y^{-1}P_{\ge1}\ul{\tilde u_a}}_{L^2H^{\frac12}(\ul I\times\R^d)}\les_{u} 1.
\end{align}
\end{lem}
\begin{proof}
For simplicity, we assume $\iota=1$; the proof for $\iota>0$ is similar. Denote $\mu=e^{-\nu|s-s_*|}$. 
We mostly follow the proof of Lemma~\ref{lem:regularity}. As earlier, we consider a compact subinterval $[\tau_-,\tau_+]\subset(0,\tau]$.
We begin by providing the map $(v,w)\mapsto(v_*,w_*)$, which defines the sequence $\{(v_n,w_n)\}$ by \eqref{eq:seq def} and the decomposition $\mathring u_a+\tilde u_a=u_a$ as its limit.
Let
\[
f\coloneqq(i\d_t +\De)u_a-\mc N(u_a)=\bb{2\na\varphi_a\cdot\na u+((i\d_t+\De)\varphi_a) u}+(\varphi_a-\varphi_a^{1+4/d})\mc N(u).
\]
Denoting $g^{k,N}$ and $h^{k,N}$ identically to \eqref{eq:g 1,N def} but without $\psi^{\ge R/N}$, we decompose $u_a=v_*+w_*$ as
\begin{align*}
w_*&=\sum_N\P_N^\pm e^{i(t-\tau_\pm)\De}\bb{u_a(\tau_\pm)}+\sum_N\P_N^\pm\mc D^\mp\bb{g^{1,N}+\phi^{j,N}*g^{2,N}}&= w_*^0+w_*^1,
\\
v_*&=\sum_N\P_N^\pm\mc D^\mp f+\sum_N\P_N^\pm\mc D^\mp\bb{h^{1,N}+\phi^{j,N}*h^{2,N}}&= v_*^0+v_*^1.
\end{align*}
We keep measuring $w$ in $S$. To measure $v$, we use $L^2_t$-based norms
\begin{align*}
\norm{v}_{Y}\coloneqq&\norm{\norm{\mu\jp{Nx}^{\ep_d} P_N v}_{L^2L^{2_1}}}_{\ell^2_N},
\\
\norm{v}_{Y^\al}\coloneqq&\norm{\norm{\mu\jp{Nx}^\al P_N v}_{L^2L^{2_1}(|x|\ge (a+\al)\sqrt t)}}_{\ell^2_N},\quad\al>0.
\end{align*}

We show versions of the bootstrapping bounds \eqref{eq:regularity bootstrap w}--\eqref{eq:regularity bootstrap v}: for $R\ge1$,
\begin{align}
\label{eq:regularity bootstrap w sec6}
\norm{w_*}_S&\les O_R(1)+(R^{-\s_d}+R^{O(1)}(o_a(1)+o_\tau(1)))\norm w_S,
\\
\label{eq:regularity bootstrap v sec6, 0}
\norm{v_*}_{Y}&\les a^{O(1)}\norm{\mu\ul u}_{L^2H^{\frac12}(a-1\le|y|\le a)}+(R^{-\s_d}+R^{O(1)}(o_a(1)+o_\tau(1)))\norm{v}_{Y}.
\end{align}

Decomposing $u=\mathring u+\tilde u$ as in Lemma~\ref{lem:regularity}, we have
\begin{align}\label{eq:u_a bdd}
&\norm{u_a}_{L^{p,\infty}L^p}+\norm{\norm{\chi_{[M,2M]}u}_{L^p L^p(|x|\ge(a-1)\sqrt t)}}_{\ell^\infty_M}
\\\les&\norm{\ul{\mathring u}}_{L^\infty L^p(\ul I\times\{|y|\ge a-1\})}+\norm{\tilde u}_{L^pL^p((0,\tau]\times\R^d)}=o_a(1)+o_\tau(1).\nonumber
\end{align}
By \eqref{eq:weighted estimate on w} and \eqref{eq:u_a bdd}, for $k=1,2$, we have
\begin{align}
\label{eq:ineff part sec6}
\norm{\norm{|Nx|^{\ep_d}\psi^{<R/N}g^{k,N}}_{L^{p',\infty}L^{p'}}}_{\ell^2_N}
&\les\norm{u_a}_{L^{p,\infty}L^p}^{\frac 2d}\norm{u_a}_{L^\infty L^2}^{\frac2d}\norm{w}_S
\\&\les(o_a(1)+o_\tau(1))\norm{w}_S.\nonumber
\end{align}
Recall that in \eqref{eq:g bound}, we showed
\begin{equation}\label{eq:g bound prep sec6}
\norm{\norm{\jp{Nx}^{\ep_d-\s_d}\psi^{\ge R/N}g^{k,N}}_{L^2L^{2_{-1}}}}_{\ell^2_N}\les R^{-\s_d}\norm w_S.
\end{equation}
Applying \eqref{eq:ineff Ppm} to \eqref{eq:ineff part sec6}, \eqref{eq:Ppm N stri} to \eqref{eq:g bound prep sec6}, and taking a summation, we obtain
\begin{equation}\label{eq:w1 bound sec6}
\norm{w_*^1}_S\les(R^{-\s_d}+R^{O(1)}(o_a(1)+o_\tau(1)))\norm{w}_S.
\end{equation}
Since $e^{it\De}:L^2\to S$ is bounded, we also have $\norm{w_*^0}_S\les1$. Hence, we obtain \eqref{eq:regularity bootstrap w sec6}.

We turn to showing \eqref{eq:regularity bootstrap v sec6, 0}. 
We first estimate $v_*^0$; note that
\begin{equation}
\label{eq:f du part}
\norm{\mu t^{\frac14}(2\na\varphi_a\cdot\na u+((i\d_t+\De)\varphi_a) u)}_{L^2\dot H^{-\frac12}}\les a^{O(1)}\norm{\mu\ul u}_{L^2 H^{\frac12}(a-1\le|y|\le a)}.
\end{equation}
On $\{a-1\le|y|\le a\}$, by the radial Sobolev inequality and $\frac12-\frac1p\le\frac d{2d+4}$, we have
\[
\norm{\mc N(\ul u)}_{L^{p'}}=\norm{\ul u}_{L^{p}}^{1+\frac4d}\les\norm{\ul u}_{L^p}^{\frac2d}\norm{\ul u}_{H^{\frac d{2d+4}}}^{1+\frac2d}\les\norm{\ul u}_{L^p}^{\frac2d}\norm{\ul u}_{L^2}^{\frac2d}\norm{\ul u}_{H^{\frac12}}.
\]

Hence, by \eqref{eq:u_a bdd} and $\supp(\varphi_a-\varphi_a^{1+4/d})\subset\{(a-1)\sqrt t\le|x|\le a\sqrt t\}$, we have
\begin{align}
\label{eq:f Nu part Lorentz}
\norm{\mu(\varphi_a-\varphi_a^{1+4/d})\mc N(u)}_{L^{p',\infty}L^{p'}}&\les\norm{\mu\ul u}_{L^2H^{\frac12}(a-1\le|y|\le a)}.
\\
\label{eq:f Nu part}
\norm{\norm{\mu\chi_{[M,2M]}(\varphi_a-\varphi_a^{1+4/d})\mc N(u)}_{L^{p'}L^{p'}}}_{\ell^2_M}&\les\norm{\mu\ul u}_{L^2H^{\frac12}(a-1\le|y|\le a)}.
\end{align}
Applying \eqref{eq:Ppm interior local smoothing} to \eqref{eq:f du part} and \eqref{eq:Ppm interior local strichartz} to \eqref{eq:f Nu part}, we obtain
\begin{equation}\label{eq:v0 bound sec6}
\norm{v_*^0}_{Y}\les a^{O(1)}\norm{\mu\ul u}_{L^2H^{\frac12}(a-1\le|y|\le a)}.
\end{equation}

Similarly to \eqref{eq:w1 bound sec6}, we have the bound for $v_*^1$
\[
\norm{v_*^1}_Y\les(R^{-\s_d}+R^{O(1)}(o_a(1)+o_\tau(1)))\norm{v}_{Y}.
\]
Combining with \eqref{eq:v0 bound sec6}, we obtain \eqref{eq:regularity bootstrap v sec6, 0}.

We now show a regularity-improvement estimate: for $0<\al<1$ and $0<\rho<\frac12$,
\begin{equation}\label{eq:regularity bootstrap v sec6, al}
\norm{v_*}_{Y^{\al+\rho}}\les_{\al,\rho} a^{O(1)}\bb{\norm{\mu\ul u}_{L^2H^{\frac12}(a-1\le|y|\le a)}+\norm v_Y+\norm{v}_{Y^{\al}}}.
\end{equation}

Applying \eqref{eq:Ppm exterior stri} to \eqref{eq:f du part} and \eqref{eq:f Nu part Lorentz}, we obtain
\begin{equation}\label{eq:v0 bound sec6, al}
\norm{v_*^0}_{Y^{\al+\rho}}\les_\al a^{O(1)}\norm{\mu\ul u}_{L^2H^{\frac12}(a-1\le|y|\le a)}.
\end{equation}

Similarly to \eqref{eq:ineff part sec6}, for $k=1,2$, we have
\begin{equation}\label{eq:ineff part for v sec6}
\norm{\norm{\mu|Nx|^{\ep_d}\psi^{<R/N}h^{k,N}}_{L^{p',\infty}L^{p'}}}_{\ell^2_N}\les(o_a(1)+o_\tau(1))\norm{v}_Y.
\end{equation}
Since $0<\ep_d<1$ and $0<\al<1$, we also have analogues of \eqref{eq:g bound prep},
\begin{align}
\norm{\norm{\mu|Nx|^{\ep_d}\psi^{\ge R/N}h^{k,N}\one_{|x|<(a+\al)\sqrt t}}_{L^2L^{2_{-1}}}}_{\ell^2_N}&\les \norm{v}_{Y},
\label{eq:g2 sec6}
\\
\norm{\norm{\mu|Nx|^{\al}\psi^{\ge R/N}h^{k,N}\one_{|x|\ge(a+\al)\sqrt t}}_{L^2L^{2_{-1}}}}_{\ell^2_N}&\les \norm{v}_{Y}+\norm v_{Y^{\al}}.
\label{eq:g3 sec6}
\end{align}
Applying \eqref{eq:ineff Ppm} to \eqref{eq:ineff part for v sec6}, \eqref{eq:Ppm exterior stri} to \eqref{eq:g2 sec6}, and \eqref{eq:Ppm N stri'} to \eqref{eq:g3 sec6}, we obtain
\begin{equation}\label{eq:v1 Y}
\norm{v_*^1}_{Y^{\al+\rho}}\les a^{O(1)}(\norm v_Y+\norm v_{Y^\al}).
\end{equation}
Combining with \eqref{eq:v0 bound sec6, al}, we obtain \eqref{eq:regularity bootstrap v sec6, al} as desired.

We begin the main proof. For $R=R(M)\gg1$, $a\gg_R1$, and $\tau\ll_R1$, repeating the proof of Lemma~\ref{lem:regularity} with \eqref{eq:regularity bootstrap w sec6}--\eqref{eq:regularity bootstrap v sec6, 0} yields \eqref{eq:local smoothing of tilde u sec6} and
\begin{equation}\label{eq:Y norm final}
\norm{ v_n}_{Y}\les a^{O(1)}\norm{\mu\ul u}_{L^2H^{\frac12}(a-1\le|y|\le a)},\quad n\gg1.
\end{equation}
Since $Y\hook Y^{\ep_d}$, we can iteratively apply \eqref{eq:regularity bootstrap v sec6, al} to obtain
\begin{equation}\label{eq:Y al norm final}
\norm{v_n}_{Y^\al}\les a^{O(1)}\norm{\mu\ul u}_{L^2H^{\frac12}(a-1\le|y|\le a)},\quad n\gg1,\quad \al=\ep_d,\tfrac12,\tfrac78,\tfrac54.
\end{equation}
Passing to a limit of $n\to\infty$, \eqref{eq:Y norm final} and \eqref{eq:Y al norm final} hold also on $\mathring u_a$. Now observing
\begin{align*}
\norm{N^{\frac54}\norm{P_N\ul{\mathring u_a}}_{L^2(|y|\ge a+\frac54)}}_{\ell^2_{N\ge1}}&\les\norm{\norm{\jp{Ny}^{\frac54}P_N\ul{\mathring u_a}}_{L^{2_1}(|y|\ge a+\frac54)}}_{\ell^2_N},
\\
\norm{N\norm{P_N\ul{\mathring u_a}}_{L^2}}_{\ell^2_{N\le1}}
&\les \norm{\norm{\jp{Ny}^{\frac54}P_N\ul{\mathring u_a}}_{L^{2_1}(|y|\ge a+\frac54)}+a^{\frac54}\norm{P_N\ul{\mathring u_a}}_{L^{2_1}}}_{\ell^2_{N\le1}},
\end{align*}
we obtain \eqref{eq:S<L and X<L sec6}. This completes the proof.
\end{proof}
\section{Quantification of the mass radiation}\label{sec:main}
In this section, we establish the key quantitative mechanism underlying Theorem~\ref{thm:loglog}, linking concentration in the self-similar region to the size of the radiative component.
Let $u$ be a solution to \eqref{eq:NLS} defined on $(0,1]$ that blows up at time $t=0$.
We quantify the following philosophy.

``If the radiative component of $\ul u$ is small around time $s_*$, then the local mass of $\ul u$ in the self-similar regime $\{|y|\sim1\}$ is also small near $s_*$.''

Let $u^\pc$ be defined in \eqref{eq:u pc def}. Decomposing $u=\mathring u+\tilde u$ and $u^\pc=\mathring u^\pc+\tilde u^\pc$ as in Lemma~\ref{lem:regularity}, we measure the radiation in terms of $\tilde u$ and $\tilde u^\pc$.
Note that for a Strichartz pair $(q,r)$ such that $q>2$, by \eqref{eq:S embed}, \eqref{eq:local smoothing of tilde u}, and \eqref{eq:X embed Hs}, we have
\begin{align}
\label{eq:tilde u embed Lp cap H1/2}
&\norm{\tilde u}_{L^qL^r((0,1]\times\R^d)}+\norm{\tilde u^\pc}_{L^qL^r([1,\infty)\times\R^d)}+\norm{\jp y^{-1}P_{\ge1}\ul{\tilde u}}_{L^2H^{\frac12}((-\infty,0]\times\R^d)}\les1
\\
\label{eq:mathring u embed Hs}
&\norm{\na\ul{\mathring u}(s)}_{H^{\frac14}(|y|\ge a)},\,\norm{\na\ul{\mathring u^\pc}(s)}_{H^{\frac14}(|y|\ge a)}\les a^{-1}+a^{-\frac54},\quad s\le 0,\quad a>0.
\end{align}
The following is our main proposition in its primary form:
\begin{prop}\label{prop:main}
Let $(q,r)$ be a Strichartz pair with $2<q\le p$. There exists $C_1>0$ such that the following holds:

Let $M>0$, $R_0\ge1$, $\rho>0$, and $\delta\ll_{M,R_0,\rho}1$. Denote $C=C_1|\log\delta|$. Let $\kappa\ll_{u,\delta}1$ and $s_*\le-\kappa^{-10}$ be parameters such that, on the interval $[s_*-\kappa^{-10},s_*+\kappa^{-10}]$,
\begin{equation}\label{eq:eta def}
\norm{\ul{\tilde u}}_{L^qL^r}+\norm{\ul{\tilde u^\pc}}_{L^qL^r}
+\norm{\jp y^{-1}P_{\ge1}\ul{\tilde u}}_{L^2H^{\frac12}}
+\norm{P_{e^{-\kappa^{-2}}\le\cdot \le e^{\kappa^{-2}}}\ul{\tilde u}}_{L^\infty L^2}\le e^{-C^2\kappa^{-1}}.\hspace{-2mm}
\end{equation}
Then, we have
\begin{equation}\label{eq:main}
\norm{y\ul u}_{L^2L^2([s_*-1,s_*+1]\times\{\sqrt\kappa\le|y|\le R_0\})}\le\rho\delta\sqrt\kappa.
\end{equation}
\end{prop}

The remainder of this section is devoted to showing Proposition~\ref{prop:main} and its variants.

Let $\al=\al(d)>0$ be a constant to be fixed later. We denote 
\[
\eta=e^{-C^2\kappa^{-1}},\quad R=\eta^{-\al}.
\]
Then, the relative sizes of the parameters $\eta$, $R$, and $\kappa$ are ordered as
\[\eta\,\ll\,  R^{-1}\ll \, \kappa.\]
In particular, any power of $\kappa^{-1}$ is negligible compared to $R$ and $\eta^{-1}$, so we suppress such factors by writing $\kappa^{-O(1)}$.

First, we provide an estimate of $\mathring u$: on the interval $[s_*-\kappa^{-10},s_*+\kappa^{-10}]$,\footnote{Indeed, this decay bound is the only estimate for which we employ $\mathring u^\pc$ and $\tilde u^\pc$.}
\begin{equation}\label{eq:mathring u decay}
\norm{\ul{\mathring u}}_{L^2H^{\frac54}(\kappa/4\le|y|\le 4R)}+\norm{|y|^{\frac34}\ul{\mathring u}}_{L^2L^2(\kappa/4\le|y|\le 4R)}\les\kappa^{-O(1)}.
\end{equation}
For the proof of \eqref{eq:mathring u decay}, we use the identity
\begin{equation}\label{eq:pc trick identity}
-\tfrac i2ye^{\frac i4|y|^2}\oul{\mathring u}=(e^{\frac i4|y|^2}\na\oul{\mathring u}-\na\ul{\mathring u^\pc})
+\na(\ul{\mathring u^\pc}-e^{\frac i4|y|^2}\oul{\mathring u})
\end{equation}
on the annulus $\Omega=\{\kappa/4\le|y|\le 4R\}$.
We have
\begin{align}\label{eq:mathring u decay - I bound}
&\norm{\na(\ul{\mathring u^\pc}-e^{\frac i4|y|^2}\oul{\mathring u})}_{L^2H^{\frac14}(\Omega)}
\\\les& R^{O(1)}(\norm{\na\ul{\mathring u}}_{L^2H^{\frac14}(\Omega)}+\norm{\na\ul{\mathring u^\pc}}_{L^2H^{\frac14}(\Omega)}+\norm{\ul{\mathring u}}_{L^2L^2(\Omega)}).\nonumber
\end{align}
By the identity \eqref{eq:ul u^pc} and the $L^qL^r$-norm bounds in \eqref{eq:eta def}, we have
\begin{equation}\label{eq:mathring u decay - II bound}
\norm{\ul{\mathring u^\pc}-e^{\frac i4|y|^2}\oul{\mathring u}}_{L^2L^2(\Omega)}=\norm{-\ul{\tilde u^\pc}+e^{\frac i4|y|^2}\oul{\tilde u}}_{L^2L^2(\Omega)}\les R^{O(1)}\eta.
\end{equation}

Interpolating between \eqref{eq:mathring u decay - I bound} and \eqref{eq:mathring u decay - II bound} yields a bootstrapping bound
\begin{align*}
\norm{y\ul{\mathring u}}_{L^2L^2(\Omega)}\les&\norm{\na\ul{\mathring u}}_{L^2L^2(\Omega)}+\norm{\na\ul{\mathring u^\pc}}_{L^2L^2(\Omega)}
\\&+R^{A}\eta^{\frac15}(\norm{\na\ul{\mathring u}}_{L^2H^{\frac14}(\Omega)}+\norm{\na\ul{\mathring u^\pc}}_{L^2H^{\frac14}(\Omega)}+\norm{y\ul{\mathring u}}_{L^2L^2(\Omega)})^{\frac45}
\end{align*}
with an exponent $A=O(1)$, which leads to
\begin{equation}\label{eq:mathring u decay bootstrap}
\norm{y\ul{\mathring u}}_{L^2L^2(\Omega)}\les\norm{\na\ul{\mathring u}}_{L^2H^{\frac14}(\Omega)}+\norm{\na\ul{\mathring u^\pc}}_{L^2H^{\frac14}(\Omega)}+R^{5A}\eta.
\end{equation}
Choosing $\al\ll1$, by \eqref{eq:mathring u embed Hs} and \eqref{eq:mathring u decay bootstrap}, we obtain \eqref{eq:mathring u decay}.

Since $u$ is not sufficiently regular, we construct an approximate solution $v$ that solves \eqref{eq:NLS harmonic} in the region $\{|y|\ge\kappa\}$. Such $v$ will be given in the form
\begin{equation}\label{eq:v def sec5}
v(s)=\begin{cases}
v_n(s)&,\quad s\in[s_n,s_{n+1}),\quad 0\le n< n_+
\\
0&,\quad\text{otherwise}
\end{cases}
\end{equation}
with a sequence of times $s_0<s_1<\cdots<s_{n_+}$ in $(s_*-\kappa^{-10},s_*+\kappa^{-10})$ such that
\begin{align}
\label{eq:s_n gap}
&\kappa^p\le s_{n+1}-s_n\le2\kappa^p,
\\
\label{eq:s_n assump mathring}
&\norm{\ul{\mathring u}(s_n)}_{H^{\frac54}(\kappa/4\le|y|\le 4R)}+\norm{|y|^{\frac34}\ul{\mathring u}(s_n)}_{L^2(\kappa/4\le|y|\le 4R)}\les\kappa^{-O(1)}.
\end{align}
Denote $\chi=\psi^{\kappa\sqrt t\le\cdot\le R\sqrt t}$ and $\tilde\chi=\psi^{\gtrsim\kappa\sqrt t}$.
Let $0\le n< n_+$.
On the interval $I_n=[e^{s_n},e^{s_n+2\kappa^p}]$, we approximate the function $\chi u$, which solves the equation
\begin{equation}\label{eq:psi u solves}
(i\d_t+\De)(\chi u)-\mc N(\chi u)=2\na\chi\cdot\na u+((i\d_t+\De)\chi) u+(\chi-\chi^{1+\frac4d})\mc N(u)\eqqcolon f.
\end{equation}
We decompose $f$ into the $\mathring u$-originated part near the origin and the remainder
\[
\mathring f=\psi^{\sim\kappa\sqrt t}\bb{2\na\chi\cdot\na\mathring u+((i\d_t+\De)\chi)\mathring u+(\chi-\chi^{1+\frac4d})\mc N(\mathring u)},\quad\tilde f=f-\mathring f.
\]

Then, we claim that there exists a solution $w_n\in C^0H^{\frac34}\cap L^pW^{\frac34,p}(I_n\times\R^d)$ to
\begin{equation}\label{eq:w_n solves}
(i\d_t+\De)w_n-\tilde\chi\mc N(w_n)=\mathring f,\quad w_n(t_n)=\chi\mathring u(t_n)
\end{equation}
with the initial time $t_n=e^{s_n}$, which satisfies\footnote{All spacetime norms of $w_n$ hereafter are measured on the time interval $I_n$.}
\begin{align}\label{eq:w_n-psi mathring u small}
&\norm{\ul{w_n}-\ul{\chi u}}_{L^\infty H^{-1}\cap L^pL^p\cap L^qL^r(|y|\ge\kappa)}\les \kappa^{-O(1)}R^{-\frac14},
\\&\label{eq:w_n H1 bound}
\norm{\ul{w_n}}_{L^\infty H^{\frac34}}+\norm{e^{-\frac i4|y|^2}\ul{w_n}}_{L^\infty H^{\frac34}}\les\kappa^{-O(1)}.
\end{align}

We show the claim in the case $t_n=e^{s_n}=1$ for simplicity; the proof in general is similar.
First, we show \eqref{eq:w_n-psi mathring u small} via a perturbation argument. By \eqref{eq:X embed u}, we have
\begin{align}\label{eq:m def}
\norm{u}_{L^pL^p(|x|\ge\kappa/16)}&\le\norm{\mathring u}_{L^pL^p(|x|\ge\kappa/16)}+\norm{\tilde u}_{L^pL^p(|x|\ge\kappa/16)}
\\&\les|I_n|^{\frac1p}\norm{\mathring u}_{L^\infty L^p(|x|\ge\kappa/16)}+\eta^{\frac qp}\les\kappa\cdot\kappa^{\frac dp-\frac d2}=o_\kappa(1).\nonumber
\end{align}

By \eqref{eq:eta def}, \eqref{eq:mathring u decay}, \eqref{eq:m def}, and $\norm{\mathring u}_{L^\infty L^2}\les1$, each term in $\tilde f$ can be estimated as
\begin{align*}
&\norm{2\na\chi\cdot\na\tilde u+((i\d_t+\De)\chi)\tilde u}_{L^2H^{-\frac12}}\les R\eta,
\\
&\norm{(\chi-\chi^{1+\frac4d})(\mc N(u)-\mc N(\mathring u))}_{L^2L^{2_{-1}}}\les\norm{\tilde u}_{L^qL^r}\les\eta,
\\
&\norm{\psi^{\sim R}(2\na\chi\cdot\na\mathring u)}_{L^2L^2}\les R^{-1}\norm{\psi^{\sim R}\mathring u}_{L^2H^1}\les \kappa^{-O(1)}R^{-1},
\\
&\norm{\psi^{\sim R}((i\d_t+\De)\chi)\mathring u}_{L^2L^2}\les \norm{\psi^{\sim R}\mathring u}_{L^2L^2}\les \kappa^{-O(1)}R^{-\frac34},
\\
&\norm{\psi^{\sim R}\mc N(\mathring u)}_{L^2L^{2_{-1}}}\les R^{\frac {d}2-1}\norm{\psi^{\sim R}\mathring u}_{L^2L^{\infty}}\les\norm{\psi^{\sim R}\mathring u}_{L^2L^2}^{\frac12}\norm{\psi^{\sim R}\mathring u}_{L^2H^1}^{\frac12}\les \kappa^{-O(1)}R^{-\frac38}.
\end{align*}
Then, by Strichartz and local smoothing estimates, we obtain
\begin{equation}\label{eq:K+ tilde}
\norm{\mc D^+\tilde f}_{L^\infty H^{-1}\cap L^pL^p\cap L^qL^r}\les \kappa^{-O(1)}R^{-\frac14}.
\end{equation}
By the Strichartz estimate and \eqref{eq:weighted shao Lq}, for sufficiently small $\theta=\theta(d)>0$, we have
\begin{align}
&\norm{e^{it\De}(\chi\tilde u(t_n))}_{L^pL^p\cap L^qL^r(|x|\ge\kappa/16)}\les\norm{\tilde u_{\le1}(t_n)}_{\dot H^{\theta}}+\kappa^{-O(1)}\norm{\tilde u_{>1}(t_n)}_{\dot H^{-\theta}},\nonumber
\\
&\norm{\chi\tilde u(t_n)}_{H^{-1}}\les\norm{\tilde u(t_n)}_{H^{-1}(|y|\le 4R)}\les R\norm{\tilde u_{\le1}(t_n)}_{\dot H^\theta}+\norm{\tilde u_{>1}(t_n)}_{\dot H^{-1}},\label{eq:refer H-1}
\end{align}
which are $O(R\eta)$ due to \eqref{eq:eta def}.
Then, by \eqref{eq:m def}, \eqref{eq:K+ tilde}, and standard perturbation theory with \eqref{eq:psi u solves}, a unique solution $w_n$ to \eqref{eq:w_n solves} on $I_n$ exists and satisfies \eqref{eq:w_n-psi mathring u small}.

We turn to showing \eqref{eq:w_n H1 bound}.
By \eqref{eq:s_n assump mathring}, $\ul{w_n}(s_n)=\ul{\chi\mathring u}(s_n)$ satisfies
\begin{equation}\label{eq:w_n init regularities}
\norm{\ul{w_n}(s_n)}_{H^{\frac34}}+\norm{|y|^{\frac34}\ul{w_n}(s_n)}_{L^2}+\norm{e^{-\frac i4|y|^2}\ul{w_n}(s_n)}_{H^{\frac34}}\les\kappa^{-O(1)}.
\end{equation}

We use the local smoothing effect for $\mathring f$, which is supported in the regime $\{|x|\les\kappa\}$. By \eqref{eq:mathring u decay}, $\mathring f$ is $\kappa^{-O(1)}$-bounded in $L^2H^{\frac14}(I_n\times\R^d)$.
Combining with \eqref{eq:w_n init regularities}, we obtain
\begin{align}\label{eq:w_n X bootstrap}
\norm{w_n}_{L^\infty H^{\frac34}\cap L^pW^{\frac34,p}}\les&\norm{w_n(t_n)}_{H^{\frac34}}+\norm{\mathring f}_{L^2 H^{\frac14}}+\norm{\tilde\chi\mc N(w_n)}_{L^{p'}W^{\frac34,p'}}
\\\les&\kappa^{-O(1)}+\norm{w_n}_{L^pL^p(|x|\ge\kappa/16)}^{\frac4d}\norm{w_n}_{L^pW^{\frac34,p}}.\nonumber
\end{align}
By \eqref{eq:w_n-psi mathring u small} and \eqref{eq:m def}, we have $\norm{w_n}_{L^pL^p(|x|\ge\kappa/16)}\ll1$.
Thus, \eqref{eq:w_n X bootstrap} yields \eqref{eq:w_n H1 bound} for $w_n$.  
Since $\mathring f$ is localized in $\{|x|\les\kappa\}$, the pseudo-conformal transform of $\mathring f$ is also $\kappa^{-O(1)}$ in $L^2H^{\frac14}$; similarly working with the pseudo-conformal transform of $w_n$ completes showing \eqref{eq:w_n H1 bound} as claimed.

We now specify the sequence $\{s_n\}_{n\le n_+}$. By \eqref{eq:mathring u decay}, there exists $s_0\in(s_*-\kappa^{-10},s_*-\kappa^{-10}+\kappa^p)$ satisfying \eqref{eq:s_n assump mathring}. We then choose $s_n$ inductively; once $s_{n-1}$ is chosen, by \eqref{eq:eta def}, \eqref{eq:mathring u decay}, and \eqref{eq:w_n-psi mathring u small}, there exists $s_{n}>s_{n-1}$ satisfying \eqref{eq:s_n gap}, \eqref{eq:s_n assump mathring}, and
\begin{equation}\label{eq:s_n w_n assump}
\norm{(\ul{w_{n-1}}-\ul{\chi\mathring u})(s_{n})}_{L^r(|y|\ge\kappa)}\les\kappa^{-O(1)}R^{-\frac14}.
\end{equation}
We iterate the choice of $s_n$ over $0\le n \le n_+ $, where $n_+$ is an index such that $s_{n_+}\ge s_*+\kappa^{-10}-1$. By \eqref{eq:s_n gap}, $n_+\le\kappa^{-O(1)}$ follows.

We estimate $(\ul{w_{n-1}}-\ul{w_{n}})(s_{n})$ for $1\le n\le n_+$. By \eqref{eq:w_n H1 bound}, we have
\begin{equation}\label{eq:s_n w_n H1 bdd}
\norm{(\ul{w_{n-1}}-\ul{w_{n}})(s_{n})}_{H^{\frac34}},\,\norm{e^{-\frac i4|y|^2}(\ul{w_{n-1}}-\ul{w_{n}})(s_{n})}_{H^{\frac34}}\les\kappa^{-O(1)}.
\end{equation}
By \eqref{eq:w_n H1 bound} and the identity \eqref{eq:ix/2t u}, we have
\begin{equation}\label{eq:w_n decay}
\norm{|y|^{\frac34}\ul{w_{n-1}}(s_{n})}_{L^2}\les\norm{y\ul{w_{n-1}}(s_{n})}_{H^{-\frac14}}+\norm{\ul{w_{n-1}}(s_{n})}_{H^{\frac34}}\les\kappa^{-O(1)}.
\end{equation}
Taking the sum of \eqref{eq:w_n init regularities} and \eqref{eq:w_n decay}, we have
\begin{equation}\label{eq:s_n w_n L21 bdd}
\norm{|y|^{\frac34}(\ul{w_{n-1}}-\ul{w_{n}})(s_{n})}_{L^2}\les\kappa^{-O(1)}.
\end{equation}
Since $\ul{\chi\mathring u}(s_n)=\ul{w_n}(s_n)$, interpolating \eqref{eq:s_n w_n assump} and \eqref{eq:s_n w_n L21 bdd} yields for small $\al_d>0$ that
\begin{equation}\label{eq:s_n w_n L21 small}
\norm{\jp y^{\frac12}(\ul{w_{n-1}}-\ul{w_{n}})(s_{n})}_{L^2(|y|\ge\kappa)}\les\kappa^{-O(1)} R^{-\al_d}.
\end{equation}
Interpolating between \eqref{eq:s_n w_n H1 bdd} and \eqref{eq:s_n w_n L21 small}, we obtain
\begin{align}\label{eq:s_n w_n B1/2 small}
&\norm{(\ul{w_{n-1}}-\ul{w_{n}})(s_{n})}_{B^{\frac12}_{2,1}(|y|\ge\kappa)},\norm{e^{-\frac i4|y|^2}(\ul{w_{n-1}}-\ul{w_{n}})(s_{n})}_{B^{\frac12}_{2,1}(|y|\ge\kappa)}
\\\les&\kappa^{-O(1)} R^{-\frac{1}{3}\al_d}.\nonumber
\end{align}
We now begin the virial calculation, which mainly follows Section~\ref{subsec:nonexistence of AP}. We work on $v$ given in \eqref{eq:v def sec5} with $v_n=e^{-\frac i8|y|^2}\ul {w_n}$.
Since $w_n$ solves \eqref{eq:NLS} in the regime $\{|x|\ge2\kappa\sqrt t\}$, for $a\ge2\kappa$, the virial integral \eqref{eq:Virial fight mother} for $v$ is bounded by the sum of boundary terms for the jump discontinuities at times $s_n$, $0\le n\le n_+$:
\begin{equation}\label{eq:jump error}
\mc E_n\coloneqq-\bb{2\mu\cdot\Im([v_n,v_n]_a-[v_{n-1},v_{n-1}]_a)+\d_s\mu\cdot(V_{\varphi^a}(v_n)-V_{\varphi^a}(v_{n-1}))}(s_n).
\end{equation}
Here we use the convention $v_{-1}=v_{n_+}=0$, the bilinear operator $[\cdot,\cdot]_a$ is defined in Lemma~\ref{lem:bilinear}, and we set $\mu(s)=e^{-\frac14|s-s_*|}$.

We estimate \eqref{eq:jump error}. For $n=0$ and $n=n_+$, \eqref{eq:w_n H1 bound} and \eqref{eq:w_n decay} yield
\begin{align}\label{eq:error bound n=0,n+}
|\mc E_n|\les\mu(s_n)\cdot\kappa^{-O(1)}+|\d_s\mu(s_n)|\cdot\kappa^{-O(1)}\les e^{-\frac14\kappa^{-10}}\cdot\kappa^{-O(1)}\ll\eta.
\end{align}
For $1\le n\le n_+-1$, we use the following identity for $k=n$ and $k=n-1$:
\[
2[v_k,v_k]_a(s_n)=[\ul{w_k},\ul{w_k}]_a(s_n)+[e^{-\frac i4|y|^2}\ul{w_k},e^{-\frac i4|y|^2}\ul{w_k}]_a(s_n).
\]
By the bilinear estimate \eqref{eq:bilinear}, \eqref{eq:w_n H1 bound}, and \eqref{eq:s_n w_n B1/2 small}, we have
\begin{align}\label{eq:bilinear bound 1<n<n+}
&|[\ul{w_n},\ul{w_n}]_a(s_n)-[\ul{w_{n-1}},\ul{w_{n-1}}]_a(s_n)|
\\\les&\kappa^{-O(1)}\norm{(\ul{w_n}-\ul{w_{n-1}})(s_n)}_{B^{\frac12}_{2,1}(|y|\ge\kappa)}\cdot\max_{k=n,n-1}\norm{\ul{w_k}(s_n)}_{B^{\frac12}_{2,1}}\les\kappa^{-O(1)}R^{-\frac1{3}\al_d}\nonumber
\end{align}
and the same holds for $e^{-\frac i4|y|^2}\ul{w_n}$ and $e^{-\frac i4|y|^2}\ul{w_{n-1}}$.
By \eqref{eq:w_n decay} and \eqref{eq:s_n w_n L21 small}, we have
\begin{equation}\label{eq:V bound 1<n<n+}
|V_{\varphi^a}(v_n)-V_{\varphi^a}(v_{n-1})|(s_n)=|V_{\varphi^a}(\ul{w_n})-V_{\varphi^a}(\ul{w_{n-1}})|(s_n)\les\kappa^{-O(1)}R^{-\al_d}.
\end{equation}
By \eqref{eq:error bound n=0,n+}--\eqref{eq:V bound 1<n<n+}, the total sum of $\mc E_n$ over $0\le n\le n_+$ is $O(\kappa^{-O(1)}R^{-\frac1{3}\al_d})$.

We now verify \eqref{eq:v 4/d is -2} for $v$, which was used to control the integral \eqref{eq:tag nonlinear integral for Sec 5}. Similarly to \eqref{eq:refer H-1}, $\ul{\chi\tilde u}$ is $O(R\eta)$ in $L^\infty H^{-1}$.
Thus, \eqref{eq:w_n H1 bound}, \eqref{eq:X embed Hs}, and \eqref{eq:w_n-psi mathring u small} yield
\[
\norm{\ul{w_n}-\ul{\chi\mathring u}}_{L^\infty H^{\frac34}(|y|\ge\kappa)}\les\kappa^{-O(1)},\quad \norm{\ul{w_n}-\ul{\chi\mathring u}}_{L^\infty H^{-1}}\les \kappa^{-O(1)}R^{-\frac14}.
\]
Assuming $\al_d>0$ is sufficiently small, an interpolation then yields
\begin{equation}\label{eq:Lp Linfty small v-u}
\norm{e^{\frac i8|y|^2}v_n-\ul{\chi\mathring u}}_{L^\infty L^\infty(|y|\ge\kappa)}=\norm{\ul {w_n}-\ul{\chi\mathring u}}_{L^\infty L^\infty(|y|\ge\kappa)}\les R^{-d\al_d}.
\end{equation}
By \eqref{eq:X embed u} and \eqref{eq:Lp Linfty small v-u}, $v$ satisfies \eqref{eq:v 4/d is -2} in the regime $\{|y|\le R^{\al_d}\}$.

For the exterior regime $\{|y|\ge R^{\al_d}\}$, the integral \eqref{eq:tag nonlinear integral for Sec 5} is small: by \eqref{eq:w_n H1 bound} and the radial Sobolev inequality with $(d-1)(\frac12-\frac1p)\ge\frac1p$, we have
\[
\I\int_{|y|\ge R^{\al_d}}|v|^pdy\les R^{-p(d-1)(\frac12-\frac1p)\cdot \al_d}\sum_{0\le n<n_+}\norm{\ul{w_n}}_{L^\infty H^{\frac34}(|y|\ge R^{\al_d})}^p\les\kappa^{-O(1)} R^{-\al_d}.
\]
We now follow the analysis in Section~\ref{subsec:nonexistence of AP} to obtain an analog of \eqref{eq:ODE ineq}. On each $I_n$, \eqref{eq:X embed Hs} yields $\mathring f\in L^\infty H^{\frac12-4\s_d}$. Thus, by local smoothing and a Picard iteration, $|\na|^{1-4\s_d}w_n\in S(I_n)$ follows. For $r>0$, since $S(I_n)\subset L^2B^{\ep_d}_{2_1,2}(|x|>r)$ and $\ep_d-4\s_d>\frac{d-2}{2d}$, by the radial Sobolev embedding, $w_n\in L^2C^1(|x|>r)$.
Hence, $\ga'$ is continuous on $(0,\infty)$.
We denote $\ga''$ in the viscosity sense:
\begin{equation}\label{eq:ga'' convention}
\ga''(a)\coloneqq\inf\{h''(a):\ep>0,\,h\in C^2((a-\ep,a+\ep)),\,h(a)=\ga(a),\,h\ge\ga\}.
\end{equation}

Similarly to \eqref{eq:ODE ineq}, for $2\kappa\le a<R^{\al_d}$, we obtain
\begin{equation}\label{eq:ODE ineq for radial}
\int_a^\infty\Bb{\frac{a+r}8-\frac c{r^3}} e^{-\gamma(r)} dr+\Bb{\gamma''(a)-\frac c{a^2}}e^{-\gamma(a)}\le\kappa^{-O(1)}R^{-\frac13\al_d}
\end{equation}
for some constant $c=c(d,M)<\infty$.
Since Proposition~\ref{prop:main} is stronger with larger $R_0$, we may assume $R_0\ge 2c^{\frac14}$.
Let $\s=\s(d,\al)>0$ be a sufficiently small number, so that the right-hand side of \eqref{eq:ODE ineq for radial} is below $\frac{c}{a^2}e^{-\ga(a)}$ for $a\le \kappa^{-1}$ such that $\ga(a)\le\s C^2\kappa^{-1}$. 
Then, since $\frac r8\ge\frac c{r^3}$ for $r\ge R_0$, multiplying $e^{\gamma(a)}$ to \eqref{eq:ODE ineq for radial} yields
\begin{align}
\gamma''(a)&\le\begin{cases}
\frac{2c}{a^{2}}+\int_{a}^{R_0}\frac{c}{r^{3}}e^{\gamma(a)-\gamma(r)} dr & ,\quad a\in[2\kappa,R_0]\\
\frac{2c}{a^{2}}-\int_{a}^{\infty}(\frac{a}{4}-\frac c{r^3})e^{\gamma(a)-\gamma(r)} dr & ,\quad  a\in[R_0,\kappa^{-1}]
\end{cases}\label{eq:ode ineq'}
\\
\label{eq:ode ineq assump}
\text{if}\quad\gamma(a)&\le\s C^2\kappa^{-1}.
\end{align}

Building on \eqref{eq:ode ineq'}, we provide a lemma on the function $\gamma$.

\begin{lem}\label{lem:k<a*<1}
Let $\delta\ll_{M,R_0,\rho}1$ and $C_1>0$.
\begin{enumerate}
\item For $\kappa\ll_{\delta,C_1}1$ for which \eqref{eq:main} fails, we have
\begin{equation}\label{eq:gamma at R_0}
\gamma(R_0)\le\kappa^{-\frac34}.
\end{equation}
\item
Given any $\be>2$, assume further that $\delta\ll_\be1$. For $\kappa\ll_{\delta,C_1}1$ for which \eqref{eq:main} fails, we have
\begin{equation}
\gamma'(R_0)\le\be|\log\delta|\kappa^{-\frac12}.\label{eq:gamma' at R_0}
\end{equation}
\item
Given any $\zeta>0$ and $\theta\gg_{\delta,\zeta}1$, for $\kappa\ll_{\delta,C_1,\theta}1$ for which \eqref{eq:main} fails, we have
\begin{equation}\label{eq:ode alternative}
\text{if }\quad \ga'(R_0)>\zeta\kappa^{-\frac12}\,,\quad\text{then}\quad-\log\sqrt\kappa+\ga(\theta\sqrt\kappa)\les_{\delta,\zeta}\theta.
\end{equation}
\end{enumerate}
\end{lem}
\begin{proof}
For $a\in[2\kappa, R_0]$, let $\ul\ga$ be the increasing Lipschitz function
\[
\ul\gamma(a)\coloneqq\min\{\inf\gamma([a,R_0]),\s C^2\kappa^{-1}\}.
\]
We use the convention \eqref{eq:ga'' convention} also for $\ul\ga''$.
We claim for $a\in[2\kappa,R_0]$ that
\begin{equation}\label{eq:tilde gamma ''}
\ul\gamma''(a)\le\frac{2c}{a^2}+\int_a^{R_0}\frac c{r^3} dr\le\frac{3c}{a^2}.
\end{equation}

If $\ul\gamma(a)=\gamma(a)$, then $\gamma(a)-\gamma(r)\le0$ for $r\in[a,R_0]$ and $\ul\ga''(a)\le\ga''(a)$. Thus, \eqref{eq:ode ineq'} yields \eqref{eq:tilde gamma ''}. Otherwise, we have $\ul\gamma''(a)=0$. Hence, we obtain \eqref{eq:tilde gamma ''}.

We investigate bounds on the function $\ga$.
Since \eqref{eq:main} fails, by \eqref{eq:w_n-psi mathring u small}, we have
\begin{equation}\label{eq:mass lower bound of gamma}
\int_{\sqrt\kappa}^{R_0}\kappa^{-1}r^2e^{-\gamma(r)}dr=\kappa^{-1}\I\norm{yv}_{L^2(\sqrt\kappa\le|y|\le R_0)}^2\gtrsim\rho^2\delta^2-o_\kappa(1).
\end{equation}
By the mass bound $M(u)\le M$ and \eqref{eq:w_n-psi mathring u small}, we have
\begin{equation}\label{eq:mass upper bound of gamma}
\int_{0}^{\kappa^{-1}} e^{-\gamma(r)}dr=\I\norm{v}_{L^2(|y|\le \kappa^{-1})}^2\les M.
\end{equation}
We show \eqref{eq:gamma at R_0}. If \eqref{eq:gamma at R_0} fails, then $\ul\gamma(R_0)>\kappa^{-\frac34}$. Combining with \eqref{eq:mass lower bound of gamma}, there exists $a\in[\sqrt\kappa,R_0]$ such that $\ul\gamma'(a)\gtrsim_{R_0}\kappa^{-\frac34}$.
Then, for $r\in[\frac12\sqrt\kappa,\sqrt\kappa]$, by integrating \eqref{eq:tilde gamma ''}, we have $\ul\ga'(r)\gtrsim_{R_0}\kappa^{-\frac34}$ and hence $\ul\ga(r)=\ga(r)$. Thus, we have
\[
\int_{\frac12\sqrt\kappa}^{\sqrt\kappa}e^{-\ga(r)}dr=\int_{\frac12\sqrt\kappa}^{\sqrt\kappa}e^{-\ul\ga(r)}dr\ge e^{\ep_0\kappa^{-\frac14}}e^{-\ul\ga(\sqrt\kappa)}\ge e^{\ep_0\kappa^{-\frac14}}\cdot\frac{1}{R_0}\int_{\sqrt\kappa}^{R_0}e^{-\ga(r)}dr
\]
for small $\ep_0=\ep_0(R_0)>0$, contradicting \eqref{eq:mass lower bound of gamma} and \eqref{eq:mass upper bound of gamma} for $\kappa\ll_{\ep_0,\delta}1$.

We turn to proving \eqref{eq:gamma' at R_0}. If \eqref{eq:gamma' at R_0} fails, then for $\ep>0$ and $\delta\ll_\ep1$, \eqref{eq:tilde gamma ''} yields
\[
\ul\gamma'(r)\ge\ul\ga'(R_0)-\tfrac{3c}r>(\be-\ep)|\log\delta|\kappa^{-\frac12},\quad r\in[\ep\sqrt\kappa,R_0].
\]
In particular, $\ul\ga(r)=\ga(r)$. Choosing $\ep=\ep(\be)\ll1$,
we obtain
\[
\int_{\sqrt\kappa}^{R_0}\kappa^{-1}r^2e^{-\gamma(r)}dr\les \frac{e^{-\ga(\sqrt\kappa)}}{(\be-\ep)|\log\delta|\kappa^{-\frac12}}\les e^{-(1-\ep)(\be-\ep)|\log\delta|}\int_{\ep\sqrt\kappa}^{\sqrt\kappa}e^{-\gamma(r)}dr\les \delta^{2+\ep}M
\]
due to \eqref{eq:mass upper bound of gamma}. For $\delta\ll_{\be,M,\rho}1$, this contradicts \eqref{eq:mass lower bound of gamma}.

We show \eqref{eq:ode alternative}. Integrating \eqref{eq:tilde gamma ''} from $\ul\ga'(R_0)=\ga'(R_0)>\zeta\kappa^{-\frac12}$ yields for  $\theta\gg_\zeta1$
\begin{equation}\label{eq:ulul ga for superpoly}
\ul\ga'(r)\gtrsim\zeta\kappa^{-\frac12}\text{ and }\ul\ga(r)=\ga(r),\quad r\in[\theta\sqrt\kappa,R_0].
\end{equation}
We argue by contradiction; assume $-\log\sqrt\kappa+\ga(\theta\sqrt\kappa)\gg_{\delta,\zeta}\theta$. Then, \eqref{eq:ulul ga for superpoly} yields
\[
-\log\sqrt\kappa+\ga(r)\gtrsim\zeta \kappa^{-\frac12}r,\quad r\in[\theta\sqrt\kappa,R_0].
\]
Thus, for \eqref{eq:mass lower bound of gamma} to hold, there exists $r_0\in[\sqrt\kappa,\theta\sqrt\kappa]$ such that
\[
-\log\sqrt\kappa+\ul\ga(r_0)\le-\log\sqrt\kappa+\ga(r_0)\les_\delta \kappa^{-\frac12}r_0.
\]
Then, there also exists $a\in[r_0,\theta\sqrt\kappa]$ such that $\ul\ga'(a)\gg_{\delta,\zeta}\kappa^{-\frac12}$. Integrating \eqref{eq:tilde gamma ''} from $a$ yields $\ga'=\ul\ga'\gg_{\delta,\zeta}\kappa^{-\frac12}$ on $[\frac12\sqrt\kappa,a]$. Integrating $\ga'$ from $r_0$ then yields
\[
-(-\log\sqrt\kappa+\ga(r))\gg_{\delta,\zeta}1,\quad r\in[\tfrac12\sqrt\kappa,\tfrac34\sqrt\kappa],
\]
contradicting \eqref{eq:mass upper bound of gamma}. This completes the proof.
\end{proof}
\begin{proof}[Proof of Proposition~\ref{prop:main}]
We fix any $\be>2$.\footnote{Here we restrict to $\be>2$ because we showed $\gamma'(R_0)\le\be|\log\delta|\kappa^{-\frac12}$ for $\be>2$.}
We show that Proposition~\ref{prop:main} holds for any fixed constant $C_1>\be\sqrt{\frac{\pi}{2\s}}$.
Arguing by contradiction, assume that \eqref{eq:main} fails. Let $\tilde\ga(a)=\gamma(a)+2c\log a$. Let $I=[R_0,a_+]\subset[R_0,\kappa^{-1}]$ be the maximal interval such that $\sup\tilde\ga(I)\le\s C^2\kappa^{-1}$. By \eqref{eq:ode ineq'}, we have
\begin{equation}\label{eq:ode ineq' gam*}
\tilde\ga''(a)=\gamma''(a)-\frac{2c}{a^2}\le-\int_a^\infty\Bb{\frac a4-\frac c{r^3}}e^{\tilde\ga(a)-\tilde\ga(r)}dr,\quad a\in I.
\end{equation}

Let $a_+'=\min\{a_+,\frac12\kappa^{-1}\}$.
We first show $\tilde\ga'(a)>0$ for $a\in [R_0,a_+']$. If not, by \eqref{eq:ode ineq' gam*}, $\tilde\ga$ is decreasing on $[a,a_+]$ and $a_+=\kappa^{-1}$. Then, by \eqref{eq:ode ineq' gam*}, we have $-\tilde\ga''(r)\gtrsim\kappa^{-1}$ for $r\in[a,\frac34\kappa^{-1}]$. Hence, $\tilde\ga(r)\le0$
follows for $r\in[\frac34\kappa^{-1},\kappa^{-1}]$, contradicting \eqref{eq:mass upper bound of gamma}.

Then, $\tilde\ga'(a)\gtrsim 1$ follows for $a\in[R_0,a_+'-2]$; otherwise, $\tilde\ga'(r)\ll1$ for $r\in[a,a_+']$ and hence $-\tilde\ga''(r)\gtrsim1$ for $r\in[a,a+1]$, contradicting $0<\tilde\ga'\ll1$ over $[a,a_+']$.

We begin the main proof.
For $a\in[R_0,a_+']$, since $\tilde\ga''\le0$ over $I$, \eqref{eq:ode ineq' gam*} yields
\begin{equation}\label{eq:ode ineq' lower substituted}
\tilde\ga'(a)\tilde\ga''(a)\le-\bb{\tfrac a4-\tfrac c{a^3}}\bb{1-e^{-\tilde\ga'(a)(a_+'-a)}}\le-\bb{\tfrac14-o_a(1)-o_{a_+'-a}(1)}\cdot a.
\end{equation}
For $a\in[R_0,a_+']$, by integrating \eqref{eq:ode ineq' lower substituted} from \eqref{eq:gamma' at R_0}, we have
\begin{equation}\label{eq:gam' bound}
0<\tilde\gamma'(a)^2\le \bb{\be|\log\delta|\kappa^{-\frac12}+O(1)}^2-\bb{\tfrac14-o_a(1)}a^2.
\end{equation}

Again, by integrating the square root of \eqref{eq:gam' bound} from \eqref{eq:gamma at R_0}, we have
\begin{align}\label{eq:gam bound}
\tilde\gamma(a)&\le O(\kappa^{-\frac34})+\int_{R_0}^a\sqrt{\bb{\be|\log\delta|\kappa^{-\frac12}+O(1)}^2-\bb{\tfrac14-o_r(1)}r^2}dr
\\&\le \bb{\tfrac{\pi}2\be^2+o_\kappa(1)}\cdot|\log\delta|^2\kappa^{-1}.\nonumber
\end{align}

Since $C|\log\delta|^{-1}=C_1>\be\sqrt{\frac\pi{2\s}}$, \eqref{eq:gam bound} stays below $\s C^2\kappa^{-1}$ for $\kappa\ll1$. Hence, $a_+'=\frac12\kappa^{-1}$ follows; plugging $a=a_+'$ into \eqref{eq:gam' bound} concludes the contradiction.
\end{proof}
Next, we show a preparatory lemma for Theorem \ref{thm:saturate loglog}.
\begin{lem}\label{lem:superexp}
Let $0<\delta\ll_{M,R_0,\rho}1$ and $C_1>0$. There exists $c_*>0$ satisfying the following: For $\theta\gg_{\delta,C_1}1$ and $\kappa\ll_\theta1$ such that Proposition~\ref{prop:main} fails,
\begin{equation}\label{eq:assump for superexp}
(\theta\sqrt\kappa)^d\int_{s_*-\kappa^{-10}}^{s_*+\kappa^{-10}}|\ul u(s,\theta\sqrt\kappa)|^2\cdot e^{-\frac14|s-s_*|}ds\ge e^{-c_*\theta}.
\end{equation}
\end{lem}
\begin{proof}
Since $\ep_d>\frac{d-2}{2d}$, there exists $\tilde\ep_d>0$ such that,
on $[s_*-\kappa^{-10},s_*+\kappa^{-10}]$,
\[
\norm{\ul{\tilde u}}_{L^2W^{\tilde\ep_d,\infty}(|y|\ge\kappa)}\les\kappa^{-O(1)}(\norm{\ul{\tilde u}}_{L^pL^p}+\norm{P_{\ge1}\ul{\tilde u}}_{L^2B^{\ep_d}_{2_1,2}(|y|\ge\kappa)})\les\kappa^{-O(1)}\norm{\tilde u}_S\les\kappa^{-O(1)}
\]
by the definition of the $S$ norm and \eqref{eq:S embed}. Interpolating with \eqref{eq:eta def}, we obtain that $\ul{\tilde u}$ is $o(\kappa^{10d})$ in $L^2L^\infty(|y|\ge\kappa)$.
Combining with \eqref{eq:Lp Linfty small v-u}, we obtain
\begin{equation}\label{eq:u-v L2 Linfty small}
\norm{e^{-\frac18|s-s_*|}(e^{\frac i8|y|^2}v-\ul{u})}_{L^2L^\infty([s_*-\kappa^{-10},s_*+\kappa^{-10}]\times\{\sqrt\kappa\le|y|\le1\})}\ll\kappa^{10d}.
\end{equation}
Thus, it suffices to show \eqref{eq:assump for superexp} for $v$. Expanding \eqref{eq:ga def} yields
\begin{align*}
-\ga(\theta\sqrt\kappa)&=\log|(\theta\sqrt\kappa)\S^{d-1}|+\log\I|v(s,\theta\sqrt\kappa)|^2
\\&=O(1)-\log(\theta\sqrt\kappa)+\log\bb{(\theta\sqrt\kappa)^d\cdot\I|v(s,\theta\sqrt\kappa)|^2}.
\end{align*}
Hence, it suffices to show for $\theta\gg1$
\begin{equation}\label{eq:assump for superexp reduced}
-\log\sqrt\kappa+\ga(\theta\sqrt\kappa)\le c_*\theta,\quad c_*=c_*(\delta,C_1).
\end{equation}
Fix $\be=\be(C_1)>0$ such that $C_1>\be\sqrt{\frac\pi{2\s}}$. If $\ga'(R_0)\le \be|\log\delta|\kappa^{-\frac12}$ holds, then Proposition~\ref{prop:main} can be shown, as mentioned at the beginning of the proof. Thus, for Proposition~\ref{prop:main} to fail, $\ga'(R_0)>\be|\log\delta|\kappa^{-\frac12}$ holds. Then, \eqref{eq:ode alternative} with $\zeta=\be|\log\delta|$ yields \eqref{eq:assump for superexp reduced}. This finishes the proof.
\end{proof}

We fix $m_M\in\N$ such that Lemma~\ref{lem:regularity sec6} with $\iota=1/3$ applies for $a>m_M$. For integers $m>m_M$, we decompose $\mathring u_m+\tilde u_m=\varphi_m u$ and $\mathring u_m^\pc+\tilde u_m^\pc=\varphi_m u^\pc$ as in Lemma~\ref{lem:regularity sec6}. We also adopt the conventions $\mathring u_{m_M}=\mathring u$ and $\tilde u_{m_M}=\tilde u$.

In the next proposition, we provide an optimal improvement of $C_1$ in Proposition~\ref{prop:main}.
\begin{prop}\label{prop:main sec6}
Let $(q,r)$ be a Strichartz pair with $2<q\le p$ and $C_1>\sqrt\pi$.

Let $M>0$, $\rho>0$, and $\delta\ll_{M,\rho}1$. Denote $C=C_1|\log\delta|$. Let $\kappa\ll_{u,\delta}1$ and $s_*\le-\kappa^{-10}$ be parameters such that, on the interval $[s_*-\kappa^{-10},s_*+\kappa^{-10}]$,
\begin{align}\label{eq:eta def sec6}
&\norm{\ul{\tilde u_m}}_{L^qL^r}+\norm{\ul{\tilde u_m^\pc}}_{L^qL^r}+\norm{\jp y^{-1}P_{\ge1}\ul{\tilde u_m}}_{L^2H^{\frac12}}
\\&+\norm{P_{e^{-\kappa^{-2}}\le\cdot \le e^{\kappa^{-2}}}\ul{\tilde u_m}}_{L^\infty L^2}\le e^{-C^2\kappa^{-1}}\eqqcolon\eta,&\quad m_M\le m\le \kappa^{-1}.\nonumber
\end{align}
Then, we have \eqref{eq:main} with $R_0=1$, namely
\begin{equation}\label{eq:main sec6}
\norm{y\ul u}_{L^2L^2([s_*-1,s_*+1]\times\{\sqrt\kappa\le|y|\le 1\})}\le\rho\delta\sqrt\kappa.
\end{equation}
\end{prop}
\begin{proof}
In this proof, we choose the time weight $\mu(s)=e^{-2\nu|s-s_*|}$ with a sufficiently small exponent $\nu=\nu(d)>0$ to exploit Lemma~\ref{lem:regularity sec6}. Let $\al>0$ and $\be>2$ be aforementioned parameters, which will be fixed later. Also, let $0<\ep_0\ll\al$ and $k\in\N$ be numbers to be fixed later.

For $m_M+1\le m\le\kappa^{-1}$,
we mostly follow the analysis for \eqref{eq:ode ineq'} (namely \eqref{eq:mathring u decay}--\eqref{eq:s_n w_n assump}), replacing $\mathring u$ and the spatial cutoffs $\chi,\tilde\chi$ by $\mathring u_m$ and $\chi_m\psi^{\le R\sqrt t},\chi_{m+1}$.
We use subscripts to denote analogous notations, e.g. $v_m$ to denote the piecewise solution \eqref{eq:v def sec5}. We use the subscript $m_M$ to denote the original functions in the previous proof, e.g. $v_{m_M}\coloneqq v$.

For $m_M+1\le m\le \kappa^{-1}$ and $0\le n<n_{+,m}$, by \eqref{eq:S<L and X<L sec6} and \eqref{eq:tilde u embed Lp cap H1/2}--\eqref{eq:mathring u embed Hs}, we have
\begin{align*}
&\norm{\na\ul{\mathring u_m}}_{L^2H^{\frac14}([s_n,s_{n+1}]\times\{|y|\ge m\})}+\norm{\na\ul{\mathring u_m^\pc}}_{L^2H^{\frac14}([s_n,s_{n+1}]\times\{|y|\ge m\})}
\\\les& \kappa^{-O(1)}\norm{e^{-\nu|s-s_n|}\ul u}_{L^2H^{\frac12}(m-1\le|y|\le m)}\les\kappa^{-O(1)}.
\end{align*}
For $m=m_M$, the same bound holds due to \eqref{eq:mathring u embed Hs}.

We track $\mu$-weighted improvements of this bound by introducing the parameters
\begin{align*}
\zeta_m&\coloneqq\I\norm{\na\ul{\mathring u_m}}_{H^{\frac14}(|y|\ge m)}^2+\I\norm{\na\ul{\mathring u_m^\pc}}_{H^{\frac14}(|y|\ge m)}^2+(R^{5A}\eta)^2\,\les\,\kappa^{-O(1)},
\\
\omega_m&\coloneqq\I\norm{\ul u}_{L^2(m\le|y|\le m+1)}^2+(R^{5A}\eta)^2.
\end{align*}

For each $0\le n<n_{+,m}$, since the right-hand side of \eqref{eq:mathring u decay bootstrap} for the interval $[s_{n,m},s_{n+1,m}]$ is bounded by $\sqrt{\mu(s_{n,m})^{-1}\zeta_m}$, repeating the previous proof gives \eqref{eq:w_n-psi mathring u small}--\eqref{eq:w_n H1 bound} for $w_{n,m}$ with $\sqrt{\mu(s_{n,m})^{-1}\zeta_m}$-times smaller right-hand sides
\begin{align}
\norm{\ul{w_{n,m}}-\psi^{\le R}\ul{u}}_{L^\infty H^{-1}\cap L^pL^p\cap L^qL^r(|y|\ge m+1)}&\les\kappa^{-O(1)} R^{-\frac14}\cdot\textstyle\sqrt{\mu(s_{n,m})^{-1}\zeta_m},\label{eq:w(n,m)-u Lp small}
\\
\norm{\ul{w_{n,m}}}_{L^\infty H^{\frac34}}+\norm{e^{-\frac i4|y|^2}\ul{w_{n,m}}}_{L^\infty H^{\frac34}}&\les\kappa^{-O(1)}\cdot\textstyle\sqrt{\mu(s_{n,m})^{-1}\zeta_m}.\label{eq:w(n,m) H 3/4 bound}
\end{align}
Hence, we obtain an improved version of \eqref{eq:ODE ineq for radial}
\[
\int_a^\infty\Bb{\frac{a+r}8-\frac c{r^3}} e^{-\gamma_m(r)}dr+\Bb{\gamma_m''(a)-\frac c{a^2}}e^{-\gamma_m(a)}\le\kappa^{-O(1)}R^{-\frac13\al_d}\zeta_m,\quad a\ge m+1,
\]
where $\mu(s_{n,m})^{-1}$ is canceled by the multipliers $\mu$ and $\d_s\mu=O(\mu)$ in \eqref{eq:jump error}.

Consequently, we obtain an improved version of the lower case in \eqref{eq:ode ineq'}: 
\begin{align}\label{eq:ode ineq' sec6}
&\ga_m''(a)\le
\frac{2c}{a^{2}}-\int_{a}^{\infty}\Bb{\frac{a}{4}-\frac c{r^3}}e^{\gamma_m(a)-\gamma_m(r)} dr,\quad  a\in[m+1,\kappa^{-1}]
\\
\text{if}\quad&\ga_m(a)\le\s C^2\kappa^{-1}+|\log\zeta_m|\label{eq:ode ineq' sec6 condt}
\end{align}
with $\s=\s(d,\al)>0$.
By \eqref{eq:S<L and X<L sec6} and \eqref{eq:eta def sec6}, we obtain
\begin{align*}
\zeta_{m+1}&\les\kappa^{-O(1)}(\I\norm{\ul{\mathring u_m}}_{H^{\frac12}(m\le|y|\le m+1)}^2+\I\norm{\ul{\tilde u_m}}_{H^{\frac12}(m\le|y|\le m+1)}^2)+(R^{5A}\eta)^2
\\
&\les\kappa^{-O(1)}\I(\norm{\ul{\mathring u_m}}_{L^2(m\le|y|\le m+1)}^2+\norm{\ul{\mathring u_m}}_{L^2(m\le|y|\le m+1)}\norm{\na \ul{\mathring u_m}}_{L^2(|y|\ge m)})+(R^{5A}\eta)^2
\\
&\les\kappa^{-O(1)}\max\{\omega_{m},\sqrt{\omega_{m}\zeta_{m}}\}.
\end{align*}
Applying this estimate $k$ times, we obtain for $m\ge m_M+k$
\begin{equation}\label{eq:zeta iteration}
|\log\zeta_m|\ge -O(|\log\kappa|)+(1-2^{-k})\min_{m-k\le n<m}|\log\omega_n|.
\end{equation}

We argue by contradiction. Assume that \eqref{eq:main sec6} fails.
For $j\ge 0$, we denote
\[
m_j=m_M+j\lceil\ep_0|\log\delta|\kappa^{-\frac12}\rceil.
\]
To deduce a contradiction, we show the following for $j=0,\ldots,\lceil10/\ep_0\rceil$:

(i) $\ga_{m_j}$ satisfies \eqref{eq:gam' bound}--\eqref{eq:gam bound} over $[m_j+4,m_{j+2})$,

(ii) for $l\in\{j,j+1\}$ and integers $n\in[m_l+1,m_{j+2})\cap[m_j+4,m_{j+2})$,
\begin{equation}\label{eq:gam(a)=w(n)}
\I\norm{v_{m_l}}_{L^2(n\le|y|\le n+1)}^2\sim \omega_n\gg\max\{ R^{-\frac13}\zeta_{m_j},R^{-\frac13}\zeta_{m_{j+1}},(R^{5A}\eta)^2\}.
\end{equation}
Note that we obtain \eqref{eq:gam(a)=w(n)} once we show for $n\in[m_j+4,m_{j+2})$
\begin{equation}\label{eq:gam(a)=w(n) sufficient}
\I\norm{v_{m_j}}_{L^2(n\le|y|\le n+1)}^2\gg\max\{R^{-\frac13}\zeta_{m_j},R^{-\frac13}\zeta_{m_{j+1}},(R^{5A}\eta)^2\},
\end{equation}
due to the triangle inequality and the consequence of \eqref{eq:w(n,m)-u Lp small},
\[
\I\norm{e^{\frac i8|y|^2}v_{m_l}-\ul u}_{L^2(n\le|y|\le n+1)}^2\les\kappa^{-O(1)}R^{-\frac12}\zeta_{m_l}\ll R^{-\frac13}\zeta_{m_l}.
\]
We prove this claim by induction. Hereafter, we often use the integral identity
\[
\int_{n}^{n+1}e^{-\ga_{m_j}(r)}dr=\I\norm{v_{m_j}}_{L^2(n\le|y|\le n+1)}^2.
\]
For $j=0$, assuming $\ep_0\ll1$, the integral in \eqref{eq:gam bound} for $\ga_{m_0}=\ga$ is below $\s C^2\kappa^{-1}$ for $a\in[R_0,m_2)$. Thus,  $a_+'\ge m_2$ follows and \eqref{eq:gam' bound}--\eqref{eq:gam bound} hold for $a\in [m_0+4,m_2)$. In particular, $\ga(a)\le O(\ep_0C^2\kappa^{-1})$. Comparing to $|\log(R^{-\frac13}\zeta_{m_l})|\gtrsim\al C^2\kappa^{-1}$ for $l\ge0$, for $\ep_0\ll\al$, \eqref{eq:gam(a)=w(n) sufficient} holds for all $n\in[m_0+4,m_2)$, yielding \eqref{eq:gam(a)=w(n)}.

We show the induction step; let $j\ge1$.
By \eqref{eq:gam(a)=w(n)} for $(j-1)$, there exist $a_-\in[m_j+1,m_j+2]$ and $a_+\in[m_j+3,m_j+4]$ such that $|\ga_{m_j}(a_\pm)-\ga_{m_{j-1}}(a_\pm)|=O(1)$.
Thus, there exists $a_*\in[m_j+1,m_j+4]$ such that
\begin{equation}\label{eq:a* def}
|\ga_{m_j}(a_*)-\ga_{m_{j-1}}(a_*)|=O(1)\text{ and }(\ga_{m_j}-\ga_{m_{j-1}})'(a_*)\le O(1).
\end{equation}
Since \eqref{eq:gam' bound} and \eqref{eq:gam bound} hold for $\ga_{m_{j-1}}$ over $[m_{j-1}+4,m_{j+1})$, we have
\begin{align}\label{eq:gam(a)=gam(m) prep}
|\ga_{m_{j-1}}(a)-\ga_{m_{j-1}}(a_*)|&=O(\ep_0|\log\delta|^2\kappa^{-1}),\quad a\in[m_{j-1}+4,m_{j+1}),
\\
\label{eq:ga j-1 at a*}
\ga_{m_{j-1}}(a_*)&=O(|\log\delta|^{2}\kappa^{-1}).
\end{align}
By \eqref{eq:gam(a)=w(n)} for $(j-1)$ and \eqref{eq:gam(a)=gam(m) prep}, we obtain
\begin{equation}\label{eq:gam(a)=gam(m)}
||\log \omega_n|-\ga_{m_{j-1}}(a_*)|=O(\ep_0|\log\delta|^2\kappa^{-1}),\quad n\in [m_{j-1}+4,m_{j+1}).
\end{equation}
By \eqref{eq:zeta iteration} and \eqref{eq:gam(a)=gam(m)}, we have
\begin{equation}\label{eq:gam bound conseq for Im}
|\log\zeta_{m_l}|\ge(1-2^{-k})(\ga_{m_{j-1}}(a_*)-O(\ep_0|\log\delta|^2\kappa^{-1})),\quad l\in\{j,j+1\}.
\end{equation}
By \eqref{eq:a* def} and \eqref{eq:gam bound conseq for Im}, assuming $\ep_0\ll1$ and $k\gg1$, we have
\[
\ga_{m_j}(a_*)\le O(\ep_0|\log\delta|^2\kappa^{-1})+(1-2^{-k})^{-1}|\log\zeta_{m_j}|\le\tfrac\s2 C^2\kappa^{-1}+|\log\zeta_{m_j}|.
\]
In particular, $\ga_{m_j}$ satisfies \eqref{eq:ode ineq' sec6 condt} at $a_*$.

Let $[a_*,a_+]\subset[a_*,m_{j+2}]$ be the maximal interval over which \eqref{eq:ode ineq' sec6 condt} holds. By \eqref{eq:gam' bound}--\eqref{eq:gam bound} for $\ga_{m_{j-1}}(a_*)$ and \eqref{eq:a* def}, we obtain \eqref{eq:gam' bound}--\eqref{eq:gam bound} for $\ga_{m_{j}}(a_*)$.
Repeating the previous proof provides \eqref{eq:gam' bound}--\eqref{eq:gam bound} over $[a_*,a_+]$. In particular, we have
\begin{equation}\label{eq:ga mj '}
\ga_{m_j}'(a)=O(|\log\delta|\kappa^{-\frac12}),\quad a\in[a_*,a_+],
\end{equation}
thus $a_+-a_*\gtrsim|\log\delta|\kappa^{-\frac12}$ follows. Hence, \eqref{eq:gam' bound}--\eqref{eq:gam bound} extend to $m_{j+2}$ for $\ep_0\ll1$.

It remains to show \eqref{eq:gam(a)=w(n)}. 
Let $a\in[a_*,m_{j+2})$. By \eqref{eq:a* def} and \eqref{eq:ga mj '}, we have
\begin{equation}\label{eq:ga m bound}
|\ga_{m_j}(a)-\ga_{m_{j-1}}(a_*)|\le O(1)+|\ga_{m_j}(a)-\ga_{m_j}(a_*)|\les\ep_0|\log\delta|^2\kappa^{-1}.
\end{equation}
Since $C_1>\sqrt\pi$, we can choose $\al>0$ and $\be>2$ so that \eqref{eq:gam bound} for $\ga_{m_{j}}$ yields
\begin{equation}\label{eq:ga primary bound 2}
e^{-\ga_{m_{j}}(a)}\ge e^{-(\frac\pi2\be^2+o_\kappa(1))\cdot|\log\delta|^2\kappa^{-1}}\gg e^{2(5A\al-1) C_1^2|\log\delta|^2\kappa^{-1}}= (R^{5A}\eta)^2.
\end{equation}
By \eqref{eq:ga m bound}, \eqref{eq:ga j-1 at a*}, and \eqref{eq:gam bound conseq for Im}, choosing sufficiently small $\ep_0>0$ and sufficiently large $k\in\N$, we have for $l\in\{j,j+1\}$
\begin{equation}\label{eq:ga primary bound 1}
e^{-\ga_{m_{j}}(a)}\gg e^{-\frac13\al C_1^2|\log\delta|^2\kappa^{-1}-(1-2^{-k})(\ga_{m_{j-1}}(a_*)-O(\ep_0|\log\delta|^2\kappa^{-1}))}\ge R^{-\frac13}\zeta_{m_l}.
\end{equation}

For each $n\in[m_j+4,m_{j+2})\subset [a_*,m_{j+2})$, integrating \eqref{eq:ga primary bound 2} and \eqref{eq:ga primary bound 1} over $[n,n+1)$ yields \eqref{eq:gam(a)=w(n) sufficient}. Hence, \eqref{eq:gam(a)=w(n)} holds and the induction step is complete.

Now we are ready to obtain a contradiction.
We choose $j=\lceil10/\ep_0\rceil$;
plugging $a=m_j+4\ge10|\log\delta|\kappa^{-\frac12}$ into \eqref{eq:gam' bound} makes its right-hand side negative and leads to a contradiction. This finishes the proof.
\end{proof}

\section{Proofs of Theorem~\ref{thm:loglog}, Corollary \ref{cor:loglog la Bourgain}, and Theorem~\ref{thm:saturate loglog}}\label{sec:implications}
We provide the proofs of the results for radial solutions, namely Theorems \ref{thm:loglog} and \ref{thm:saturate loglog}, and Corollary \ref{cor:loglog la Bourgain}.
By time translation and the time-reversal symmetry, we may assume that $u$ exists on $(0,1]$ and blows up backward in time at $t=0$. The main ingredients are Propositions \ref{prop:main} and \ref{prop:main sec6}, whose notations we adopt in this section.

First, we detect the asymptotic profile $z^*$ for Theorem~\ref{thm:loglog} (1).
\begin{lem}\label{lem:negative Sobolev}
There exists $z^*\in L^2(\R^d)$ satisfying the following for $\theta\in[0,1)$:
\begin{align}\label{eq:negative Sobolev}
&\norm{u(t)-z^*}_{\dot H^{-\theta}(\R^d)}\les_{M,\theta} t^{\frac\theta2},&t>0,
\\
\label{eq:exterior vanish}
&\limsup_{t\to0}\norm{u(t)-z^*}_{L^2(|x|\ge R\sqrt t)}=o_R(1),&R>0.
\end{align}
\end{lem}
\begin{proof}
Since $\sup_t\norm{u(t)}_{L^2}\les1$, $u(t_n)$ converges weakly in $L^2$ to some $z^*$ for some sequence $\{t_n\}$ with $t_n\to 0$. We first show \eqref{eq:negative Sobolev} with this $z^*$; by scaling argument, it suffices to consider when $t=1$. We use a truncated equation,
\begin{equation}\label{eq:exterior cutoff nls}
\d_tP_{\le1}(\psi^{>1}u)=-iP_{\le1}(-\psi^{>1}\De u+\psi^{>1}\mc N(u)).
\end{equation}
On the time interval $(0,1]$, by \eqref{eq:X embed u} and \eqref{eq:tilde u embed Lp cap H1/2}, $u=\mathring u+\tilde u$ is bounded in $L^\infty L^2\cap L^pL^p(|x|\ge1)$. Thus, $P_{\le1}(\psi^{>1}\mc N(u))$ is bounded in $L^2L^{2_{-1}}$, hence in $L^2\dot H^{-\theta}$.
On the other hand, since $\De u$ is bounded in $L^\infty \dot H^{-2}$, $P_{\le1}(\psi^{>1}\De u)$ is bounded in $L^\infty\dot H^{-\theta}$.
Thus, $\d_tP_{\le 1}(\psi^{>1}u)$ is bounded in $L^1\dot H^{-\theta}$. Integrating in time yields
\[
\norm{P_{\le1}(\psi^{>1}(u(1)-z^*))}_{\dot H^{-\theta}}\les1.
\]
Other terms for \eqref{eq:negative Sobolev} are straightforward; by the mass bound, $P_{>1}(u(1)-z^*)$ and $P_{\le1}(\psi^{\le1}(u(1)-z^*))$ are bounded in $\dot H^{-\theta}$. Hence we obtain \eqref{eq:negative Sobolev}. 

We turn to showing \eqref{eq:exterior vanish}. Let $N=R^{-\frac12}$ and $\theta\in(0,1)$. By \eqref{eq:negative Sobolev}, we have
\begin{equation}\label{eq:vanish l}
\norm{P_{<N/\sqrt t}(u(t)-z^*)}_{L^2}\les (N/\sqrt t)^{\theta}\norm{u(t)-z^*}_{\dot H^{-\theta}}\les N^\theta.
\end{equation}
We decompose $u=\mathring u+\tilde u$ as in Lemma~\ref{lem:regularity}. By \eqref{eq:X embed Hs}, we have
\begin{equation}\label{eq:vanish h mathring}
\norm{P_{\ge N/\sqrt t}\mathring u(t)}_{L^2(|x|\ge R\sqrt t)}\les(N/\sqrt t)^{-\theta}\norm{\psi^{\ge R\sqrt t}\mathring u(t)}_{\dot H^\theta}+\jp{NR}^{-10}\les (NR)^{-\theta}.
\end{equation}
Since $\tilde u\in S((0,1])$, $\norm{P_{\ge K}\tilde u}_{L^\infty L^2}$ vanishes as $K\to\infty$. Thus, we have
\begin{equation}\label{eq:vanish h tilde}
\norm{P_{\ge N/\sqrt t}(\tilde u(t)-z^*)}_{L^2}\to0\quad\text{ as }t\to0.
\end{equation}

Summing \eqref{eq:vanish l}, \eqref{eq:vanish h mathring}, and \eqref{eq:vanish h tilde}, we conclude \eqref{eq:exterior vanish}.
\end{proof}

\begin{proof}[Proof of Theorem~\ref{thm:loglog} \eqref{enu:parabolic convergence}]
Let $z^*$ be as in Lemma~\ref{lem:negative Sobolev}. We argue by contradiction; assume there exist $\ep>0$, $r>0$, and a time sequence $\{t_n\}$ such that $t_n\to 0$ and 
\begin{equation}\label{eq:proof of thm:loglog(1) - > epsilon}
\norm{u(t_n)-z^*}_{L^2(|x|\ge r\sqrt{t_n})}\ge\ep.
\end{equation}

By \eqref{eq:exterior vanish}, there exists $R_0>0$ such that
\begin{equation}\label{eq:proof of thm:loglog(1) - < epsilon/2}
\norm{u(t_n)-z^*}_{L^2(|x|\ge R_0\sqrt{t_n})}\le\tfrac\ep3,\quad n\gg1.
\end{equation}
Denote $s_n=\log t_n$.
Since $z^*\in L^2(\R^d)$, by \eqref{eq:proof of thm:loglog(1) - > epsilon} and \eqref{eq:proof of thm:loglog(1) - < epsilon/2}, we have
\begin{equation}\label{eq:proof of thm:loglog(1) - trig}
\norm{\ul u(s_n)}_{L^2( r\le|y|\le R_0)}=\norm{u(t_n)}_{L^2( r\sqrt{t_n}\le|x|\le R_0\sqrt{t_n})}\ge\tfrac\ep3,\quad n\gg1.
\end{equation}

Let $\delta=\delta(\ep,r,R_0)>0$ be a number to be fixed later. 
We use Proposition~\ref{prop:main} with $\rho=1$ and any $\kappa\le(r/4)^2$. Since $\tilde u,\tilde u^\pc$ have finite $S$ norm and $s_n\to-\infty$ as $n\to\infty$, for $n\gg1$, the condition \eqref{eq:eta def} of Proposition~\ref{prop:main} holds for $s_*=s_n$ and hence
\begin{equation}\label{eq:proof of thm:loglog(1) - main prop conseq}
\norm{\ul u}_{L^2L^2([s_n,s_n+1]\times\{r/4\le|y|\le 4R_0\})}
\le\kappa^{-\frac12}\norm{y\ul u}_{L^2L^2([s_n,s_n+1]\times\{\sqrt\kappa\le|y|\le 4R_0\})}\le\delta.\hspace{-2mm}
\end{equation}

We use the local mass
\[
M_\loc(s)=V_{\varphi_\loc}(s),\quad \varphi_\loc=\psi^{ r\le\cdot\le R_0},
\]
where we follow the notation in \eqref{eq:V def}.
By \eqref{eq:V'=W}, we have
\begin{equation}\label{eq:local mass bound}
|\d_sM_\loc(s)|\les_{r,R_0}\norm{\ul u(s)}_{H^{\frac12}(r/4\le|y|\le 4R_0)}^2,\quad s\in[s_n,s_n+1].
\end{equation}
By \eqref{eq:proof of thm:loglog(1) - trig}, we have $M_\loc(s_n)\ge(\ep/3)^2$. Thus, integrating \eqref{eq:local mass bound} yields
\begin{equation}\label{eq:proof of thm:loglog(1) - claim}
M_\loc(s)\ge (\tfrac\ep3)^2- O_{r,R_0}(1)\cdot\norm{\ul u}_{L^2H^{\frac12}([s_n,s]\times\{r/4\le|y|\le4R_0\})}^2, \quad s\in[s_n,s_n+1].
\end{equation}

Decomposing $u=\mathring u+\tilde u$ as in Lemma~\ref{lem:regularity}, by \eqref{eq:X embed Hs} and \eqref{eq:tilde u embed Lp cap H1/2}, we have
\[
\norm{\ul u}_{L^2 H^{\frac12}([s_n,s]\times\{r/4\le|y|\le 4R_0\})}^2\les_{r,R_0}(s-s_n)+o_n(1).
\]
Hence, choosing small $\delta>0$, \eqref{eq:proof of thm:loglog(1) - claim} contradicts \eqref{eq:proof of thm:loglog(1) - main prop conseq}, completing the proof.
\end{proof}
We now prove Theorem~\ref{thm:loglog} \eqref{enu:loglog}. In this proof, we employ the decompositions $\mathring u_m+\tilde u_m$ and Proposition~\ref{prop:main sec6} to detect the sharp coefficient $\sqrt{2\pi}$.
\begin{proof}[Proof of Theorem~\ref{thm:loglog} \eqref{enu:loglog}]
We use the convention $s=\log t$ and denote
\begin{equation}\label{eq:ul la def}
\ul\la(s)=t^{-\frac12}\la_\delta(t)=\inf\{r>0:\norm{(\ul u-\ul{z^*})(s)}_{L^2(|y|\ge r)}\le\delta\}.
\end{equation}

Then, \eqref{eq:loglog} reduces to showing for $B_0>2\pi$, $\delta\ll_{B_0}1$, and $|s_0|\gg_{u,\delta}1$
\begin{equation}\label{eq:loglog in backward}
\int_{s_0}^0\ul\la^2(s)ds\le B_0|\log\delta|^2\cdot(\log|s_0|)^{-1}|s_0|.
\end{equation}

By \eqref{eq:z* def} and \eqref{eq:ul la def}, we have
\begin{equation}\label{eq:delta/3 (1)}
\ul\la(s)\le 1\text{ and }\norm{(\ul{u}-\ul{z^*})(s)}_{L^2(\ul\la(s)\le|y|\le1)}\gtrsim\delta,\quad |s|\gg1.
\end{equation}

Since $z^*\in L^2(\R^d)$, $\ul{z^*}(s)$ vanishes in $L^2(|y|\le1)$ as $s\to-\infty$. Thus, \eqref{eq:delta/3 (1)} yields
\begin{equation}\label{eq:delta/3 (3)}
\delta^{-1}\norm{y\ul u(s)}_{L^2(\ul\la(s)\le|y|\le 1)}\ge\ul{\la}(s)\cdot\delta^{-1}\norm{\ul u(s)}_{L^2(\ul\la(s)\le|y|\le1)}\gtrsim\ul\la(s),\quad |s|\gg1.
\end{equation}
We fix $B>2\pi$, $q\in(2,p]$, and $C_1>\sqrt\pi$ such that $qC_1^2<B<B_0$. Let $\rho=\rho(B)>0$ be a number to be fixed later. We apply Proposition~\ref{prop:main sec6} with
\[\kappa=B|\log\delta|^2\cdot(\log|s_0|)^{-1}.\]
Let $\E$ be the set of $s_*\in\Z$ at which Proposition~\ref{prop:main sec6} is not applicable, namely
\[
\E=\Z\cap([-\kappa^{-10},0]\cup\{s_*\in[s_0,-\kappa^{-10}):\eqref{eq:eta def sec6}\text{ fails}\}).
\]
By the definition of the $S$ norm, \eqref{eq:S embed}, and \eqref{eq:local smoothing of tilde u sec6}, we have
\begin{equation}\label{eq:E size}
\#\E\le \kappa^{-O(1)}\cdot e^{qC^2\kappa^{-1}}\le\kappa^{-O(1)}|s_0|^{qC_1^2B^{-1}}=o_{s_0}(\kappa|s_0|),\quad C=C_1|\log\delta|.
\end{equation}
For $s_*\notin\E$, Proposition~\ref{prop:main sec6} yields \eqref{eq:main sec6}. Hence, for $\mc B\subset[s_0,0]\cap\Z$,
\[
\sum_{s_*\in\mc B}\int_{s_*-1}^{s_*}\norm{y\ul u(s)}_{L^2(\sqrt\kappa\le |y|\le 1)}^2ds\les \sum_{s_*\in\mc B\setminus\E}\int_{s_*-1}^{s_*}\norm{y\ul u(s)}_{L^2(\sqrt\kappa\le|y|\le 1)}^2ds+\#\E
\]
is bounded by $O(\rho^2\delta^2\kappa\#\mc B)+o_{s_0}(\kappa|s_0|)$.
Thus, by \eqref{eq:delta/3 (3)}, we obtain
\begin{align}\label{eq:level set main for B}
&\sum_{s_*\in\mc B}\int_{s_*-1}^{s_*}\ul\la^2(s)ds
\le\kappa\#\mc B+\sum_{s_*\in\mc B}\int_{s_*-1}^{s_*}\ul\la^2(s)\cdot \one_{\sqrt\kappa\le\ul\la(s)}ds
\\\le&\kappa\#\mc B+O(\delta^{-2})\cdot(\rho^2\delta^2\kappa\#\mc B+o_{s_0}(\kappa|s_0|))\nonumber
\\
\le& (1+O(\rho^2))B|\log\delta|^2\cdot(\log|s_0|)^{-1}\#\mc B+o((\log|s_0|)^{-1}|s_0|).\nonumber
\end{align}
Setting $\mc B=[s_0,0]\cap\Z$ and $\rho\ll1$, we obtain \eqref{eq:loglog in backward}. This completes the proof.
\end{proof}
\begin{proof}[Proof of Theorem~\ref{thm:saturate loglog} \eqref{enu:Thm profile}]
It suffices to show that if $0<T-t_*\ll1$ and \eqref{eq:saturate loglog} holds at $t_*$, then \eqref{eq:precise exponential} also holds at $t_*$. We mostly follow the proof of Theorem~\ref{thm:loglog} \eqref{enu:loglog} for $s_0=\log t_*$, setting $\rho=1$, $R_0=1$, $q=p$, and $\delta>0$ with which Lemma~\ref{lem:superexp} holds.

We first show the lower bound. Let $B=B(\ep,\delta)>0$ be a number to be fixed later. We use Lemma~\ref{lem:superexp} instead of Proposition~\ref{prop:main}. Fix $C_1>0$ such that $pC_1^2<B$. Let
\begin{align*}
\kappa&\coloneqq B|\log\delta|^2\cdot(\log|s_0|)^{-1},
\\
\E&\coloneqq\Z\cap([-\kappa^{-10},0)\cup[s_0,s_0+\kappa^{-10})\cup\{s_*\in[s_0,-\kappa^{-10}):\eqref{eq:eta def}\text{ fails}\}),
\\
\mc A&\coloneqq\Z\cap\{s_*\in[s_0,-\kappa^{-10}):\eqref{eq:main}\text{ holds}\},
\\
\mc B&\coloneqq\Z\cap[s_0,-\kappa^{-10})\setminus(\E\cup\mc A).
\end{align*}
Similarly to \eqref{eq:E size}, $pC_1^2<B$ yields $\#\E=o_{s_0}(\kappa|s_0|)$.
By \eqref{eq:delta/3 (3)} and \eqref{eq:main}, we have
\begin{equation}\label{eq:A contribution}
\int_{s_*-1}^{s_*}\ul\la^2(s)ds\les_\delta\norm{y\ul u}_{L^2L^2([s_*-1,s_*]\times\{|y|\le1\})}^2\les \kappa\les_\delta B(\log|s_0|)^{-1},\,\,\; s_*\in\mc A.
\end{equation}
By \eqref{eq:level set main for B}, \eqref{eq:A contribution}, and $\#\E=o(\kappa|s_0|)=o((\log|s_0|)^{-1}|s_0|)$, we obtain
\begin{equation}\label{eq:int upper}
I\coloneqq\int_{s_0}^0\ul\la^2(s)ds\les_\delta B(\log|s_0|)^{-1}|s_0|+(\log|s_0|)^{-1}\#\mc B.
\end{equation}
On the other hand, \eqref{eq:saturate loglog} yields $I\gtrsim\ep(\log|s_0|)^{-1}|s_0|$. Thus, choosing sufficiently small $B=B(\ep,\delta)\ll1$, we obtain
\begin{equation}\label{eq:B big}
\#\mc B\gtrsim_{\delta,\ep}|s_0|.
\end{equation}
Given any $\theta\gg_{\delta,C_1}1$,
Lemma~\ref{lem:superexp} yields \eqref{eq:assump for superexp} for $s_*\in\mc B$. Thus, we have
\[
\int_{s_0}^0|\ul u(s,\theta\sqrt\kappa)|^2ds\gtrsim\sum_{s_*\in\mc B}\int_{s_*-\kappa^{-10}}^{s_*+\kappa^{-10}}|\ul u(s,\theta\sqrt\kappa)|^2\cdot e^{-\frac14|s-s_*|}ds\gtrsim_{\delta,\ep} \frac{e^{-c_*\theta}}{(\theta\sqrt\kappa)^d}\cdot|s_0|
\]
with $c_*=c_*(\delta,C_1)>0$. Substituting $r=\theta\sqrt{B|\log\delta|^2}$, renormalizing yields the lower bound of \eqref{eq:precise exponential}.

We now show the upper bound. We fix $C_1=C_1(d)$ as the number given in Proposition~\ref{prop:main} and choose $B>pC_1^2$. Let  $\delta_*=e^{-r/\sqrt B}$ and
\begin{align*}
\kappa_*&\coloneqq B|\log\delta_*|^2\cdot(\log|s_0|)^{-1}=r^2(\log|s_0|)^{-1},
\\
\E_*&\coloneqq\Z\cap([-\kappa_*^{-10},0]\cup\{s_*\in[s_0,-\kappa_*^{-10}):\eqref{eq:eta def}\text{ fails}\}).
\end{align*}
As earlier, we have $\#\E_*=o_{s_0}(\kappa_*|s_0|)$. For $s_*\notin\E_*$, Proposition~\ref{prop:main} yields
\[
\norm{\ul u}_{L^2L^2([s_*-1,s_*]\times\{\sqrt{\kappa_*}\le|y|\le1\})}\le\delta_*.
\]
Taking an $\ell^2$-summation over all integers $s_*\in[s_0,0]$, we obtain
\begin{equation}\label{eq:saturate upper prep}
\norm{\ul u}_{L^2L^2([s_0,0]\times\{\sqrt{\kappa_*}\le|y|\le1\})}\les\delta_*\sqrt{|s_0|}\les e^{-r/\sqrt B}\sqrt{|s_0|}.
\end{equation}
Since \eqref{eq:saturate upper prep} itself is not pointwise in radius, we estimate higher regularity of $\ul u$ for interpolation.
We decompose $u=\mathring u+\tilde u$ as in Lemma~\ref{lem:regularity} and denote the $L^2(\R^d)$-critical scaling by $f^\sc(y)=\kappa_*^{\frac d4}f(\sqrt{\kappa_*} y)$.
Since $\ep_d>\frac{d-2}{2d}$, there exists $\tilde\ep_d>0$ such that, by the definition of the $S$ norm, the radial Sobolev inequality, and \eqref{eq:X embed Hs},
\begin{align*}
\norm{\ul{\tilde u}}_{L^2W^{\tilde\ep_d,\infty}([s_0,0]\times\{|y|\ge\sqrt{\kappa_*}\})}&\les\kappa_*^{-O(1)}\norm{\ul{\tilde u}}_{L^2 B^{\ep_d}_{2_1,2}([s_0,0]\times\{|y|\ge\sqrt{\kappa_*}\})}\les\kappa_*^{-O(1)},
\\
\norm{\ul{\mathring u}^\sc}_{L^\infty W^{\tilde\ep_d,\infty}([s_0,0]\times\{|y|\ge1\})}&\les\norm{\ul{\mathring u}^\sc}_{L^\infty H^1([s_0,0]\times\{|y|\ge1\})}\les1.
\end{align*}
Taking a summation, we obtain
\begin{equation}\label{eq:high sobolev for upper}
\norm{\ul u^\sc}_{L^2W^{\tilde\ep_d,\infty}([s_0,0]\times\{|y|\ge1\})}\les\sqrt{|s_0|}.
\end{equation}
Interpolating between \eqref{eq:saturate upper prep} and \eqref{eq:high sobolev for upper}, we obtain for small $c=c(d)>0$ that
\[
\int_{s_0}^0|\ul u(s,\sqrt{\kappa_*})|^2ds\les\kappa_*^{-\frac d2}\norm{\ul u^\sc}_{L^2L^\infty([s_0,0]\times\{1\le|y|\le2\})}^2\les r^{-d}e^{-cr}(\log|s_0|)^{\frac d2}|s_0|,
\]
renormalizing which yields the upper bound of \eqref{eq:precise exponential}. This completes the proof.
\end{proof}
\begin{proof}[Proof of Theorem~\ref{thm:saturate loglog} \eqref{enu:Thm z*}]
We argue by contradiction. Assume $u_0,z^*\in H^\al(\R^d)$ for some $\al>0$, where we may take $\al\ll1$.
We follow the proof of Lemma~\ref{lem:regularity} with $I=(\tau_-,\tau_+]=(0,1]$, replacing the term $u(\tau_-)$ in the truncated Duhamel formula \eqref{eq:decomp PN(>u)} by $z^*$. This is justified by the convergence to $z^*$ \eqref{eq:z* def} and the exterior regularity of $u$ \eqref{eq:u_a bdd}. Repeating the proof with an additional factor $N^\al$ in the estimates for $\tilde u$, we obtain $\dot H^\al$-critical regularity bounds of $\tilde u\in S((0,1])$:
\begin{align}\label{eq:scale break}
&\norm{\tilde u}_{S^\al((0,1])}\coloneqq\norm{N^\al\norm{\tilde u_N}_{L^\infty L^2}}_{\ell^2_N}
+\norm{N^\al\norm{\jp{Nx}^{\ep_d}\tilde u_N}_{L^2L^{2_1}}}_{\ell^2_N}<\infty,
\\\label{eq:Hs local smoothing}
&\norm{e^{\frac\al2|s|}\jp y^{-1}P_{\ge1}\ul{\tilde u}(s)}_{L^2H^{\frac12+\al}((-\infty,0]\times\R^d)}<\infty.
\end{align}

We then show Proposition~\ref{prop:main} with the Strichartz pair $(p,p)$, assuming only the bounds on $\tilde u$ in \eqref{eq:eta def}.
We follow the previous proof, replacing Sobolev norm and decay exponents $\frac14+\frac k2$, $k\in\Z$ by $\s_d+\frac k2=\frac{\ep_d}{100d}+\frac k2$.
As noted above \eqref{eq:mathring u decay}, the proof works without introducing $\tilde u^\pc$ once we show the decay bound in \eqref{eq:mathring u decay}, namely
\begin{equation}\label{eq:mathring u decay sec6}
\norm{|y|^{\frac12+\s_d}\ul{\mathring u}}_{L^2L^2([s_*-\kappa^{-10},s_*+\kappa^{-10}]\times\{\kappa/4\le|y|\le4R\})}\les\kappa^{-O(1)}.
\end{equation}
We show this decay. For $0\le\theta<1$ and $s\le-\kappa^{-10}$, Hardy's inequality yields
\[
\norm{\ul{z^*}(s)}_{\dot H^{-\theta}(|y|\le 4R)}\les R\norm{\ul{z^*}(s)}_{L^2(|y|\le 4R)}\les R\norm{z^*}_{L^2(|x|\le 4R\sqrt{e^s})}\les R\cdot(R\sqrt{e^s})^\al\les1.
\]
Summing with \eqref{eq:negative Sobolev} yields the same bound for $\ul u(s)$. Then, \eqref{eq:eta def} for $\tilde u$ yields
\[
\norm{\ul{\mathring u}}_{L^2\dot H^{-\theta}(|y|\le 4R)}\les \kappa^{-10}(\norm{\ul u}_{L^\infty\dot H^{-\theta}(|y|\le 4R)}+R^{\frac12-\frac1p+\theta}\norm{\ul{\tilde u}}_{L^pL^p})\les\kappa^{-10},\,\,\;0\le\theta<1.
\]
Interpolating this estimate with \eqref{eq:X embed Hs}, we obtain for each dyadic number $r\le R$
\[
\norm{\ul{\mathring u}}_{L^2L^2(r\le|y|\le 2r)}\les\kappa^{-O(1)} r^{-\theta},\quad\theta<\tfrac25(\tfrac32-3\s_d).
\]
Choosing $\theta=\frac12+2\s_d$, summing over $r\le R$ yields \eqref{eq:mathring u decay sec6} as desired.

We now recall the proof of Theorem~\ref{thm:saturate loglog} \eqref{enu:Thm profile}, where we deduced \eqref{eq:B big}. 
Since $\tilde u\in S^\al\subset L^2\dot B^\al_{2_1,2}$ and $\tilde u\in L^\infty L^2$, interpolating yields $\tilde u\in L^qL^p$ for some $q>p$. Thus, $e^{-\nu s}\ul{\tilde u}\in L^pL^p$ holds for small $\nu>0$. Combining with \eqref{eq:scale break}--\eqref{eq:Hs local smoothing}, we obtain that all norms of $\tilde u$ in \eqref{eq:eta def} are $O(e^{-\nu(|s_*|-\kappa^{-10})})$. Thus, \eqref{eq:eta def} for $\tilde u$ holds for all $s_*\le-2\kappa^{-10}$. Consequently, \eqref{eq:main} holds for all $s_*\le-2\kappa^{-10}$. Then $\#\mc B\le2\kappa^{-10}$ follows, contradicting \eqref{eq:B big}. This finishes the proof.
\end{proof}
In the next proof, we use that if a radial solution $u$ to \eqref{eq:NLS} of mass $M(Q)$ scatters neither forward nor backward, then $u$ is a scaling of $e^{it}Q$. This result is known in every dimension; see Li--Zhang \cite{Li2010almost_periodic} for $d\ge 4$ and Dodson \cite{Dodson21} for $2\le d\le 15$.
\begin{proof}[Proof of Corollary \ref{cor:loglog la Bourgain}]
Let $C_1>\sqrt{2\pi}$. Let $a>0$ and $\delta=e^{-a/C_1}$. We denote $\la=\la_\delta$ as in Theorem~\ref{thm:loglog}. Assume $a\gg1$ so that in Theorem~\ref{thm:loglog},
\[
C=\sqrt{2\pi+o_\delta(1)}|\log\delta|<C_1|\log\delta|=a.
\]
Then, since the right-hand side of \eqref{eq:loglog} diverges as $t_*\to T$, there exists a sequence of times $t_n\to T$ such that
\begin{equation}\label{eq:la bound by R loglog}
\la(t_n)\le a\sqrt{\frac{T-t_n}{\log|\log(T-t_n)|}}.
\end{equation}
Denote $z_n=u(t_n)-z^*$.
By the definition of $\la$, we have
\begin{equation}\label{eq:u-z* la}
M(\one_{|x|\ge\la(t_n)}z_n)\le\delta^2=e^{-2aC_1^{-1}}.
\end{equation}
By Theorem~\ref{thm:loglog} \eqref{enu:parabolic convergence}, we have
\begin{equation}\label{eq:u-z* la =M}
\lim_{n\to\infty}M(z_n)=M(u)-M(z^*).
\end{equation}
By \eqref{eq:z* def}, \eqref{eq:KillipVisan conseq}, and Fatou's lemma, we have $M(z^*)\le M(u)-M(Q)$. 
We separate two cases: (i) $M(z^*)<M(u)-M(Q)$, (ii) $M(z^*)=M(u)-M(Q)$.

(i) Assume $M(z^*)<M(u)-M(Q)$.
For sufficiently large $a$, \eqref{eq:u-z* la}--\eqref{eq:u-z* la =M} yield
\[
\limsup_{n\to\infty}M(\one_{|x|<\la(t_n)}u(t_n))=\limsup_{n\to\infty}M(\one_{|x|<\la(t_n)}z_n)\ge M(u)-M(z^*)-\delta^2>M(Q).
\]
Then, we obtain \eqref{eq:loglog la Bourgain} for $r\ge a$ due to \eqref{eq:la bound by R loglog}.

(ii) We turn to the case $M(z^*)=M(u)-M(Q)$. For this case, we first fix any $r>0$. Passing to a subsequence, we perform a profile decomposition for $\{z_n\}$ as in Proposition~\ref{prop:chenjie linear}. By \eqref{eq:profile decomp-nonlinear profile decomp}, there exists an index $j$ such that $v^j$ does not scatter forward, which forces $M(v^j)\ge M(Q)$. By \eqref{eq:u-z* la =M}, we have $M(z_n)\to M(Q)$, hence $v^j$ is the only nonlinear profile, $M(v^j)=M(Q)$, and $M(z^{>j}_n)\to0$. If $v^j$ scatters backward, the backward nonlinear evolution of $z_n$ is uniformly bounded in $L^pL^p$. Then, choosing $\tau<T$ such that $\norm{e^{it\De}z^*}_{L^pL^p([\tau-T,0]\times\R^d)}\ll1$, standard perturbation theory yields $\norm{u}_{L^pL^p([\tau,t_n]\times\R^d)}=O(1)$, which contradicts $t_n\to T$. Thus, $v^j$ scatters neither forward nor backward, hence is a scaling of $e^{it}Q$.

Since $M(z_n^{>j})\to0$ as $n\to\infty$, we have
\begin{equation}\label{eq:to rescaled Q}
\norm{z_n-\rho_n^{-\frac d2}e^{i\omega_n}Q(\rho_n^{-1}x)}_{L^2}\to0\quad\text{as}\quad n\to\infty
\end{equation}
for some sequences $\{\rho_n\}\subset(0,\infty)$ and $\{\omega_n\}\subset[0,2\pi]$. Thus, we have
\begin{align*}
&\limsup_{t\to T}M(\one_{|x|\le r\sqrt{\frac{T-t}{\log|\log(T-t)|}}}u(t))\ge\limsup_{n\to\infty}M(\one_{|x|\le r\sqrt{\frac{T-t_n}{\log|\log(T-t_n)|}}}z_n)
\\
\ge&\limsup_{n\to\infty}M(\one_{|x|\le ra^{-1}\la(t_n)}z_n)\ge \limsup_{n\to\infty}M(\one_{|x|\le ra^{-1}\la(t_n)\rho_n^{-1}}Q),
\end{align*}
where we use \eqref{eq:to rescaled Q} for the last inequality.

Since $Q(x)\gtrsim e^{-|x|-o(|x|)}$, by \eqref{eq:u-z* la} and \eqref{eq:to rescaled Q}, we have
\[
\la(t_n)\rho_n^{-1}\ge a\cdot (1-o_a(1))C_1^{-1},\quad n\gg_a1.
\] 
Thus, taking $a\to\infty$ and $C_1\to\sqrt{2\pi}$ yields \eqref{eq:loglog la Bourgain}. This completes the proof.
\end{proof}
\bibliographystyle{abbrv}
\bibliography{citationforloglog}
\end{document}